\documentclass[11pt]{amsart}
\usepackage[margin=1in]{geometry}
\usepackage{comment}
\usepackage{graphicx}
\usepackage{hyperref}
\usepackage{amsmath}
\usepackage{amssymb}
\usepackage{amsthm}
\usepackage{mathtools}
\usepackage{mathrsfs}
\usepackage{braket}
\usepackage{float}
\usepackage{graphicx}
\usepackage{listings}
\usepackage{adjustbox}
\usepackage{framed} 
\usepackage{ltablex}
\usepackage{xcolor}
\usepackage{pgfplots}
\usepackage{tikz}
\usepackage{tikz-3dplot}
\usetikzlibrary{calc}
\usetikzlibrary{angles}
\usetikzlibrary{shapes.geometric}
\usetikzlibrary{cd}
\tdplotsetmaincoords{60}{115}
\pgfplotsset{compat=newest}
\usetikzlibrary{cd}
\usepackage[backend=biber]{biblatex}
\usepackage{bbm}
\usepackage{amsthm}
\usepackage{enumitem}

\newtheorem*{remark}{Remark}
\newtheorem*{theorem*}{Theorem}
\numberwithin{equation}{section}
\newtheorem{theorem}[equation]{Theorem}

\newtheorem{corollary}[equation]{Corollary}

\newtheorem{lemma}[equation]{Lemma}

\newtheorem{proposition}[equation]{Proposition}

\newtheorem{definition}[equation]{Definition}
\makeatletter

\renewcommand*\env@matrix[1][*\c@MaxMatrixCols c]{
  \hskip -\arraycolsep
  \let\@ifnextchar\new@ifnextchar
  \array{#1}}
\tagsleft@false\let\veqno\@@eqno
\makeatother

\renewcommand{\qed}{$\hfill\square$}

\newcommand{\clif}{\ensuremath{\mathcal{C}\ell_{0,n}}}
\newcommand{\clifn}[1]{\mathcal{C}\ell_{{#1}}}
\newcommand{\para}{\ensuremath{\$\mathbb{R}^{n+1}}}
\newcommand{\cpara}{\ensuremath{\$\widehat{\mathbb{R}}^{n+1}}}
\newcommand{\lip}{\ensuremath{\$\Gamma_{n+1}}}

\def\ol{\ensuremath\overline}
\newcommand{\defn}{\ensuremath{\scalebox{1.5}{$\lrcorner$}}}
\DeclareFontFamily{U}{mathx}{}
\DeclareFontShape{U}{mathx}{m}{n}{<-> mathx10}{}
\DeclareSymbolFont{mathx}{U}{mathx}{m}{n}
\DeclareMathAccent{\widehat}{0}{mathx}{"70}
\DeclareMathAccent{\widecheck}{0}{mathx}{"71}

\newcommand{\pU}{\mathbf{U}}
\newcommand{\pV}{\mathbf{V}}
\newcommand{\pW}{\mathbf{W}}
\newcommand{\pX}{\mathbf{X}}
\newcommand{\N}{\mathbb{N}}    % Natural numbers
\newcommand{\Z}{\mathbb{Z}}    % Integers
\newcommand{\Q}{\mathbb{Q}}    % Rational numbers
\newcommand{\R}{\mathbb{R}}    % Real numbers
\newcommand{\C}{\mathbb{C}}    % Complex numbers
\newcommand{\hyp}{\mathbb{H}}

\newcommand{\U}{\mathbb{U}}
\newcommand{\B}{\mathbb{B}}\newcommand{\pY}{\mathbf{Y}}

\begin{document}
\title{The Lipschitz Spinor-Higher Horosphere Correspondence}
\author{Orion Zymaris}

\begin{abstract}
In a paper of Mathews \cite{Mathews_Spinors_horospheres}, an isomorphism is constructed between two-component complex spinors and horospheres in $\mathbb{H}^3$ carrying `spin decorations'. A recent arXiv preprint of Mathews and Varsha \cite{mathews2024quaternions} extends this result to the case of `quaternionic spinors' and spin decorated horospheres in $\mathbb{H}^4$.

The following work generalises these results to an equivariant correspondence between two-component `Lipschitz spinors' with entries drawn from the Lipschitz group of a Clifford algebra, null multiflags in generalised Minkowski space, and higher-dimensional horospheres that carry an extension of the Mathews spin decoration. This correspondence allows spinors to be applied to horospheres in any dimension of hyperbolic space. 

A generalisation of Mathews' complex lambda lengths to arbitrary dimension is also given, along with some relations on these lambda lengths. In particular, the noncommutative hyperbolic Ptolemy relation of \cite{mathews2024quaternions} is extended to arbitrary dimension.

This paper is adapted from the author's thesis, available at \url{https://doi.org/10.26180/31361224}.
\end{abstract}

\maketitle
\tableofcontents

\setcounter{section}{-1}
\section{Introduction}
\label{sec:Introduction}
\subsection{Background}
Though a thorough treatment of the groups underlying spinors and spin structures was carried out in 1913 by \'{E}lie Cartan \cite{Cartan1913} in his work on projective groups (which would now be called projective representations of groups), the significance of the spinor was best understood by the early pioneers of quantum mechanics.\footnote{The term `spinor' first appears in a paper of the physicist van der Waerden \cite{Waerden1928}, who attributes it to Paul Ehrenfest.} Groundbreaking work of Pauli \cite{Pauli1927} and Dirac \cite{Dirac1928} in the late 1920's made use of spinors and matrix representations of certain spin groups to describe the quantum phenomenon of `spin' and introduce relativistic effects into the bourgeoning field of quantum mechanics. In the realm of pure mathematics, spinors and spinor bundles would become an important tool in geometry and topology, helping to answer questions about the existence or non-existence of certain structures on manifolds or the index of elliptic operators (\cite{LawsonMichelsohn} is a standard reference). And, across both physics and mathematics, maybe the most important variety of spinor is the two-component complex spinor $(\xi,\eta) \in \C^2$.

Since its introduction the two-component complex spinor has been thoroughly explored, its properties mapped and collated. These spinors carry representations of Spin$(3)$ and Spin$(1,3)$ as the groups $SU(2)$ and $SL(2,\C)$ respectively, two of the low-dimensional exceptional isomorphisms between spin groups and certain matrix groups. There are many other representations of the spin groups (including other representations of Spin$(3)$ and Spin$(1,3)$), but these $2 \times 2$ matrix groups and associated two-component complex spinors continue to be among the most fruitful. This raises a question; can this two-component form be extended to represent other spin groups?

It can indeed. To do this, we make a shift; by considering two-component complex spinors as projective coordinates on the Riemann sphere $\widehat{\C}$, the action of $SL(2,\C)$ becomes an action by M\"{o}bius transformations on $\widehat{\C}$.\footnote{This perspective is well-known; see, for example, \cite{SpinorsAndSpacetime}.} Results about generalised M\"obius transformations in higher dimensions are already known; the machinery for generalised M\"{o}bius transformations was introduced by Vahlen \cite{Vahlen1902}, forgotten, rediscovered by Fueter \cite{Fueter1926},\cite{Fueter1927}, forgotten again, and rediscovered by Maass \cite{Maass1949}. Here we make use of the more refined formulation found in work of Ahlfors (\cite{Ahlfors1985},\cite{Ahlfors1986}), and built upon by others such as Cao and Waterman \cite{Cao1998}, Waterman \cite{WATERMAN199387}, Lounesto and Latvamaa \cite{LounestoLatvamaaPaper}, and Wada \cite{Wada01091990}.

In this formulation, the vector space $\C$ and its compactification $\widehat{\C}$ are generalised to the paravector spaces $\para$ (with $1$ real component and $n$ complex components) and their compactifications $\cpara$. Two-component objects (the \textit{Lipschitz spinors}) that provide a form of `projective coordinates' for $\cpara$ are defined; these Lipschitz spinors draw their entries from the \textit{Lipschitz groups} $\lip$, which in turn are drawn from the Clifford algebras $\clif$ that govern the isometries of $\para$. 

Generalisations $SL(2,\lip)$ of the special linear groups are also defined, again with entries drawn from the Lipschitz groups $\lip$. These matrices act on Lipschitz spinors by matrix multiplication, or they can equivalently be considered as the components of M\"{o}bius transformations over $\cpara$. This results in generalised two-component spinors that carry representations of the higher spin groups Spin$(1,n+2)$ as matrix subgroups of $\mathcal{M}(2,\clif)$.

A key property of spinors %, and the main source of their mystery,
is the subtle nature of their geometric interpretation; they carry representations not just of the rotation groups $SO(p,q)$ but of their double covers $Spin(p,q)$. In many cases\footnote{In the case of indefinite signature the covering maps are more complicated.} this leads to an action that is locally a rotation, but is globally well defined up to `rotation' by $4\pi$. The standard geometric construction, a way to present the spinor as coordinates of some geometric object, is the null flag due to Penrose and Rindler \cite{SpinorsAndSpacetime}; in this picture, a flag structure is constructed on the light-cone in $\R^{1,3}$. They also offer an interpretation of a spinor as a tangent vector to the sphere at infinity, which is modelled as the space of rays of the light-cone.

A recent paper of Mathews \cite{Mathews_Spinors_horospheres} composes this null flag picture with some work of Penner\footnote{Penner's construction is in $\R^{1,2}$, but it generalises neatly to higher dimensions.} \cite{PennerPunctured}, \cite{PennerBook}. Penner provides a map from points of the light-cone in $\R^{1,3}$ to horospheres in the hyperboloid model of hyperbolic space that sits within the light-cone; combined with Penrose and Rindler's picture, the result is an equivariant and bijective correspondence between complex two-component spinors and \textit{spin decorated} horospheres in 3-dimensional hyperbolic space $\hyp^3$. A \textit{spin decoration} is an extra structure defined in the tangent space of a horosphere, essentially a choice of orthonormal basis for the tangent space which is then lifted to the spin double cover of the orthonormal frame bundle. 

The space of complex spinors has a natural symplectic form, and Mathews shows this form has a geometric interpretation as a generalisation of Penner's \textit{lambda length} between horocycles. A hyperbolic form of the Ptolemy relation between lambda lengths in $\hyp^3$ is also given. In a later paper \cite{mathews2024quaternions}, Mathews and Varsha extend this correspondence to one between spin decorated horospheres in $\hyp^4$ and two-component spinors with quaternionic entries.
\subsection{Key Results}
This work extends the Mathews spinor-horosphere correspondence to the case of arbitrary dimensional hyperbolic space. To be precise, an equivariant correspondence between Lipschitz spinors $(\xi,\eta) \in S\lip$ and decorated horospheres\footnote{Extension to include spin is discussed in Section \ref{sec:Fourth}.} in $\hyp^{n+2}$ is given explicitly in the upper half-space model. Generalisations to $S\lip$ of the symplectic and Hermitian forms on $\C^2$ are also defined.
\begin{theorem*}[Explicit Spinor-Horosphere Correspondence]
The spinor $\kappa=(\xi,\eta) \in S\lip$ maps under $\Phi$ to a horosphere decorated with $n$ parallel oriented line fields (labelled  $i_1,i_2,...,i_n$) and centred at $\xi\eta^{-1}$ in the upper half-space model (taken to be $\infty$ if $\eta=0$).
\begin{itemize}
 \item If $\eta \neq 0$ the horosphere has Euclidean diameter $|\eta|^{-2}$, and the north pole specification for the $i_j$ direction field is $\eta^{'}i_j\ol{\eta}$.
 \item If $\eta = 0$ then the horosphere has Euclidean height $|\xi|^2$ and the $i_j$-direction field is specified by $\xi i_j\xi^*$.
\end{itemize}
\end{theorem*}
\noindent In Section \ref{sec:Fourth}, we lift our objects to their spin covers and recover the full correspondence. As we use the generalised Minkowski space to mediate between the space of Lipschitzz spinors and hyperbolic space, we also have a certain null flag structure $\widetilde{\mathcal{MF}}^n$ appearing in the correspondence.
\begin{theorem*}[Generalised Spinor-Horosphere Correspondence]
 There is an explicit, smooth, bijective,\\\noindent $SL(2,\lip)$ equivariant correspondence between the following:
 \begin{enumerate}[label=\roman*)]
 \item Lipschitz spinors $S\lip$,
 \item Spin multiflags $\widetilde{\mathcal{MF}}^n$,
 \item The space $S$Hor$^{n+2}$ of spin decorated horospheres in $\hyp^{n+2}$.
\end{enumerate}
\end{theorem*}
\noindent The real and complex lambda lengths of Penner \cite{PennerPunctured} and Mathews \cite{Mathews_Spinors_horospheres} are extended to a generalised Clifford-algebra-valued lambda length. These generalised lambda lengths relate to a generalisation of the symplectic form on complex spinors (we name this generalised form \textit{the bracket}).
\begin{theorem*}[Algebra of Lambda Lengths]
 The lambda length between two spin decorated horospheres with spinor coordinates $\kappa_1 = (\xi_1,\eta_1)$ and $\kappa_2 = (\xi_2,\eta_2)$ is given by taking their bracket. Properly,
 \[\lambda_{12} = \{\kappa_1,\kappa_2\} = \xi_1^*\eta_2 - \eta_1^*\xi_2.\]
\end{theorem*}
\noindent These lambda lengths also satisfy the noncommutative Ptolemy relation of \cite{mathews2024quaternions}, alongside other interesting properties.
\begin{theorem*}[Noncommutative Ptolemy Relations]
 Given four spin decorated horospheres $(H_j,W_j)$, $j \in \{1,2,3,4\}$ in $\hyp^{n+2}$ with lambda lengths $\lambda_{jk}$ between $(H_j,W_j)$, $(H_k,W_k)$, their lambda lengths satisfy
\[\lambda_{31}^{-1}\lambda_{23}^*\lambda_{42}^{-1}\lambda_{14}^*+\lambda_{31}^{-1}\lambda_{43}^*\lambda_{24}^{-1}\lambda_{12}^*=1.\]
\end{theorem*}

\subsection{Acknowledgements}
This paper was adapted from the author's PhD thesis, which was supported by Australian Research Council grant DP210103136. The author also wishes to thank his supervisor Daniel Mathews, and Varsha, the two of whose work forms the basis of this paper. 

\section{Clifford Algebras, Lipschitz Spinors, and Special Linear Groups}
\label{sec:First}
In \cite{Ahlfors1985}, Ahlfors reintroduces a method due to Vahlen \cite{Vahlen1902}, Maass \cite{Maass1949} and others to construct the isometry groups of $\hyp^n$ as $2\times 2$ matrix groups with entries drawn from the Lipschitz group\footnote{Also referred to as the Clifford group; we'll exclusively use the term Lipschitz group \cite{Lipschitz1886}.} of a particular Clifford algebra (introduced in \cite{CliffordAlgebras}). These matrices act on the boundary of the upper half-space model as the components of M\"{o}bius transformations which extend to isometries of the interior in the standard way, by considering their action on the endpoints of geodesics. The actions of these groups are precisely what we need to generalise the spinor-horosphere correspondence of \cite{Mathews_Spinors_horospheres} and \cite{mathews2024quaternions}.
\subsection{Clifford Algebras}
\textit{The conventions and definitions for Clifford algebras in this work are drawn from Lawson and Michelson's `Spin Geometry' \cite{LawsonMichelsohn}, and Lounesto's `Clifford Algebras and Spinors' \cite{Lounesto_2001}. Some other helpful resources are \cite{AtiyahBottShapiroCMod} and \cite{Riesz1993}}.%Note the more abstract definition of a Clifford algebra has been cut
\\
\noindent There are several ways to define a Clifford algebra, but for our purposes it is best defined explicitly in terms of a concrete basis. \label{not:Rpq}Consider the vector spaces $\R^{p,q}$ with the associated quadratic form $Q_{p,q}=x_1^2+x_2^2+...+x_p^2-x_{p+1}^2-...-x_{p+q}^2$.
\begin{definition}[Clifford Algebra]
\label{defn:CliffordAlgebra}
Take the real vector space $\R^{p,q}$ with basis $i_1,i_2,...,i_{p+q}$ and norm $N_{p,q}$ (with the standard polarisation $N_{p,q}(\cdot,\cdot)$) such that $i_1,...,i_p$ are of norm $1$ and $i_{p+1},...,i_{p+q}$ are of norm $-1$. \label{not:Clifpq}The \textit{Clifford algebra} $\clifn{p,q}$ over $\R$ is the $\R$-algebra generated by monomials in the $i_j$ subject to the relations
\begin{align*}
i_j^2 &= 1,\ 1 \leq j \leq p,\\
i_j^2 &= -1,\ p+1 \leq j \leq p+q,\\
i_ji_k &=-i_ki_j,\ k \neq j. \tag*{\defn}
\end{align*}
\end{definition}
\noindent There is a standard basis of $\clifn{p,q}$ consisting of all $2^{p+q}$ possible products $i_{j_1}i_{j_2}...i_{j_k}$ for $k \leq p+q$ and $j_1 < j_2 ... < j_k$. \label{not:multiindex}As we might have products of many $i_j$, in general we write $i_{I}$ for a multi-index $I= \{j_1,j_2,...,j_k\}$ to indicate the product $i_{j_1}i_{j_2}...i_{j_k}$. Elements $x\in \clifn{p,q}$ can then be written explicitly in terms of the standard basis as $x = \sum_{I \subseteq \{1,2,...,n\}} x_Ii_I$, where $x_I \in \R$. Elements of $\clifn{p,q}$ will generally be notated with a lowercase Latin character.

An element $x \in \clifn{p,q}$ is \textit{reduced} if it is written in terms of the standard basis; that is, if its monomials have indices in increasing order and it does not contain monomials with repeated indices, for example $i_1i_2i_1$. When we refer to monomials in $\clifn{p,q}$ from now on, we will assume they are reduced. 

The subspaces of $\clifn{p,q}$ generated by reduced monomials $i_I$ with $|I|=d$ are the subspaces of degree $d$; the elements of the subspace of degree $d$ are then said to have degree $d$. As all monomials live in exactly one such subspace, all monomials have a well-defined degree. The degree is not additive on multiplication of elements (for example, $i_j\cdot i_j$ is of degree 0, not degree 2), but it is subadditive, and additive when taken in $\Z_2$; this gives Clifford algebras a $\Z_2$ grading structure.\footnote{see \cite{LawsonMichelsohn}, Chapter 1 for a full introduction to Clifford algebras and their grading structure.}

There are several important involutions on the Clifford algebra, generalising the complex conjugate structure of $\C$.
\begin{definition}[The Involutions] We have the following involutions on the Clifford algebra $\clifn{p,q}$:
\begin{enumerate}[label=\roman*)]
\item \label{not:GradeInv}The {\normalfont grade involution} $\cdot': \clifn{p,q} \to \clifn{p,q}$ is the algebra automorphism that extends the map $V' = -V$ on $\R^{p,q}$ to the entire Clifford algebra. Given an explicit element $x = \sum_I x_Ii_I \in \clifn{p,q}$ for some multi-indexes $I$, $x'$ is obtained from $x$ by sending every $i_j$ to $-i_j$.
\item \label{not:RevInv}Given an element $x = \sum_I x_Ii_I \in \clifn{p,q}$, the {\normalfont reversal (or reverse) involution} $\cdot^*: \clifn{p,q} \to \clifn{p,q}$ is the algebra map that reverses each monomial of $x$ (say, $i_1i_2i_3 \to i_3i_2i_1$). This is an antiautomorphism; that is, $(xy)^*=y^*x^*$. 
\item \label{not:CliffConj}Doing both (in either order) gives the third morphism, the {\normalfont Clifford conjugation} $\ol{\cdot}= \cdot^{\prime*} = \cdot^{*\prime}$, also an antiautomorphism. \null\hfill\defn
\end{enumerate}
\end{definition}
\noindent The Clifford algebra $\clifn{0,1}$ is the complex numbers $\C$, and there the reversal involution becomes trivial and $\cdot'=\ol{\cdot}$ is just complex conjugation. Clifford conjugation is the natural extension of the complex conjugate, and it plays the same role in defining the norm and Hermitian product (see Definitions \ref{defn:CliffordNorm}, \ref{defn:ParavectorInnerProduct}).

In $\C$, the reals form an invariant subspace under complex conjugation. The three involutions $\cdot',\cdot^*,\ol{\cdot}$ also have invariant subspaces that help determine the structure of the algebra.
\begin{lemma}
\label{MonomialInvolutionInvarianceLemma}
 Monomials of degree $0,2 \mod 4$ are invariant under $\cdot'$, while monomials of degree $0,1 \mod 4$ are invariant under $\cdot^*$. Monomials of degree $0,3 \mod 4$ are invariant under $\ol{\cdot}$. Monomials that are not invariant under an involution are sent to their negative by the involution instead.
\end{lemma}
\noindent\textit{Proof.} The involution $\cdot'$ sends $i_j \to -i_j$, and so a monomial $m = i_{j_1}i_{j_2}...i_{j_k}$ of degree $k$ is sent to $(-1)^ki_{j_1}i_{j_2}...i_{j_k}=(-1)^km$. If $k$ is even, the monomial is invariant. 

The reversal involution $\cdot^*$ for $m$ can also be achieved through $\frac{k(k-1)}{2}$ transpositions of neighbouring $i_j$. Each transposition introduces a minus sign, so $m^*=(-1)^{\frac{k(k-1)}{2}}m$. Invariance requires $2|k(k-1)/2$, and since $k,k-1$ are coprime either $4|k$ and $k=0 \mod 4$ or $4|(k-1)$ and $k=1 \mod 4$. 

Combining the above two involutions, for a monomial $m$ of degree $k$ we find $\ol{m}=(-1)^{\frac{k(k+1)}{2}}m$, which by a similar argument implies invariance when $k=0\mod 4$ or $k=3 \mod 4$. It is clear that these involutions negate monomials that are not otherwise invariant.\qed
\\\\
\label{not:Clifpqpm}Using the grade involution, the Clifford algebra decomposes with a $\mathbb{Z}_2$ grading as $\clifn{p,q} =\clifn{p,q}^+ \oplus \clifn{p,q}^-$, where $\clifn{p,q}^{\pm} = \{s \in \clifn{p,q}\ |\ s'= \pm s\}$. The even component $\mathcal{C}\ell^+_{p,q}$ is a subalgebra.

Though Clifford algebra multiplication does not commute in general, on restriction to the \textit{real part} (elements of degree zero) the Clifford product commutes.
\begin{lemma}
\label{ABBALemma}
 Given $a,b \in \mathcal{C}\ell_{p,q}$, $\text{Re}(ab)=\text{Re}(ba)$.
\end{lemma}
\noindent\textit{Proof.} We expand $a$ as $a = \sum_{I \subset \{1,...,p+q\}} a_Ii_I$, where the $a_I$ are real coefficients. Similarly, $b=\sum_Jb_Ji_J$. When we take the product $ab$ the pairs of terms that multiply to a real number must share their multi-index.
\[\text{Re}(ab) = \sum_{I\subseteq\{1,2,...,p+q\}}a_Ib_I(i_I)^2 = \sum_Ib_Ia_I(i_I)^2 = \text{Re}(ba).\tag*{\qed}\]
\subsubsection{Vectors and Paravectors}
The degree zero elements of $\clifn{p,q}$ are the \textit{scalars}; degree one elements are precisely the \textit{vectors} of $\R^{p,q}$, the underlying vector space of the Clifford algebra. \label{not:Paravectors}We can add scalars to vectors to get the subset of \textit{paravectors}\footnote{The altered position of the indices in $\$\R^{q+1,p}$ is justified by (\ref{eqn:CliffordNormIndices}) and Proposition \ref{LipschitzGroupIsomorphismLemma}. The dollar sign on any object indicates a general `paravector-ness'; For more discussion of paravectors see \cite{Porteous1969} page 254, \cite{Lounesto_2001}, and \cite{ParavectorsIn3D}.} $\$\R^{q+1,p} = \R \oplus \R^{p,q} \subseteq \clifn{p,q}$. Paravectors are generally denoted by boldface uppercase letters, for example $\pV$.

While it is more common to use a Clifford algebra to describe isometries of the underlying vector space $\R^{p,q}$, following the work of Ahlfors (see e.g. \cite{Ahlfors1985}) we will be taking the point of view where paravectors are the more important subset. There are various reasons to take this approach, which we will touch on, but a key reason is that the paravector point of view provides a natural continuation of the language of the complex spinor under the action of $SL(2,\C)$.

When $(p,q)=(0,1)$, this paravector space is just the complex numbers, and when $(p,q) = (0,2)$ we get a subspace of the quaternions which can be identified with the span of $\{1,i,j\}$. In general, $\$\R^{q+1,p}$ is a vector space with one real part and $p+q$ `imaginary parts', with unit generators of $N_{p,q}$-norm $\pm 1$ depending on their signature.

Paravectors are invariant under reversal, so given $\pV \in \$\R^{q+1,p}$ we have $\pV^*=\pV,\quad \pV'=\ol{\pV}$. For $\pV \in \$\R^{q+1,p}$, consider the product $\pV\ol{\pV} = \pV\pV'$. It is invariant under Clifford conjugation, so $\ol{\pV\pV'} = \ol{\pV'}\ol{\pV} = \pV\pV'$.

By Lemma \ref{MonomialInvolutionInvarianceLemma}, this invariance implies the product is the sum of monomials of degree $0,3 \mod 4$; but as a product of two paravectors the degree is bounded to $\leq 2$, implying it is purely of degree zero and is thus a scalar, so $\pV\ol{\pV} = \pV\pV' \in \R,\ \forall \pV \in \$\R^{q+1,p}$.
\begin{definition}
\label{defn:CliffordNorm}
The {\normalfont norm} of the paravector $\pV \in \$\R^{q+1,p}$ is the map
 \[N:\ \$\R^{q+1,p} \to \R,\ N(\pV)=\pV\ol{\pV}=\ol{\pV}\pV.\]
This norm will not in general be positive when $p > 0$. We may also refer to the norm as the {\normalfont magnitude squared} $|\pV|^2 = N(\pV)$, and when $p=0$ we define the {\normalfont magnitude} $|\pV| = \sqrt{|\pV|^2}$ of $\pV$. The terms `magnitude squared' and `norm' may be used interchangeably. \null\hfill\defn
\end{definition}
\noindent The product $\pV\ol{\pV}$ can be computed explicitly in terms of the standard basis. Given:
\[\pV = V_0+\sum_{j=1}^{p+q}V_ji_j,\ V_0,V_j \in \R\]
we find
\begin{equation}
\label{eqn:CliffordNormIndices}
N(\pV) = V_0^2 - \sum_{j=1}^{p+q}V_j^2i_j^2 = V_0^2 - \sum_{j=1}^{p}V_j^2 + \sum_{j=p+1}^{p+q}V_j^2.
\end{equation}
\begin{remark} As $\R^{p,q} \subset \$\R^{q+1,p}$, it inherits the norm $N$ of Definition \ref{defn:CliffordNorm}. But it already has a norm $N_{p,q}$ via Definition \ref{defn:CliffordAlgebra}. These are simply the negative of one another; if we take $V = \sum_{j=1}^{p+q}V_ji_j \in \R^{p,q}$, then
    \[N_{p,q}(V) = \sum_{j=1}^pV_j^2-\sum_{j=p+1}^{p+q}V_j^2,\]
    which is $-N(V)$, comparing to (\ref{eqn:CliffordNormIndices}).
\end{remark}
\noindent The norm can be polarised to give an inner product in the standard way.
\begin{definition}
\label{defn:ParavectorInnerProduct}
Given elements $\pV,\pW \in \$\R^{q+1,p}$, we define the inner product (or {\normalfont dot product})
\begin{equation}
\label{eqn:DotProduct}
\pV \cdot \pW = \frac{1}{2} \left(N(\pV+\pW)-N(\pV)-N(\pW) \right).
\end{equation}
It can equivalently be defined as
\[\pV \cdot \pW = \text{Re}(\pV\ol{\pW}) = \frac{\pV\ol{\pW}+\pW\ol{\pV}}{2} = V_0W_0 - \sum_{j=1}^pV_jW_j + \sum_{j=p+1}^{p+q}V_jW_j,\]
where $\pV = V_0 + \sum_{j=1}^{p+q}V_ji_j$ and $\pW$ is defined similarly.\null\hfill\defn
\end{definition}
\noindent Clifford multiplication of vectors in the underlying vector space $\R^{p,q}$ satisfies a \textit{polarisation identity}.\footnote{See \cite{LawsonMichelsohn} (1.4); note the sign difference, due to a different choice for the ideal that defines the Clifford algebra.} For $V,W \in \R^{p,q}$,
\begin{equation}
\label{eqn:PolarisationIdentity}
VW+WV=2N_{p,q}(V,W).
\end{equation}
\subsubsection{The Norm on $\clifn{p,q}$}
The norm on paravectors is a real-valued map, and this holds when extended to products of paravectors; given $a \in \clifn{p,q}$ satisfying 
\[a = \pV_1\pV_2...\pV_k, \text{ where } \pV_1,\pV_2,...,\pV_k \in \$\R^{q+1,p},\]
we find
\[a\ol{a} = \pV_1\pV_2...\pV_k\ol{\pV_k}...\ol{\pV_2}\ol{\pV_1} = N(\pV_1)N(\pV_2)...N(\pV_k) \in \R.\]
In fact a norm can be defined over the whole Clifford algebra by taking the real component of $a\ol{a}$, though it isn't necessarily useful for arbitrary elements.
\begin{definition}
    The norm of an element $a \in \clifn{p,q}$ is defined by 
    \[N:\ \clifn{p,q} \to \R,\ N(a)=\text{Re}(a\ol{a}).\tag*{\defn}\]
\end{definition}
\noindent This can also be polarised to give a form in the standard way.
\begin{definition}
    \label{def:CliffordDotProd}
    Given elements $a,b \in \clifn{p,q}$, we define the inner product (or dot product)
    \[a \cdot b = \frac{1}{2}(N(a+b)-N(a)-N(b)).\]
    Equivalently, 
    \[a \cdot b = \text{Re}(a\ol{b}) = \text{Re}\left(\frac{a\ol{b}+b\ol{a}}{2}\right).\tag*{\raisebox{-0.5em}{\defn}}\]
\end{definition}
\noindent Elements $a \in \clifn{p,q}$ that have a left or right inverse $a^{-1}$ are \textit{invertible} (and $a^{-1}$ is the two-sided inverse), because Clifford algebras can be represented as matrix algebras and this property holds for square matrices. \label{not:ClifInvSubgrp}Notate the subgroup of invertible elements in the Clifford algebra as $\mathcal{C}\ell^{\times}_{p,q}$. The inverse of an element $a \in \clifn{p,q}$ can always be written as $a^{-1} = \ol{a}(a\ol{a})^{-1}$, and if you restrict focus to products of paravectors then $N(a)=a\ol{a}$ is a real number and finding its multiplicative inverse is trivial.

\subsubsection{Paravectors as Exponentiated Vectors}
\label{sec:ParaExponential}
One way to think about the significance of paravectors is by exponentiating vectors; the exponential of the underlying vector space for the Clifford algebra will coordinatise the unit paravector sphere. 

\label{not:exponentialmap}The exponential map is defined\footnote{The exponential is actually convergent over the whole Clifford algebra; see \cite{LawsonMichelsohn}, Chapter 1, Section 2.} by the standard Taylor series $e^{V} = \sum_{j=0}^{\infty}\frac{1}{j!}V^j$. In the complex case covered in \cite{Mathews_Spinors_horospheres} this is also the standard Lie group exponential from $\R$ to the unit circle in $\C$ (Euler's map $\theta \to e^{i\theta}$), but the image of $\R^{p,q}$ under exponentiation doesn't admit a Lie group\footnote{The soon-to-be-defined Lipschitz and paravector Lipschitz groups in which the vectors and paravectors reside are Lie groups, however, so this can be seen as a restriction of a Lie group-Lie algebra correspondence to a subspace.} structure in general. For the remainder of Section \ref{sec:ParaExponential} we will take $(p,q) = (0,n)$ to reduce to the case $\$\R^{n+1,0} = \para$, which is positive definite.
\begin{proposition}
\label{EulerGeneralisationLemma}
Given a vector $V\in \R^{0,n}$ with $N(V)=1$ and $\theta \in \R$:
\[e^{V\theta} = \cos(\theta)+V\sin(\theta) \in \para.\]
\end{proposition}
\noindent\textit{Proof.} For a vector $V \in \R^{0,n}$, $V'=\ol{V} = -V$. We can compute powers of V as% follows:
%\begin{align*}
%V^2 &=V(-V')=-|V|^2,\\
%V^m &= \begin{cases}
% (-1)^{\frac{m}{2}}|V|^m &\quad \textit{if m is even}, \\
% (-1)^{\frac{m-1}{2}}|V|^{m-1}V &\quad \textit{if m is odd}.
%\end{cases}
%\end{align*}
%We can put this in a more compact form:
\[V^m =(-|V|^2)^{\lfloor m/2 \rfloor}V^{m\mod 2}\]
where $\lfloor \cdot \rfloor$ is the floor function and $m\mod 2$ is the standard modulus function. Substituting this into the exponential gives us
\[e^{V} =\sum_{m=0}^{\infty}\frac{(-|V|^2)^{\lfloor m/2 \rfloor}}{m!}V^{m\mod 2} = \left(\sum_{m=0}^{\infty}\frac{(-1)^m|V|^{2m}}{(2m)!}+\sum_{m=0}^{\infty}\frac{(-1)^m|V|^{2m}}{(2m+1)!}V\right).\]
We recognise the Taylor series for $\cos|V|$ and $\sin|V|$:
\[e^{V} = \left(\cos|V|+\frac{V}{|V|}\sin|V| \right).\]
Introducing $\theta$ gives us
\[e^{V\theta} = \cos(|V\theta|)+\frac{V\theta}{|V\theta|}\sin(|V\theta|) = \cos(|V|\theta)+\frac{V}{|V|}\sin(|V|\theta).\]
%where the last equality holds if $\theta \geq 0$. If $\theta < 0$, we instead find
%\[e^{V\theta} = \cos(-|V|\theta)+\frac{V\theta}{-|V|\theta}\sin(-|V|\theta) = \cos(|V|\theta)+\frac{V}{|V|}\sin(|V|\theta).\]
By normalising so $|V|=1$, we find the desired expression. \hfill\qed
\begin{remark}
Exponentiating degree $2$ forms $B \in \lip$ of unit norm will result in the same trigonometric identity: 
\[e^{B\theta} = \cos(\theta)+B\sin(\theta),\]
as these satisfy
\[B^2=B(-\ol{B})=-|B|^2 = -1.\]
However, for forms $F \in \lip$ of degree $3$ or $4$ with unit norm we instead find $F^2=|F|^2$, giving the identity
\[e^{F\theta}=\cosh(\theta)+F\sinh(\theta).\]
Generally, given any homogeneous element $Y \in \lip$ of degree $d$ and unit norm, if $d$ is congruent to $0,3 \mod 4$ we find 
\[e^{Y\theta}=\cosh(\theta)+Y\sinh(\theta).\]
If, instead, $d \cong 1,2 \mod 4$, we find
\[e^{Y\theta} = \cos(\theta)+Y\sin(\theta).\]
\end{remark}
\noindent We can exploit this exponential map to parameterise the unit $n$-sphere, just as the unit circle can be parameterised by $e^{i\theta}$. Take the Clifford algebras $\clif$; their paravector subspaces $\para$ have a positive-definite norm via Definition \ref{defn:CliffordNorm}. Consider the unit $n$-sphere $S^n$ in $\para \subset \clif$ under this norm, and consider the subset $S^{n-1} \subset S^n$ formed by the unit $(n-1)$-sphere in $\R^{0,n} \subset \para$. Take a unit vector $V \in S^{n-1}$ and an angle $\theta \in \R$. By the proposition above, $e^{V\theta}$ has unit length, and so lies in $S^n$. 

Geometrically, to construct $e^{V\theta}$ we pick the point $V$ in $S^{n-1}$, take the plane spanned by $1$ and $V$, and rotate in this plane by angle $\theta$ from $1$ towards $V$ (see Figure \ref{fig:SphereParameterisation}).

\begin{figure}[ht!]
\centering
\begin{tikzpicture}[tdplot_main_coords, scale = 2.5]

\coordinate (P) at ({1/sqrt(3)},{1/sqrt(3)},{1/sqrt(3)});
\coordinate (V) at ({1/sqrt(2)},{0},{1/sqrt(2)});
\coordinate (O) at ({0},{0},{0});
\coordinate (U) at ({0},{1},{0});

\shade[ball color = lightgray,
	opacity = 0.5
] (0,0,0) circle (1cm);
 
\tdplotsetrotatedcoords{0}{135}{0};
\draw[dashed,
	tdplot_rotated_coords,
	gray
] (0,0,0) circle (1);

\tdplotsetrotatedcoords{90}{90}{90};
\draw[dashed,
	tdplot_rotated_coords,
	gray
] (1,0,0) arc (0:360:1);

\draw[-stealth] (0,0,0) -- (0,1.30,0)
	node[below right] {$1$};

\draw[-stealth] (0,0,0) -- (1.80,0,0) 
	node[below left] {$i_1,i_2,...$};
\draw[-stealth] (0,0,0) -- (0,0,1.30)
	node[above] {$i_j,i_{j+1}...$};

\draw[thick, -stealth] (0,0,0) -- (P) node[right] {$e^{V\theta}$};
\draw[thick, -stealth] (0,0,0) -- (V) node[right] {$V$};

\draw[fill = lightgray!50] (P) circle (0.5pt);
\draw[fill = lightgray!50] (V) circle (0.5pt);
\draw[fill = lightgray!50] (U) circle (0.5pt);

\pic [draw,->,red,"$\theta$", angle radius=0.4cm, angle eccentricity=1.5] {angle = U--O--P};
\end{tikzpicture}
\caption{Parameterisation of the sphere.}
\label{fig:SphereParameterisation}
\end{figure}
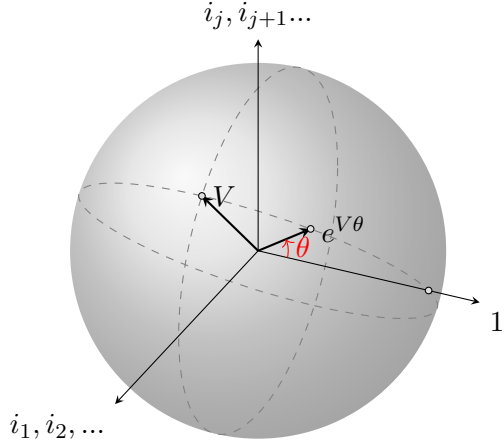
\subsection{The Lipschitz Group}
\label{sec:LipschitzGroup}
$\clifn{p,q}$ can act on itself in various ways, commonly by some variant of the adjoint action. To ensure the action is reversible (what we're really interested in down the line are isometries), we'll restrict to the action of invertible elements.
\begin{definition}[Adjoint Actions]
\label{defn:TwistedAdjointSigma}
Given an element $x \in \mathcal{C}\ell^{\times}_{p,q}$, we define an action 
\[\sigma: \mathcal{C}\ell^{\times}_{p,q} \to End(\clifn{p,q}),\ \sigma(x)(y) = xyx^*.\]
This action is multiplicative; that is,
\[\sigma(x)\sigma(y) =\sigma(xy).\]
\label{not:TwistAdj}This definition is a bit unusual; the more common\footnote{See, for example, \cite{LawsonMichelsohn}.} {\normalfont twisted adjoint action} is defined by
\[\widetilde{Ad}_xy = x'yx^{-1},\] 
which is essentially $\sigma$ up to a scaling factor. \label{not:Adj}The standard {\normalfont adjoint action} is
\[Ad_xy=xyx^{-1}.\tag*{\defn}\]
\end{definition}
\noindent The adjoint action, when restricted appropriately, gives isometries of $\R^{p,q}$, while the twisted adjoint resolves a subtlety where the standard adjoint action misses some orientation reversing isometries. However, we want to maintain the scaling factor; it appears naturally in discussion of the lambda length defined in Section \ref{LambdaLengthsSection}. So the action $\sigma(x)$ is by isometries, plus a dilation.

Clifford algebras essentially exist to describe isometries, but in general they contain many elements that do not act via an isometry. To combat this, we focus our attention on elements whose action under $\sigma$ restricts to an automorphism on some subspace of the Clifford algebra. The obvious choice is to consider the subgroup of the Clifford algebra whose action fixes the underlying vector space $\R^{p,q}$; this then defines a map $\R^{p,q} \to \R^{p,q}$. This is called a \textit{Lipschitz group}.
\begin{definition}[Lipschitz Groups]
\label{def:LipschitzGrp}
 The {\normalfont Lipschitz group} $\Gamma_{p,q}$ is the group generated by products of invertible vectors $V \in \R^{p,q}$. We can also define the even Lipschitz group:
\[\Gamma_{p,q}^+ = \Gamma_{p,q} \cap \mathcal{C}\ell^+_{p,q}.\]
For all such objects if $q=0$ we drop it from the notation, writing $\Gamma_{p,0}= \Gamma_p$ and so on. \hfill\defn
\end{definition}
\noindent There are alternative characterisations of the Lipschitz group, which can prove useful. Lounesto (\cite{Lounesto_2001}, p. 220) tells us the Lipschitz group $\Gamma_{p,q}$ has the two following alternative characterisations:
\begin{align*}
 \Gamma_{p,q} &= \{s \in \clifn{p,q}\ |\ \forall V \in \R^{p,q},\ s'Vs^{-1} \in \R^{p,q}\},\\
 \Gamma_{p,q} &= \{s \in \clifn{p,q}\ |\ \forall V \in \R^{p,q},\ sVs^* \in \R^{p,q} \text{ and } N(s) \in \R \backslash \{0\}\}.
 \end{align*}
In the second characterisation we seem to have an extra condition, that the norm of $s$ is real; but this is implied by $s'Vs^{-1} \in \R^{p,q}$ in the first characterisation. Within the Lipschitz groups, we find the spin groups.
\begin{definition}[Spin Group]
\label{def:SpinGrp}
    The spin group Spin$(p,q)$ is the `unit' subgroup of $\Gamma^+_{p,q}$. 
    \[\text{Spin}(p,q) = \{s\ |\ s \in \Gamma^+_{p,q} \text{ and } N(s)=\pm 1\}.\]
    When $q=0$, we instead define
    \[\text{Spin}(n) = \{s\ |\ s \in \Gamma^+_n \text{ and } N(s) = 1\}.\tag*{\defn}\]
\end{definition}
\noindent It is a standard result that the twisted adjoint action of the Lipschitz group is isomorphic to the action of the orthogonal group $O(p,q)$ on $\R^{p,q}$; that is, it gives the isometries\footnote{see, for example, \cite{LawsonMichelsohn} Corollary 2.6.} of $\R^{p,q}$.

However, to utilise Ahlfors' work \cite{Ahlfors1985},\cite{Ahlfors1986} (which builds its objects from paravectors rather than vectors), we consider the isometries of the paravectors $\$\R^{q+1,p}$ instead.
\begin{definition}[Paravector Lipschitz Group]
\label{def:ParaLipschitzGrp}
 The paravector Lipschitz group\footnote{A good discussion is in the book by Porteous in \cite{Porteous1969}, page 254.} $\$\Gamma_{q+1,p}$ is the subgroup of $\clifn{p,q}$ generated by products of invertible paravectors. \hfill\defn
\end{definition}
\noindent As with the more standard Lipschitz group, there are multiple ways to define the paravector Lipschitz group; these are equivalent, just different characterisations, and sometimes these alternate definitions prove useful. From Lounesto (\cite{Lounesto_2001} p. 225), we have that the paravector Lipschitz group $\$\Gamma_{q+1,p}$ is alternatively characterised by
\begin{align*}
 \$\Gamma_{q+1,p} &= \{s \in \clifn{p,q}\ |\ \forall\ \pV \in \$\R^{q+1,p},\ s'\pV s^{-1} \in \$\R^{q+1,p}\},\\
 \$\Gamma_{q+1,p} &= \{s \in \clifn{p,q}\ |\ \forall\ \pV \in \$\R^{q+1,p},\ s\pV s^* \in \$\R^{q+1,p}\text{ and } N(s) \in \R \backslash \{0\}\}.
 \end{align*}
\label{not:LipschitzMonoid}Following \cite{TheoryOfCliffordBianchiGroups}, we also define the \textit{Lipschitz monoid} $\$\Gamma_{q+1,p}^{\triangleright} = \$\Gamma_{q+1,p} \cup \{0\}$ for later convenience. 

The notation of the group $\$\Gamma_{q+1,p}$ is suggestive; although it's drawn from the Clifford algebra with signature $(p,q)$, we've moved around the indices. Indeed, this group is isomorphic to one of the more standard Lipschitz groups.
\begin{proposition}
\label{LipschitzGroupIsomorphismLemma}
With respect to their group structure, $\$\Gamma_{q+1,p} \cong \Gamma^+_{q+1,p}$.
\end{proposition}
\noindent \textit{Proof.} See p.225 of \cite{Lounesto_2001}, and p.3 of \cite{Ahlfors1985}; the identification follows because the norm on paravectors has signature $(q+1,p)$ (see (\ref{eqn:CliffordNormIndices})). \qed
\\\\
This also relates the paravector spin groups to the more usually defined ones.
\begin{corollary}
\label{cor:ParavSpinGrp}
The subgroup
\[\$\text{pin}(q+1,p)= \{s \in \$\Gamma_{q+1,p}\ |\ N(s) = \pm 1 \}\]
is isomorphic to Spin$(q+1,p)$. The subgroup
\[\$\text{pin}(n+1) = \{s \in \$\Gamma_{n+1}\ |\ N(s) = 1\}\] 
is isomorphic to Spin$(n+1)$.\qed
\end{corollary}
\noindent Taken together, Proposition \ref{LipschitzGroupIsomorphismLemma} and Corollary \ref{cor:ParavSpinGrp} allow us to write the standard spin and Lipschitz groups in the language of paravectors. Note that this isomorphism doesn't respect the $\Z_2$ grading of the Clifford algebra.

Though Lipschitz groups are not in general commutative, the following lemma allows for some flexibility when rearranging terms.
\begin{lemma}
 \label{PartialParavectorCommutingLemma}
 For any $a \in \$\Gamma_{q+1,p}$ and paravector $\pV \in \$\R^{q+1,p}$, there exists $\pU \in \$\R^{q+1,p}$ such that $a\pV = \pU a'$, as well as $\pW \in \$\R^{q+1,p}$ such that $a\pV = \pW a$.
\end{lemma}
\noindent\textit{Proof.} We can rearrange the equation to find $\pU = a\pV a'^{-1}$, which is the twisted adjoint action $\widetilde{Ad}_{a'}\pV$. This action preserves $\$\R^{q+1,p}$ by definition, implying $\pU \in \$\R^{q+1,p}$ as claimed. Similarly, we can rearrange the second equation as $\pW = a\pV a^{-1} = Ad_a(\pV)$,
which is the ordinary adjoint action; this also preserves $\$\R^{q+1,p}$, so $\pW \in \$\R^{q+1,p}$.\qed
\subsubsection{The Action of $\lip$} \textbf{From here onwards we will consider only the Lipschitz groups $\lip = \$\Gamma_{n+1,0}$ associated with the negative-definite Clifford algebras $\clif$}. The following lemma gives a concrete geometric interpretation of the action of paravectors.
\begin{lemma}
 \label{ParavectorActionFixedSpaceLemma}
  Consider a unit paravector $\pU = e^{V\theta} \notin \R$, $V \in \R^{0,n}$, $\theta \in \R$. The set of fixed paravectors under the action $\sigma(\pU)$ on $\para$ is the subspace of $\para$ consisting of all vectors $W \in \R^{0,n}$ satisfying $\pU \cdot W=0$. The action $\sigma(\pU)$ is a rotation in the plane spanned by $\R$ and $\pU$, a rotation of $2\theta$ from $\R$ towards $\pU$.
\end{lemma}
\noindent\textit{Proof.} We want fixed directions, which is equivalent to finding fixed vectors as $|\pU|=1$ and $\sigma(\pU)$ is an isometry. Let $\pW \in \para$ be fixed by $\sigma(\pU)$.
\[\pU\pW\pU^* = \pU\pW \pU=\pW \implies \pU\pW = \pW \pU',\]
as $\pU^* = \pU$ on paravectors and $\pU^{-1} = \ol{\pU}|\pU|^{-2} = \ol{\pU} = \pU'$. We expand $\pU$ and $\pW$ in terms of coordinates, rearrange, and apply (\ref{eqn:PolarisationIdentity}) to find the following:
%\[\left(U_0+\sum_{j=1}^nU_ji_j\right)\left(W_0+\sum_{j=1}^nW_ji_j\right) = \left(W_0+\sum_{j=1}^nW_ji_j\right)\left(U_0-\sum_{j=1}^nU_ji_j\right),\]
%and expand brackets:
%\[U_0W_0+U_0\sum_{j=1}^nW_ji_j + \sum_{j=1}^nU_ji_jW_0+\sum_{j=1}^nU_ji_j\sum_{j=1}^nW_ji_j = W_0U_0 - W_0\sum_{j=1}^nU_ji_j + \sum_{j=1}^nW_ji_jU_0 -\sum_{j=1}^nW_ji_j\sum_{j=1}^nU_ji_j.\]
%We can shuffle the real coefficients $U_j$, $W_j$ about to cancel some terms:
%\begin{align*}
%\sum_{j=1}^nU_jW_0i_j+\sum_{j=1}^nU_ji_j\sum_{j=1}^nW_ji_j &= - \sum_{j=1}^nU_jW_0i_j - \sum_{j=1}^nW_ji_j\sum_{j=1}^nU_ji_j\\
%2\sum_{j=1}^nU_jW_0i_j &= - \left(\sum_{j=1}^nU_ji_j\sum_{j=1}^nW_ji_j + \sum_{j=1}^nW_ji_j\sum_{j=1}^nU_ji_j\right).
%\end{align*}
%But $U = \sum_{j=1}^nU_ji_j$ and $W = \sum_{j=1}^nW_ji_j$ are both elements of $\R^{0,n}$, so we can apply (\ref{eqn:PolarisationIdentity})
%\[UW+WU = 2N_{p,q}(U,W) = -2\sum_{j=1}^nU_jW_j\]
%to find
\[\sum_{j=1}^nU_jW_0i_j = \sum_{j=1}^nU_jW_j.\]
As the left side is strictly imaginary and the right is strictly real, they must both be zero. The condition $\sum_{j=1}^nU_jW_0i_j=0$ implies either $W_0=0$ or $\pU$ is real, but we've already assumed $\pU \notin \R$ so we conclude $W_0=0$ and $\pW = W \in \R^{0,n}$. The condition $\sum_{j=1}^nU_jW_j=0$ then implies $\pW$ is orthogonal under the dot product to the purely vector component of $\pU$, implying the first claim.

So our fixed subspace under $\sigma(\pU)$ is the orthogonal subspace to the vector component of $\pU$, implying the action (which we know is an isometry) must be an orientation-preserving linear isometry in the plane spanned by $\R$ and the vector component of $\pU$ (which is equivalent to the span of $\{\R,\pU\}$). The only orientation-preserving linear isometries of a vector space are rotations, implying the second part of the lemma.

To see precisely which rotation, let $\pU=e^{\theta V}$ for $V \in \R^{0,n}$ act on $1 \in \para$. Then $\sigma(\pU)(1) = e^{\theta V}(1)e^{\theta V} = e^{2\theta V}$, where the last equality follows from expanding as in Proposition \ref{EulerGeneralisationLemma}. The paravector $1$ is rotated towards $\pU$ by an angle of $2\theta$.\qed

\begin{corollary}
\label{IsometryGeneratedByRotationsCorollary}
All orientation-preserving linear isometries of the normed vector space $(\para,N)$ can be generated by a product of rotations in planes spanned by $1$ and a paravector $\not\in \R$. \qed
\end{corollary}
\subsubsection{Lie Algebras and Tangent Spaces}
There is a standard description\footnote{\cite{Lounesto_2001} pp. 221-224.} of the Lie algebras of spin groups $Spin(p,q)$ as the space of bivectors ${\bigwedge}^{\!2}\R^{p,q}$ drawn from the exterior algebra $\bigwedge\R^{p,q}$. We can find a similar description for the Lie algebra of $\lip$. 
\begin{theorem}
\label{LipschitzLieAlgebraTheorem}
 The Lie algebra $\$\mathfrak{g}_{n+1}$ of $\lip$ consists of elements of $\clif$ whose terms have degree $\leq 2$. That is,
 \[\$\mathfrak{g}_{n+1} = \R \oplus \R^{0,n} \oplus {\bigwedge}^{\!2}\R^{0,n},\]
 where ${\bigwedge}^{\!2}\R^{0,n}$ is the space of degree two elements of the Clifford algebra. The Lie bracket operation is the commutator.
\end{theorem}
\noindent We write ${\bigwedge}^{\!2}\R^{0,n}$ for the degree two elements of the Clifford algebra because there is a canonical vector space isomorphism between the exterior algebra over a vector space $V$ and the Clifford algebra over $V$ for any given quadratic form $Q$ (see \cite{LawsonMichelsohn}, Prop. 1.3). Notate this Clifford algebra by $\mathcal{C}\ell(V,Q)$. By an abuse of notation, we may then speak of the embeddings ${\bigwedge}^{\!r} V \subset \mathcal{C}\ell(V,Q)$. Note that this is only an isomorphism of vector spaces, \textit{not} of algebras.
\\\\
\noindent\textit{Proof of \ref{LipschitzLieAlgebraTheorem}.} Let $a \in \lip$ and $\pV \in \para$. We can differentiate the relations defining $\lip$:
\[a\pV a^* \in \para \text{ and } N(a) \in \R \backslash \{0\},\]
and evaluate at the identity to describe the Lie algebra. 

Consider a smooth path $a_t \in \lip,\ t \in (-\varepsilon,\varepsilon)$ for some $\varepsilon>0$ satisfying $a_0=1$. Then $a_t = 1+tH+O(t^2)$ for some element $H \in \clif$. We substitute this into our first defining relation, expand, differentiate, then evaluate at t=0:
%\[(1+tH+O(t^2))\pV(1+tH+O(t^2))^* = \pW_t,\quad \pW_t \in \para,\]
%then expand and differentiate:
%\[H\pV + (H\pV)^* + O(t)... = \dot{\pW}_t,\]
%then evaluate at $t=0$.
\[H\pV+(H\pV)^* = \dot{\pW}_0 \in T_{\pW_0}\para \cong \para.\]
We expand this relation in terms of the components $\pV= V_0 + \sum V_ji_j$ to find:
\begin{align*}
%H\left(V_0+\sum_{j=1}^nV_ji_j\right) + \left(H\left(V_0+\sum_{j=1}^nV_ji_j\right)\right)^* &\in \para\\
%\implies 
V_0(H+H^*)+\sum_{j=1}^nV_j(Hi_j+(Hi_j)^*) &\in \para.
\end{align*}
By setting $V_0=1$ and $V_j=0,\ 1\leq j \leq n$ we see that $H+H^* \in \para$. Consulting Lemma \ref{MonomialInvolutionInvarianceLemma}, this implies it only has terms of degrees $0,1,2+4m,3+4m$ for $m \in \mathbb{N}$ as other terms of degree $0+4m,1+4m$ wouldn't cancel appropriately. Similarly, by setting all $V_j$ to zero except for one $j=\widehat{j}$, we see that $Hi_{\widehat{j}}+(Hi_{\widehat{j}})^* \in \para$. This implies $Hi_{\widehat{j}}$ also only has terms of degrees $0,1,2+4m,3+4m$ for all $i_j$. If $H$ contained any terms of degree $2+4m$ (say a term $i_ji_k...$) then in the product $Hi_j$ that term would become degree $1+4m$ as the $i_j$'s cancel; this conflicts with our above conclusion, so $H$ cannot contain terms of degree $2+4m$ outside of the case $m=0$, since we allow degree 1 terms.

Now bring in our second condition:
\[N(1+tH+...)= (1+tH+...)\ol{(1+tH+...)}\in \R \backslash \{0\},\]
which implies to first order $1+t(H+\ol{H}) \in \R \backslash \{0\}$, which implies $H+\ol{H} \in \R$. Lemma \ref{MonomialInvolutionInvarianceLemma} tells us terms of degree 3 or 0 mod 4 are invariant under $\ol{\cdot}$, implying $H$ only has terms of degree $0,1+4m,2+4m$. Combining all these restrictions, the Lie algebra only allows terms of degree $0,1$, and $2$. 

For the converse, we know from Corollary 2.6 (and the discussion of Theorem 2.7) of \cite{LawsonMichelsohn} that $\lip$ has a surjective homomorphism onto $SO(n+1)$ with kernel $\R^+$. This implies it is of dimension $n(n+1)/2+1$, which is therefore the dimension of its Lie algebra. By the arguments above, the Lie algebra must be contained in the vector space  $\R \oplus \R^{0,n} \oplus {\bigwedge}^{\!2}\R^{0,n}$, and because this space is also of dimension $n(n+1)/2+1$, they are isomorphic. \qed
\\\\
As is standard (see, for example, \cite{ManifoldsTu} \S 16), we can describe tangent spaces at other points of $\lip$ by left translation. First, a definition:
\begin{definition}
 \label{BiParavectorDefinition}
 The set of elements 
 \[\mathcal{B}^n = \R \oplus \R^{0,n} \oplus {\bigwedge}^{\!2}\R^{0,n}\]
 of $\clifn{0,n}$ are called the {\normalfont bi-paravectors}.\null\hfill\defn
\end{definition}
\begin{corollary}
\label{cor:TangentSpacesLip}
 The tangent space $T_s\lip$ at $s \in \lip$ is described by the subset of $\clif$ given by products $sB$ for bi-paravectors $B \in \mathcal{B}^n$. \hfill\qed
\end{corollary}
\noindent Applying the isomorphism of Proposition \ref{LipschitzGroupIsomorphismLemma}, ``bivectors'' of $\clifn{n+1}$ with basis elements $1,i_1,i_2,...,i_n$ are precisely the bi-paravectors.

It is a standard fact of Lie groups that inverses in the group correspond to negation in the Lie algebra:
 \[(e^G)^{-1}=e^{-G}.\]
 What about the other involutions?
\begin{lemma}
\label{lem:InvolutionsLieAlg}
 Each of the involutions $\cdot'$, $\cdot^*$, and $\ol{\cdot}$ on $\lip$ corresponds to taking that same involution in the Lie algebra, for example:
 \[\ol{(e^G)} = e^{\ol{G}}. \]
\end{lemma}
\noindent\textit{Proof.} Consider the involution $\cdot'$. If we take the expansion of the exponential:
\[e^{tG} = 1+tG+t^2G^2/2+...\]
and we apply the involution to both sides:
\[(e^{tG})' = 1+tG'+t^2(G^2)'/2+... ,\]
we find that, on restriction to first order, $(e^G)'$ should represented in the tangent space by $G'$. As the bi-paravectors are closed under $\cdot'$, $G'$ is a well-defined element of the Lie algebra. A similar argument holds for the other involutions.
\qed

\subsection{Special Linear Groups}
\label{SectionTwoComponentSpinors}
In \cite{Ahlfors1985}, matrix groups are constructed that generalise the special and general linear groups of matrices over commutative rings. For our purposes, we're interested in the groups\footnote{In Ahlfors' definition the groups $SL(2,\cdot)$ are the subgroups of $GL(2,\cdot)$ with pseudo-determinant $\pm 1$, and $SL_+$ this is the connected component with identity. As we aren't interested in anything but definite and orientation-preserving maps, our naming is simplified.} $SL(2,\lip)$ as they generalise the role of $SL(2,\C)$ in providing a double cover of the isometry group of $\hyp^3$.
\begin{definition}[Special Linear Group]
\label{SL2LipDef}
 The matrix $\begin{pmatrix}
 a & b \\
 c & d 
 \end{pmatrix}$ is an element of $SL(2,\lip)$ if
 \begin{enumerate}[label=\roman*)]
 \item $a,b,c,d \in \lip^{\triangleright}$,
 \item $a^*d-c^*b = 1$,
 \item $ab^*,cd^*,c^*a,d^*b \in \para$.
 \end{enumerate}
 We refer to an element of $SL(2,\lip)$ as a {\normalfont Clifford matrix}. We similarly define the general linear group $GL(2,\lip)$ by relaxing condition ii) to $a^*d-c^*b\in \R\backslash \{0\}$.
 \null\hfill\defn
\end{definition}
\noindent  Condition i) is necessary to ensure the matrix is invertible, while condition ii) enforces the `special' part of `special linear group'. \label{not:pdet}The quantity $a^*d-c^*b$ is a noncommutative extension\footnote{There are multiple ways to generalise the determinant to the noncommutative setting; we use another notion, the quasideterminant (introduced in work of Gelfand and Retakh \cite{GelfandRetakh91} with a good overview given in \cite{gelfand2004quasideterminants}), in the discussion of Ptolemy relations (see Section \ref{sec:PtolemyRelns}).} of the determinant; we refer to it as the \textit{pseudo-determinant}, or pdet. The pseudo-determinant is multiplicative when restricted to $GL(2,\lip)$ (see \cite{Ahlfors1985}), which of course extends to $SL(2,\lip)$.

\label{not:1ptCompact}An action (M\"{o}bius transformation) is defined on the one-point compactification of the paravectors $\cpara = \para \cup \{\infty\}$ (also called the \textit{paravector sphere}, or \textit{generalised Riemann sphere}) according to the rule
\begin{equation}
\label{ParavectorMobiusAction}
 \begin{pmatrix}
 a & b \\
 c & d 
 \end{pmatrix}\pV = (a\pV+b)(c\pV+d)^{-1}.
\end{equation}
This is a bijective mapping $\cpara \to \cpara$ (see \cite{Ahlfors1985}); condition iii) is required to ensure this mapping is a genuine map $\cpara \to \cpara$.

\subsubsection{$SL(2,\lip)$ and the Isometries of $\hyp^n$}
Clifford algebras describe the isometries of a vector space. How does one translate this into the isometries of hyperbolic space? 

The action (\ref{ParavectorMobiusAction}) is a bijective map of the `generalised Riemann sphere' $\cpara$ onto itself; consider $\$\R^n \subset \para$ in the obvious way, by fixing the coefficient of $i_n$ at zero.\footnote{Recall that $\para$ is generated by $1$ and $i_1$ through $i_n$.} The subspace $\$\R^n$ is then acted on by $SL(2,\$\Gamma_{n})$, but $SL(2,\$\Gamma_{n}) \subset SL(2,\lip)$; we can thus consider $SL(2,\$\Gamma_{n})$ as a subgroup of $SL(2,\lip)$, one that fixes $\$\R^n \subset \para$. \label{not:UpperHalfSpace}In particular, this implies the upper half-space $\U^{n+1} \subset \para$ is preserved,\footnote{To be precise, we must also note that $SL(2,\lip)$ is orientation-preserving so the upper and lower half-spaces aren't interchanged.} so this subgroup naturally acts on a model of hyperbolic space.
\begin{lemma}
 Every $g \in SL(2,\$\Gamma_{n})$ maps the upper half-space $\U^{n+1} \subset \para$ onto itself, and
 \[\left|Dg(x)\right| = (gx)_n/x_n\]
 for all $x \in \U^{n+1}$, where the $Dg$ is the Jacobian matrix and $x_n$ is the coefficient of $i_n$ in the paravector expansion.
\end{lemma}
\noindent\textit{Proof.} See \cite{Ahlfors1986B} Section 4.\qed
\\\\
As Ahlfors demonstrates (following the ideas of Vahlen and Maass), this is essentially the statement of the invariance of the hyperbolic metric on $\U^{n+1}$, and this is enough to recover the entire (orientation-preserving) isometry group $PSL(2,\lip) = SL(2,\lip)/\{I,-I\}$.
\subsection{Two-Component Lipschitz Spinors}
\label{TwoComponentLipschitzSpinors}
We can projectivise the action of $SL(2,\lip)$ on the generalised Riemann sphere. Given $\xi,\eta \in \lip^{\triangleright}$ satisfying $\xi \eta^{-1} = \pV \in \para$, instead of working with the paravector $\pV$ we can work with these `projective coordinates' $(\xi,\eta)$. 
\begin{definition} 
\label{def:LipschitzSpinor}
A {\normalfont Lipschitz spinor} is a pair
\[\kappa = \begin{pmatrix}
    \xi \\
    \eta
\end{pmatrix},\]
with $\xi,\eta \in \lip^{\triangleright}$ (not both equally zero) which satisfy $\xi\ol{\eta} \in \para$, or equivalently $\xi\eta^{-1} \in \cpara$. \label{not:SLip}The set of Lipschitz spinors is denoted $S\lip$. More formally, we can define the `ratio map':
\[\mathscr{R}:\ S\lip \to \cpara,\ (\xi,\eta) \to \xi\eta^{-1},\]
which sends Lipschitz spinors to the generalised Riemann sphere. \label{not:ConjTrans}Lipschitz spinors also have a conjugate transpose, indicated by a superscript dagger:
\[\kappa^{\dagger} = \begin{pmatrix}
 \ol{\xi} & \ol{\eta}
\end{pmatrix}. \tag*{\defn}\]
\end{definition}
\noindent To the author's knowledge, this definition of a generalised two-component spinor has not appeared in the literature before. 
\begin{lemma}
\label{lem:xiStarEtaParavector}
    For $\kappa = (\xi,\eta) \in S\lip$,
    \[\xi^*\eta,\ \xi\eta^* \in \para.\]
\end{lemma}
\noindent\textit{Proof.} As $\kappa$ is a Lipschitz spinor, $\xi\eta^{-1} = \pV \in \cpara$. If $\eta=0$ then $\xi^*\eta=0 \in \para$, and if $\eta \neq 0$ we compute
\[\xi =  \pV\eta \implies \xi^*\eta = \eta^*\pV\eta = \sigma(\eta^*)(\pV) \in \para,\]
as $\sigma(x)$ for $x \in \lip$ is an action on $\para$. For $\xi\eta^*$, again if $\eta=0$ then $\xi\eta^* = 0 \in \para$. If $\eta \neq 0$, then
\[\xi= \pV \eta = \eta\eta^{-1}\pV \eta = \eta Ad_{\eta^{-1}}(\pV) = \eta \pU,\]
where $\pU = Ad_{\eta^{-1}}(\pV) \in \para$. Then 
\[\xi \eta^* = \eta \pU \eta^* = \sigma(\eta)(\pU) \in \para. \tag*{\qed}\]
\begin{corollary}
\label{cor:xistarEta}
For $(\xi,\eta)\in S\lip$, 
\[\xi^*\eta = \eta^*\xi, \quad \xi\eta^* = \eta\xi^*.\]
\end{corollary}
\noindent\textit{Proof.} Since $\xi^*\eta \in \para$ and paravectors are invariant under the reversal conjugation, $\xi^*\eta = (\xi^*\eta)^* = \eta^*\xi$. The same argument holds for $\xi\eta^*$.\qed
\\\\
The conditions of the previous corollary seem reminiscent of the conditions on the columns and rows of $SL(2,\lip)$, and that isn't a coincidence. The generalised special linear groups can also be characterised as $2\times 2$ matrices with Lipschitz spinors for rows and columns, and pseudo-determinant $1$.
\begin{proposition}
\label{prop:AlternateSLCharacterisation}
    $SL(2,\lip)$ can equivalently be characterised as the matrices $\begin{pmatrix}
        a & b \\
        c & d
    \end{pmatrix}$ satisfying
    \begin{enumerate}[label=\roman*)]
        \item $a,b,c,d \in \lip^{\triangleright}$,
        \item $a^*d-c^*b = 1$,
        \item $(a,b),(c,d),(a,c),(b,d) \in S\lip$.
    \end{enumerate}
\end{proposition}
\noindent\textit{Proof.} Obviously conditions i), ii) are identical to the definition of $SL(2,\lip)$, so we need to show that condition iii) of Definition \ref{SL2LipDef} implies condition iii) above, and vice versa. 

Note that $(a,b) \neq (0,0)$, or else the pseudo-determinant would be zero; similar holds for the other pairs in iii). If $ab^* = \pV \in \para$, then 
\[a = \pV b^{*-1} = bb^{-1}\pV b^{*-1} = b\sigma(b^{-1})(\pV).\]
As $\sigma$ is a map $\para \to \para$ this is of the form `b times a paravector'. By Lemma \ref{PartialParavectorCommutingLemma}, there exists a $\pW \in \para$ such that $a = \pW b$, which implies $a\ol{b} = |b|^2\pW$ and therefore $(a,b) \neq (0,0) \in S\lip$. The same argument holds for the other cases.

For the converse, if $(a,b) \in S\lip$ then $ab^* \in \para$ by Lemma \ref{lem:xiStarEtaParavector}, and the same argument holds for the other three spinors in iii).\qed

\subsubsection{The Action of $SL(2,\lip)$ on Lipschitz Spinors}
We want to act on these Lipschitz spinors by our special linear groups. The obvious thing to try is acting by matrix multiplication; take $(\xi,\eta)^T \in S\lip$ to be a column vector, and multiply by $A \in SL(2,\lip)$ on the left. Not only does this send Lipschitz spinors to Lipschitz spinors (ensuring it truly is an action on $S\lip$), but it is also equivalent to acting by M\"{o}bius transformations on the paravector $\xi\eta^{-1}$, so it really is the right choice.
\begin{lemma}
\label{lem:SL2DefinesSpinorAction}
Left-multiplication by $A \in SL(2,\lip)$ defines an action of $SL(2,\lip)$ on $S\lip$. The action of $A$ on $\pV \in \cpara$ corresponds equivariantly to the action by matrix multiplication on the Lipschitz spinor $(\xi,\eta)$ with $\xi\eta^{-1}=\pV$.
\end{lemma}
\noindent\textit{Proof.} Take the action by matrix multiplication:
\[\begin{pmatrix}
a & b \\
c & d
\end{pmatrix}
\begin{pmatrix}
 \xi \\
 \eta
\end{pmatrix}=
\begin{pmatrix}
a\xi + b\eta \\
c\xi + d\eta \\
\end{pmatrix}.\]
Let $\pW=(a\xi + b\eta)(c\xi + d\eta)^{-1}$ be the projective image of the transformation. We can insert $1=\eta^{-1}\eta$ between the brackets:
\begin{align*}
\pW &= (a\xi + b\eta)\eta^{-1}\eta(c\xi + d\eta)^{-1} = (a\xi\eta^{-1}+b)(c\xi\eta^{-1}+d)^{-1} = (a\pV+b)(c\pV+d)^{-1}.
\end{align*}
This is the same fractional linear transformation as (\ref{ParavectorMobiusAction}), which implies left-multiplication by $SL(2,\lip)$ is a well-defined action $S\lip \to S\lip$, and implies equivariance of the $SL(2,\lip)$ action through the map $\mathscr{R}$ from Lipschitz spinors to the generalised Riemann sphere.\qed
\\\\
We may also denote this action by $A.\kappa = A.(\xi,\eta)$. We can extend this to an action by matrix multiplication of $SL(2,\lip)$ on $\clif^2$ more generally. This is more obviously a map $\clif^2 \to \clif^2$, as algebras are closed under addition and multiplication.
\begin{lemma}
\label{TransitiveLemma}
 The action of $SL(2,\lip)$ is transitive on $S\lip$.
\end{lemma}
\noindent\textit{Proof.} Any Lipschitz spinor $(\xi,\eta)$ can be found as the first column of an element of $SL(2,\lip)$:
\begin{equation}
\label{eqn:FirstSpinorSLEmbed}
A = \begin{pmatrix}
 \xi & -\frac{1}{2}\eta^{*-1}\\
 \eta & \frac{1}{2}\xi^{*-1} 
\end{pmatrix}.
\end{equation}
It is easy to check that such an $A$ lies in $SL(2,\lip)$ for all $(\xi,\eta)$. Thus, for two spinors $\kappa_1$, $\kappa_2 \in \lip$ there are matrices $A_1$, $A_2 \in SL(2,\lip)$ with first columns $\kappa_1$ and $\kappa_2$ respectively. Then $A_1.(1,0) = \kappa_1$ and $A_2.(1,0) = \kappa_2$, and $A_2A_1^{-1} \in SL(2,\lip)$ sends $\kappa_1$ to $\kappa_2$.\qed
\\\\
We will make heavy use of the transitivity of the $SL(2,\lip)$ action to extend statements from simple examples to the general case.
\subsection{Products and Forms}
\label{sec:ProductsAndForms}
There are various important binary operations on the complex spinors $\C^2$. The Lipschitz spinors carry generalisations of these.
\subsubsection{Two Inner Products}
In Definition \ref{defn:ParavectorInnerProduct} we defined an inner product (the \textit{dot product}) on paravectors in the obvious way, essentially just the Euclidean dot product. It was also extended to $\clif$ more generally, as $a \cdot b = \text{Re}(a\ol{b})$.
\begin{lemma}
 \label{CliffordAlgInnerProductLemma}
 $\clif$ equipped with $x \cdot y:\ \clif \times \clif \to \R$ is a real inner product space.
\end{lemma}
\noindent\textit{Proof.} We must show i) symmetry, ii) linearity in the first argument, and iii) positive definiteness. For i) we appeal to the invariance of real numbers under $\ol{\cdot}$:
\[x \cdot y = \text{Re}(x\ol{y}) = \ol{\text{Re}(x\ol{y})} = \text{Re}(y\ol{x}) = y \cdot x.\]
For ii), take $a,b \in \R$ and $x,y,z \in \clif$:
\[(ax+by)\cdot z = \text{Re}((ax+by)\ol{z}) = a\text{Re}(x\ol{z}) + b\text{Re}(y\ol{z}) = a(x\cdot z) + b(y \cdot z).\]
So the map is real linear in the first argument, and in the second by symmetry. 

\noindent For iii), let $x = \sum_Ix_Ii_I$, $x_I \in \R$ be the expansion of $x \in \clif \backslash \{0\}$ in the standard basis.
\[x\cdot x = \text{Re}(x\ol{x}) = \sum_Ix_Ii_I\ol{x_Ii_I} = \sum_Ix_I^2i_I\ol{i_I}.\]
If $i_I = i_ji_k...i_z$, then $\ol{i_I} = (-i_z)...(-i_k)(-i_j)$. This implies
\begin{align*}
i_I\ol{i_I} &= i_ji_k...i_z(-i_z)...(-i_k)(-i_j) = 1,\\
x \cdot x &= \sum_Ix_I^2 > 0.\tag*{\qed}
\end{align*}
We can also put an inner product on $\clif^2$, of which the Lipschitz spinors are a subset.
\begin{definition}
\label{defn:CliffSquaredInnerProduct}
Given $D_m = (x_m,y_m) \in \clif^2$ for $m=1,2$, we define
\[(D_1|D_2) = \text{Re}(\ol{x}_1x_2+\ol{y}_1y_2).\]
We also define $|D|^2=(D|D)$.\null\hfill\defn
\end{definition}
\noindent We move the conjugation to the first terms to align with the generalised Hermitian form of Section \ref{sec:HermitianForm}. There, the choice matters; here, $(D_1 | D_2)$ is the same as if we chose Re$(x_1\ol{x}_2+y_1\ol{y}_2)$.
\begin{lemma}
For $x,y \in \clif$, $\text{Re}(x\ol{y}) = \text{Re}(\ol{x}y)$.
\end{lemma}
\noindent\textit{Proof.} Real numbers are invariant under Clifford conjugation, implying $\text{Re}(x\ol{y}) = \ol{\text{Re}(x\ol{y})} = \text{Re}(y\ol{x})$. We then apply Lemma \ref{ABBALemma} to see $\text{Re}(y\ol{x}) = \text{Re}(\ol{x}y)$.\hfill\qed
\\\\
We define a multiplication action of $\clif$ on $\clif^2$ in the obvious way; given $z \in \clif$ and $D=(x,y) \in \clif^2$, then $zD = z(x,y) = (zx,zy)$ and $Dz = (x,y)z = (xz,yz)$.
\begin{lemma}
 \label{CliffordAlgInnerProductLemma2}
 $\clif^2$ equipped with the map $(\cdot|\cdot) :\ \clif^2 \times \clif^2 \to \R$ is a real inner product space.
\end{lemma}
\noindent\textit{Proof.} 
Let $D_j$ have components $(x_m,y_m)$ for $m=1,2,3$. The map is symmetric, as
\[(D_1|D_2) = \text{Re}(\ol{x}_1x_2 + \ol{y}_1y_2) = \ol{\text{Re}(\ol{x}_1x_2 + \ol{y}_1y_2)} = \text{Re}(\ol{x}_2x_1 + \ol{y}_2y_1) = (D_2|D_1).\]
The second equality follows from the invariance of $\R$ under Clifford conjugation. It is real-linear in the first entry; for $a,b \in \R$ and $D_1,D_2,D_3 \in \clif^2$:
\[(aD_1+bD_2|D_3) = \text{Re}((a\ol{x}_1+b\ol{x}_2)x_3 + (a\ol{y}_1+b\ol{y}_2)y_3) = a(D_1|D_3) + b (D_2|D_3).\]
It is positive definite, as for $D = (x,y) \in \clif^2 \neq (0,0)$:
\[(D|D) = \text{Re}(\ol{x}x+\ol{y}y) = \text{Re}(\ol{x}x)+\text{Re}(\ol{y}y) = x\cdot x + y\cdot y > 0,\]
since the inner product on $\clif$ was shown in Lemma \ref{CliffordAlgInnerProductLemma} to be positive definite.\qed
\\\\
In practice, we're primarily interested in taking these forms on $\lip$, or on pairs $\lip^2$. The following lemma collates properties of these two inner products that will prove useful.
\begin{lemma}
 \label{InnerProductPropertiesLemma}
 Let $D_m = (x_m,y_m) \in \clif^2$ for $m = \{1,2\}$, $K = (x,y) \in \lip^2$, $\alpha\in \lip$, $c \in \R$, $b \in \clif$ satisfying $b+\ol{b} \in \R$, and $\pV,\pW \in \para$.
 \begin{enumerate}[label=\roman*)]
 \item $(D_1 \alpha| D_2 \alpha) = |\alpha|^2(D_1|D_2)$,
 \item $(K c|K b) = |K|^2 c \cdot b$,
 \item If $\text{Re}(b)=0$, then $(K|K b)= (K|b K) = 0$,
 \item $(K \pV|K \pW) = |K|^2 \pV \cdot \pW$.
 \end{enumerate}
\end{lemma}
\noindent\textit{Proof.} We can show i) by a simple computation. As $\text{Re}(\ol{x}_1x_2+\ol{y}_1y_2) = \text{Re}(x_1\ol{x}_2+y_1\ol{y}_2)$, we compute
\[(D_1 \alpha|D_2 \alpha) = \text{Re}(x_1\alpha\ol{x_2 \alpha} + y_1 \alpha \ol{y_2 \alpha}) = \text{Re}\left(x_1|\alpha|^2\ol{x_2} + y_1 |\alpha|^2 \ol{y_2}\right) = |\alpha|^2(D_1|D_2).\]
For ii), we can also do a direct computation. Computing the left-hand side:
\[(K c| K b) = c\text{Re}(x\ol{x b} + y\ol{y b}) = c\text{Re}(x\ol{b}\ol{x} + y\ol{b}\ol{y}).\]
As the real component is preserved under Clifford conjugation, we rewrite this as 
\[(K c| K b) = \frac{c}{2}\text{Re}(x(b+\ol{b})\ol{x}+y(b+\ol{b})\ol{y}).\]
Because $b+\ol{b} =2\text{Re}(b) \in \R$ we can simplify this to
\[c\text{Re}(b)\text{Re}(x\ol{x}+y\ol{y}) = c\text{Re}(b)(|x|^2+|y|^2).\]
Noting that $c \cdot b = \text{Re}(c\ol{b}) = c\text{Re}(b)$, we conclude
\[(K c| K b) = (|x|^2+|y|^2)c\text{Re}(b) = |K|^2c\cdot b.\]
For iii), consider first having $b$ on the left:
\[(K|b K) = \text{Re}(x\ol{b x} + y\ol{b y}) = \text{Re}((|x|^2+|y|^2)\ol{b}) = (|x|^2+|y|^2)\text{Re}(\ol{b}) = 0.\]
If instead we have $b$ on the right we can make use of ii):
\[(K| K b) = |K|^2 (1\cdot b) = |K|^2\text{Re}(\ol{b}) = 0.\]
For iv),  we begin with a similar computation to ii).
\[(K\pV|K\pW) = \text{Re}\left(x\pV\ol{\pW}\ol{x} + y\pV\ol{\pW}\ol{y}\right) = \frac{1}{2}\text{Re}\left(x\left(\pV\ol{\pW}+\pW\ol{\pV}\right)\ol{x} +  y\left(\pV\ol{\pW}+\pW\ol{\pV}\right)\ol{y}\right).\]
$\pV\ol{\pW}$ is a product of paravectors and thus has terms of degree $\leq 2$, which implies 
\[\left(\pV\ol{\pW}+\pW\ol{\pV}\right) = \left(\pV\ol{\pW}+\ol{\pV\ol{\pW}}\right) = 2\text{Re}(\pV\ol{\pW}).\]
Taken together, we conclude $(K\pV|K\pW) = \text{Re}(x\ol{x}+y\ol{y})\text{Re}(\pV\ol{\pW}) = |K|^2\pV\cdot\pW$.\qed
\subsubsection{The Generalised Hermitian Form}
\label{sec:HermitianForm}
We can also define a `generalised Hermitian form' that extends the Hermitian inner product of a complex vector space to our Clifford algebras.
\begin{definition}
\label{defn:genHermForm}
Given $D_m = (x_m,y_m) \in \clif^2$ for $m=1,2$, we define
\[\langle \cdot, \cdot \rangle: \clif^2 \times \clif^2 \to \clif,\ \langle D_1,D_2 \rangle = \ol{x_1}x_2+\ol{y_1}y_2.\tag*{\defn}\]
\end{definition}
\noindent Note that $(D_1 | D_2) = \text{Re}(\langle D_1,D_2 \rangle)$. This form is not technically Hermitian as its range is not a subset of $\C$, but it is a straightforward extension where $\C$ is replaced with some higher Clifford algebra and complex conjugation is replaced by the Clifford conjugate $\ol{\cdot}$.
\begin{definition}
 A {\normalfont generalised Hermitian form} on a real vector space $V$ bearing a right $\clifn{p,q}$-module structure is a function $h: V \times V \to \clifn{p,q}$ such that for all $u,v,w \in V$ and all $a,b \in \clifn{p,q}$, 
 \begin{enumerate}[label=\roman*)]
 \item $h(u,va+wb) = h(u,v)a + h(u,w)b$,
 \item $h(u,v) = \ol{h(v,u)}$, where $\ol{\cdot}$ is the Clifford conjugate. \null\hfill\defn
 \end{enumerate}
\end{definition}
\begin{lemma}
 $\langle \cdot,\cdot \rangle: \clif^2 \times \clif^2 \to \clif$ is a generalised Hermitian form.
\end{lemma}
\noindent\textit{Proof.} We show the form is linear in the second argument. Let $D_m = (x_m,y_m) \in \clif^2$ for $m=1,2,3$ and $a,b \in \clif$.
\[\langle D_1,D_2a+D_3b \rangle = (\ol{x_1}x_2+\ol{y_1}y_2)a + (\ol{x_1}x_3 + \ol{y_1}y_3)b = \langle D_1,D_2 \rangle a + \langle D_1,D_3 \rangle b.\]
To be generalised Hermitian, it must be equal to its conjugate.
\[\ol{\langle D_2,D_1 \rangle} = \ol{\ol{x_2}x_1+\ol{y_2}y_1} = \ol{x_1}x_2 + \ol{y_1}y_2 = \langle D_1,D_2 \rangle.\tag*{\qed}\]
This form has some nice properties of its own.
\begin{lemma}
 Given $D_m = (x_m,y_m) \in \clif^2$ for $m = 1,2$, $\alpha \in \lip$, and $b,c \in \clif$,
 \begin{enumerate}[label=\roman*)]
 \item $\langle \alpha D_1, \alpha D_2 \rangle = |\alpha|^2\langle D_1,D_2 \rangle$,
 \item $\langle D_1b, D_2 c \rangle = \ol{b}\langle D_1,D_2 \rangle c$.
 \end{enumerate}
\end{lemma}
\noindent\textit{Proof.} A straightforward computation. For i),
\[\langle \alpha D_1 , \alpha D_2 \rangle = \ol{x_1}(\ol{\alpha}\alpha)x_2 + \ol{y_1}(\ol{\alpha}\alpha)y_2 = |\alpha|^2\langle D_1,D_2 \rangle,\]
where we've used the antiautomorphism of the Clifford conjugate in the first equality. If we instead try right multiplication we find ii):
\[\langle D_1 b, D_2 c \rangle = \ol{b}(\ol{x_1}x_2)c + \ol{b} (\ol{y_1}y_2)c = \ol{b} \langle D_1,D_2\rangle c.\tag*{\qed}\]
The generalised Hermitian form also nicely restricts to a map on Lipschitz spinors.
\begin{lemma}
\label{lem:HermRestrictsLip}
 $\langle \cdot, \cdot \rangle$ restricts to a map $S\lip \times S\lip \to \lip^{\triangleright}$.
\end{lemma}
\noindent\textit{Proof.} Let $\kappa_1 = (\xi_1,\eta_1)$, $\kappa_2 = (\xi_2,\eta_2) \in S\lip$. Assume first that all components are non-zero. Then 
\[\langle \kappa_1,\kappa_2 \rangle = \ol{\xi_1}\xi_2 + \ol{\eta_1}\eta_2 = \ol{\eta_1}(\ol{\pV_1}\xi_2 + \eta_2) = \ol{\eta_1}(\ol{\pV_1} + \pV_2^{-1})\xi_2\]
for the paravectors $\pV_1 = \xi_1\eta_1^{-1}$, $\pV_2= \xi_2\eta_2^{-1}$. If $\ol{\pV_1}+\pV_2^{-1} \neq 0$ then this is the product of 3 elements of $\lip$, which is closed under multiplication. If $\ol{\pV_1}+\pV_2^{-1}=0$ then the form is zero. Both cases imply $\langle \kappa_1,\kappa_2 \rangle \in \lip^{\triangleright}$.

If any component $\xi_1,\xi_2,\eta_1,\eta_2$ is zero, then $\langle \kappa_1,\kappa_2 \rangle$ reduces to a single term which is a product of two elements of $\lip^{\triangleright}$. Monoids are closed under multiplication, so in this case we can also conclude $\langle \kappa_1,\kappa_2 \rangle \in \lip^{\triangleright}$. \qed
\subsubsection{The Bracket}
\label{sec:TheBracket}
The $\C^2$ space of complex spinors has a `natural' symplectic form, given by putting the spinors into a matrix and taking the determinant. We can extend this to our noncommutative Lipschitz spinors using the pseudo-determinant.
\begin{definition}[The Bracket]
 Let $D_m=(x_m,y_m)$ for $m=1,2$. {\normalfont The bracket} $\{\cdot,\cdot\}:\ \clif^2 \times \clif^2 \to \clif$ is given by
 \[\{D_1,D_2\} = \text{pdet} \begin{pmatrix}
x_1 & x_2 \\
y_1 & y_2
\end{pmatrix} = x_1^*y_2 -y_1^*x_2.\tag*{\defn}\]
\end{definition}
\noindent This form is $\clif$-sesquilinear on right-hand multiplication. For $a,b \in \clif$ and $D_1,D_2 \in \clif^2$:
\begin{align*}
\{D_1,D_2(a+b)\} &= \{D_1,D_2\}a + \{D_1,D_2\}b,\\
\{D_1(a+b),D_2\} &= a^*\{D_1,D_2\}+b^*\{D_1,D_2\}.
\end{align*}
The antisymmetry of the symplectic form on $\C^2$ is now modified, and includes the `reverse' involution.
\begin{equation}
\label{eq:BracketAntisymmetry}
 \{D_2,D_1\} = -\{D_1,D_2\}^*.
\end{equation}
Note that, because of this loosened notion of antisymmetry, $\{D,D\}$ is not necessarily zero for general elements $D \in \clif^2$. The equation $x=-x^*$ is solved by, for example, $x=i_1i_2$, so for $D = (1,i_1i_2)$,
\[\{D,D\} = \text{pdet}\begin{pmatrix}
1 & 1 \\
i_1i_2 & i_1i_2 \\
\end{pmatrix} = i_1i_2 - i_2i_1= 2i_1i_2.\]
When we restrict to Lipschitz spinors, we find things to be more nicely behaved.
\begin{lemma}
\label{SpinorFormAntisymmetryLemma}
given $a \in \clif$ and $\kappa = (\xi,\eta) \in S\lip$, $\{\kappa,\kappa a\}=0$.
\end{lemma}
\noindent \textit{Proof.} $\{\kappa,\kappa a\}=0$ is equivalent to $(\xi^*\eta-\eta^*\xi)a=0$. By Corollary \ref{cor:xistarEta} we have $\xi^*\eta = \eta^*\xi$, reducing the above to $0a=0$ which is vacuously true as $a$ is finite. \qed
\\\\
If we put a Lipschitz spinor $\kappa$ in the left-hand entry of $\{\cdot,\cdot\}$, Lemma \ref{SpinorFormAntisymmetryLemma} gives the entire kernel.
\begin{lemma}
\label{KernelLemma}
 The kernel of the map
 \[\{\kappa,\cdot\}:\ \clif^2 \to \clif,\ \kappa \in S\lip\]
 is precisely $\kappa a$ for $a \in \clif$.
\end{lemma} \noindent
\textit{Proof.} Fix a spinor $\kappa = (\xi,\eta) \in S\lip$, and let $D=(x,y) \in \clif^2$ be a prospective element of the kernel. If $\xi=0$, then $\eta \neq 0$ and
\[\{\kappa,D\} = -\eta^*x = 0 \implies x=0.\]
We can rewrite $y=\eta a$ for $a = \eta^{-1}y$, implying $D = (0,\eta a) = \kappa a$. A similar argument holds when $\eta=0$.

Assume, then, that $\xi,\eta \neq 0$. We want to find all $D=(x,y) \in \clif^2$ satisfying
\[\{\kappa,D\} = \xi^*y-\eta^*x = 0 \implies \xi^*y = \eta^*x.\]
Rewrite $x = \xi a$ for $a =\xi^{-1}x$, and similarly $y = \eta b$ for $b =\eta^{-1}y$. As $\xi$ and $\eta$ are invertible, we can rewrite $x$ and $y$ in this form without loss of generality, implying $\xi^*\eta b=\eta^*\xi a$. From Corollary \ref{cor:xistarEta} we have $\xi^*\eta = \eta^*\xi$, and both of these expressions are invertible. Therefore $a = b$ and $D = (\xi a,\eta a) = \kappa a$. This is enough to show both that elements of the kernel take the form $\kappa a$, and the converse. \qed
\\\\
We will mostly be applying the bracket to Lipschitz spinors, and when we do everything restricts to live in $\lip^{\triangleright}$.
\begin{lemma}
\label{lem:BracketRestrictsLip}
 The bracket restricts to a map
 \[\{\cdot,\cdot\}: S\lip \times S\lip \to \lip^{\triangleright}.\]
\end{lemma}
\noindent\textit{Proof.} Consider two Lipschitz spinors $\kappa_1 = (\xi_1,\eta_1)$, $\kappa_2 = (\xi_2,\eta_2) \in S\lip$. Assume first that all the components are non-zero.
\[\xi_1^*\eta_2 - \eta_1^*\xi_2 = \eta_1^*((\eta_1^*)^{-1}\xi_1^*\eta_2 - \xi_2) = \eta_1^*(\pV_1\eta_2 - \xi_2) = \eta_1^*(\pV_1 - \xi_2\eta_2^{-1})\eta_2^{-1} = \eta_1^*(\pV_1 - \pV_2)\eta_2^{-1},\]
for $\pV_1 =\xi_1\eta_1^{-1}$, $\pV_2 = \xi_2\eta_2^{-1} \in \para$, and since all non-zero paravectors are included in $\lip$, this is a product of elements of $\lip^{\triangleright}$ and is thus contained in $\lip^{\triangleright}$. If $\pV_1=\pV_2$ then the bracket maps to $0$, otherwise it maps to $\lip$. 

If any components $\xi_1,\xi_2,\eta_1,\eta_2$ are zero, then $\{\kappa_1,\kappa_2\}$ reduces to a single product of 2 elements of $\lip^{\triangleright}$. As the monoid $\lip^{\triangleright}$ is closed under multiplication, the bracket restricts as claimed. \qed
\\\\
How does our bracket interact with the action of $SL(2,\lip)$?
\begin{lemma}
 \label{SesquimorphismLemma}
Precomposition of both inputs with the action of $SL(2,\lip)$ preserves the bracket operation.
\end{lemma}
\noindent \textit{Proof.} Pick an element of $SL(2,\lip)$:
\[A = \begin{pmatrix}
a & b \\
c & d
\end{pmatrix}\in SL(2,\lip).\]
Act on two elements $D_m=(x_m,y_m) \in \clif^2$ for $m=1,2$, and take their bracket. After simplification, we find
%\begin{align*}
%\{AD_1,AD_2\} %&= (ax_1+by_1)^*(cx_2+dy_2)-(cx_1+dy_1)^*(ax_2+by_2)\\
%&=(x_1^*a^*cx_2-x_1^*c^*ax_2)+(y_1^*b^*dy_2-y_1^*d^*by_2)+(x_1^*a^*dy_2-x_1^*c^*by_2)+(y_1^*b^*cx_2-y_1^*d^*ax_2).
%\end{align*}
%One of the defining properties of $SL(2,\lip)$ is that $ab^*,cd^*,c^*a,d^*b \in \para$, and are therefore invariant under the $\cdot^*$ conjugation; thus the first two terms cancel and we are left with
\[\{AD_1,AD_2\} = x_1^*(a^*d-c^*b)y_2+y_1^*(b^*c-d^*a)x_2.\]
the expressions in brackets are copies (or conjugates) of the pseudo-determinant. 
%\[\{AD_1,AD_2\} = x_1^*\text{pdet}(A)y_2-y_1^*\text{pdet}(A)^*x_2.\]
By definition, pdet$(A)=1$, and this reduces to
\[\{AD_1,AD_2\} = x_1^*y_2-y_1^*x_2 = \{D_1,D_2\}.\tag*{\qed}\]

\subsection{Complementary Elements and a `K\"{a}hler-like' Structure}
For $\kappa\in S\lip$, the conjugate transpose $\kappa^{\dagger}$ satisfies $\kappa^{\dagger}\kappa = |\kappa|^2$. With the addition of the bracket operation, we define a new form of `conjugate'; the \textit{complement} $\check{\kappa}$, satisfying $\{\kappa,\check{\kappa}\} = -|\kappa|^2$.
\begin{definition}[Complementary Structure]
\label{def:Complement}
Given $D = (x,y) \in \clif^2$, we define $\check{D} = (y',-x')$.\null\hfill\defn
\end{definition}
\noindent Restricting again to $S\lip$, the map $\kappa \to \check{\kappa}$ is something like a complex structure in a K\"{a}hler manifold. The complementary element $\check{\kappa}$ is orthogonal to $\kappa$ under the inner product $(\cdot|\cdot)$ on $\clif^2$; in fact, as the following lemma shows, the entire subspace of $\clif^2$ generated by the complement is orthogonal to that generated by $\kappa$.
\begin{lemma}
 \label{InnerProductPropertiesLemma2}
 Let $\kappa = (\xi,\eta) \in S\lip$ and $x,y \in \clif$.
 \begin{enumerate}[label=\roman*)]
 \item $\check{\kappa} \in S\lip$,
 \item $( \kappa x| \check{\kappa} y) = 0$,
 \item $\langle \kappa , \check{\kappa} \rangle = 0$.
 \end{enumerate}
\end{lemma}
\noindent\textit{Proof.} Given $\check{\kappa}$, we write $[\check{\kappa}]_j$ for the $j$th component. For i), as $\xi\ol{\eta} \in \para$ then 
\[[\check{\kappa}]_1\ol{[\check{\kappa}]_2} = -\eta'\ol{\xi'} = -\eta'\xi^* = -(\xi\ol{\eta})^* \in \para,\]
as $\lip$ is closed under the conjugations and negation, as is $\para$. For ii), first assume $\xi\eta^{-1} = \pV \in \para$. Compute the left-hand side:
\[(\kappa x| \check{\kappa} y) = \text{Re}(\xi x \ol{\eta' y} - \eta x \ol{\xi' y}) = \text{Re}(\pV\eta x \ol{y}\eta^* - \eta x \ol{y}\eta^*\pV) = \text{Re}(AB-BA),\]
where $A=\pV=\xi\eta^{-1},\ B=\eta x \ol{y}\eta^*$. Then Lemma \ref{ABBALemma} implies $(\kappa x| \check{\kappa} y) = 0$. If $\xi\eta^{-1} = \infty$ instead, then $\eta = \eta^* = 0$ and $(\kappa x| \check{\kappa} y) = \text{Re}(\xi x \ol{y}\eta^* - \eta x \ol{y}\xi^*) = 0$. For iii), we simply compute $\langle \kappa,\check{\kappa} \rangle = \ol{\xi}\eta' -\ol{\eta}\xi' = (\xi^*\eta)' - (\eta^*\xi)' = 0$.\qed
\\\\
Part ii) of the above lemma suggests there may be a direct sum decomposition of $\clif^2$ in terms of `things generated by $\kappa$' and `things generated by $\check{\kappa}$'.
\begin{corollary}
    \label{cor:Clif2DirectSumDecomp}
    Fix a spinor $\kappa = (\xi,\eta) \in S\lip$. Then
    \[\clif^2 = \kappa \clif \oplus \check{\kappa}\clif.\]
    In particular, 
    \[(a,b) \in \clif^2 = \kappa x + \check{\kappa} y\]
    for
    \[x = \frac{\ol{\xi}a + \ol{\eta} b}{|\kappa|^2},\ y = \frac{\eta^* a - \xi^* b}{|\kappa|^2}.\]
\end{corollary}
\noindent\textit{Proof.} Let $\kappa = (\xi,\eta)$, and let $D = (a,b) \in \clif^2$. We wish to show $D =\kappa x + \check{\kappa}y$ for some $x,y \in \clif$. Expanding this expression in terms of components and solving the simultaneous equations, we find
\[x = \frac{\ol{\xi}a + \ol{\eta} b}{|\kappa|^2},\ y = \frac{\eta^* a - \xi^* b}{|\kappa|^2}.\]
We can check they give $D$ as claimed:
\[\xi x + \eta' y = \frac{|\xi|^2a + \xi\ol{\eta}b + |\eta|^2a - \eta'\xi^*b}{|\kappa|^2} = \frac{|\kappa|^2a + \xi\ol{\eta}b - (\xi\ol{\eta})^*b}{|\kappa|^2} = a,\]
where the last equality follows because $\xi\ol{\eta} \in \para$ is invariant under $\cdot^*$. Similarly $\eta x -\xi' y = b$, so we can write every element of $\clif^2$ in this form, and by Lemma \ref{InnerProductPropertiesLemma2} ii) these subspaces are orthogonal and this expression will be unique. \qed
\\\\
As an extension of the orthogonality of $\kappa$ and $\check{\kappa}$ we have the following technical lemma, which will be of use when constructing null flags.
\begin{lemma}
\label{TauKernelLemma}
 Let $\kappa = (\xi,\eta) \in S\lip$. Then any row vector $\tau = (\alpha,\beta) \in \lip^2\backslash \{0,0\}$ satisfying $\tau\kappa = 0$ must satisfy $\tau\check{\kappa} \neq 0$. Similarly, any row vector $\tau = (\alpha,\beta) \in \lip^2 \backslash \{0,0\}$ satisfying $\tau\check{\kappa} = 0$ must satisfy $\tau\kappa \neq 0$.
\end{lemma}
\noindent\textit{Proof.} Consider the sets $T_{\kappa}$, $T_{\check{\kappa}}$ of row vectors sending $\kappa$ or $\check{\kappa}$ to zero, respectively.
\[T_{\kappa} = \{(\alpha,\beta) \in \lip^2\backslash \{(0,0)\}\ |\ \alpha \xi + \beta \eta = 0\},\]
\[T_{\check{\kappa}} = \{(\alpha,\beta) \in \lip^2\backslash \{(0,0)\}\ |\ \alpha \eta' - \beta \xi' = 0\}.\]
If, say, $\xi=0$, the sets are clearly distinct, because the set $T_{(0,\eta)} = \{(\alpha,\beta) \in \lip^2\backslash \{(0,0)\}\ |\ \beta \eta = 0\}$ consists of $(\alpha,\beta) = (\alpha,0)$, while $T_{\widecheck{(0,\eta)}} = \{(\alpha,\beta) \in \lip^2\backslash \{(0,0)\}\ |\ \alpha \eta' = 0\}$ consists of $(\alpha,\beta) = (0,\beta)$. A similar fact holds if $\eta=0$ (and we cannot have $\xi=\eta=0$), so we assume going forward that all components are non-zero.

If both $\xi,\eta$ are non-zero, then $T_{\kappa}$ consists of elements $(\alpha_1,\beta_1) = \beta_1(-\eta\xi^{-1},1) \in \lip^2\backslash \{(0,0)\}$ parameterised by $\beta_1 \in \lip$ (which excludes zero), while $T_{\check{\kappa}}$ is given by elements $(\alpha_2,\beta_2) = \beta_2(\xi'\eta'^{-1},1) \in \lip^2\backslash \{(0,0)\}$, and is parameterised by $\beta_2 \in \lip$ (which, again, must be non-zero). These sets are nonempty as we can choose $\beta_1 = \beta_2 = 1$. 

Say we found a row vector $\tau$ that is within both sets. Equating the second entries implies $\beta_1=\beta_2$, and equating the first implies 
\begin{equation}
\label{eqn:TauLemmaEquality}
\eta\xi^{-1} = -\xi'\eta'^{-1}.
\end{equation}
As $(\xi,\eta) \in S\lip$ we know $\eta\xi^{-1} = (\xi\eta^{-1})^{-1} = \pV$ is a paravector, and so is $\xi'\eta'^{-1} = (\xi\eta^{-1})' = (\pV^{-1})'$. Applying this to (\ref{eqn:TauLemmaEquality}):
\[\pV = -(\pV^{-1})' = -\pV|\pV|^{-2} \implies 0 = \pV(1+|\pV|^{-2}).\]
Because $|\cdot| \geq 0$ (and therefore $(1+|\pV|^{-2})>0$), we conclude $\eta\xi^{-1}= \pV=0$. But this contradicts our assumption that $\xi,\eta \neq 0$; the sets must be disjoint. \qed
\\\\
The generalised Hermitian product, the bracket, and the map $\check{\cdot}$ together form a `K\"{a}hler-like' structure on $S\lip$; the bracket plays the role of the symplectic form, the generalised Hermitian product that of the Hermitian metric, and the map $\check{\cdot}$ that of the almost complex structure. It isn't actually a K\"{a}hler structure in general; $\langle \cdot, \cdot \rangle$ isn't actually Hermitian, and the bracket (with its added reversal conjugation) isn't quite skew-symmetric. But the analogy is strong, as the following pair of results show.
\begin{proposition}
\label{prop:HermProdIsCheckBracket}
Given elements $D_m = (x_m,y_m) \in \clif^2$ for $m=1,2$, $\langle D_1,D_2\rangle = \{\check{D}_1,D_2 \}$.
\end{proposition}
\noindent\textit{Proof.} $\{\check{D}_1,D_2\} = (\eta_1')^*\eta_2-(-\xi_1')^*\xi_2 = \ol{\xi_1}\xi_2 + \ol{\eta_1}\eta_2 = \langle D_1,D_2 \rangle$. \qed
\\\\
Restricting to the real component, this then implies
\begin{corollary}
Given elements $D_m = (x_m,y_m) \in \clif^2$ for $m=1,2$, $(D_1| D_2) =\text{Re}\left(\{\check{D}_1,D_2\}\right)$.\qed 
\end{corollary}
\noindent The complementary element also provides a nice way to embed $S\lip$ into $SL(2,\lip)$, by the following simple construction.
\begin{proposition}
\label{prop:KappaSLMatrix}
    For any $\kappa = (\xi,\eta) \in S\lip$, 
    \[A_{\kappa} = \begin{pmatrix}
        \kappa & -\frac{\check{\kappa}}{|\kappa|^2}
    \end{pmatrix} = 
    \begin{pmatrix}
        \xi & -\frac{\eta'}{|\xi|^2+|\eta|^2}\\
        \eta & \frac{\xi'}{|\xi|^2+|\eta|^2}
    \end{pmatrix} \in SL(2,\lip).\]
\end{proposition}
\noindent\textit{Proof.} Both $\kappa$ and $\check{\kappa}$ are spinors (see Lemma \ref{InnerProductPropertiesLemma2} i)), which implies their components are elements of $\lip^{\triangleright}$, which is condition i) for Clifford matrices (see Definition \ref{SL2LipDef}). Combined with Lemma \ref{lem:xiStarEtaParavector} (and the property $\xi\ol{\eta} \in \para$ for spinors $(\xi,\eta)$), we also get condition iii). All that's left is condition ii), which requires that the pseudo-determinant be $1$.
\[\text{pdet}(A_{\kappa}) = \xi^*\frac{\xi'}{|\kappa|^2} + \eta^*\frac{\eta'}{|\kappa|^2} = \frac{|\xi|^2}{|\kappa|^2}+\frac{|\eta|^2}{|\kappa|^2}=1.\tag*{\qed}\]
\subsection{Climbing the Clifford Algebra Ladder}
\label{sec:ClifAlgLadder}
The Cayley-Dickson construction is a process for building a sequence of algebras. It begins with the division algebras\footnote{Denoting the real numbers $\R$, the complex numbers $\C$, the quaternions $\hyp$, and the octonions $\mathbb{O}$.} $\R, \C, \hyp, \mathbb{O}$ over the reals, and is used to prove various theorems about these algebras. 

With a slight tweak, we can apply the method to instead construct the Clifford algebras. The following construction is a variation\footnote{There's more than one way to embed pairs of elements of $\clifn{0,n}$ into $\clifn{0,n+1}$; we've chosen to vary Lounesto's formulation to play nicely with right-multiplication in the other structures we define.} on that found in \cite{Lounesto_2001}, Chapter 21.
\begin{lemma}[Iterative Clifford Algebras]
\label{lem:LounestoCayleyDickson}
 Consider pairs $(u,v) \in \clif^2$. Then $\clifn{0,n+1}$ is isomorphic to the algebra of such pairs equipped with the product
 \[(u_1,v_1)\cdot(u_2,v_2) = (u_1u_2-v_1'v_2, u_1'v_2+v_1u_2).\]
\end{lemma}
\noindent\textit{Proof.} Every element of $\clifn{0,n+1}$ can be uniquely written as $u+i_{n+1}v$ for $u,v \in \clif$, where $i_{n+1}v$ groups all terms containing the generator $i_{n+1}$ and $u$ contains all other terms. We can take a product of two such elements:
\[(u_1+i_{n+1}v_1)(u_2+i_{n+1}v_2) = u_1u_2+i_{n+1}v_1u_2+u_1i_{n+1}v_2+i_{n+1}v_1i_{n+1}v_2,\]
and, because $u,\ v$ do not contain $i_{n+1}$, commuting it past them has the effect of sending $i_j \to -i_j$ for $j \in \{1,2,...,n\}$ as it passes. So we can write
\[(u_1+i_{n+1}v_1)(u_2+i_{n+1}v_2) = (u_1u_2-v_1'v_2)+i_{n+1}(v_1u_2+u_1'v_2),\]
which confirms the claim. \qed
\\\\
Given some $z=(u,v)$ as an element of $\clifn{0,n+1}$, the conjugate $\ol{z}$ has components $(\ol{u},-v^*)$; the `imaginary' coefficient has a different conjugation as the generator $i_{n+1}$ has to commute to the left-hand side after applying the anti-involution. This gives us a different way to view $\clif^2$, as an embedding in a higher Clifford algebra; and it dovetails nicely with the Lipschitz spinor structure.
\begin{lemma}
\label{lem:EmbedCliffordProduct}
 Given a pair of Lipschitz spinors $\kappa_m = (\xi_m,\eta_m) \in S\lip$ for $m=1,2$, if we write $z_m = \xi_m + i_{n+1}\eta_m \in \clifn{0,n+1}$ then $\ol{z_1}z_2 = \langle \kappa_1,\kappa_2 \rangle + i_{n+1}\{\kappa_1,\kappa_2\}$.
\end{lemma}
\noindent\textit{Proof.} We apply the Cayley-Dickson type construction of Lemma \ref{lem:LounestoCayleyDickson}.
\begin{align*}
\ol{z_1}z_2 %= (\ol{\xi_1},-\eta_1^*)\cdot(\xi_2,\eta_2) 
&= \left(\ol{\xi_1}\xi_2 - (-\eta_1^*)'\eta_2,\ol{\xi_1}'\eta_2+(-\eta_1^*)\xi_2\right)
%= (\ol{\xi_1}\xi_2+\ol{\eta_1}\eta_2,\xi_1^*\eta_2 - \eta_1^*\xi_2) 
= \langle \kappa_1,\kappa_2 \rangle + i_{n+1}\{\kappa_1,\kappa_2\}.\tag*{\qed}
\end{align*}
We see the two forms naturally arise in the embedding of Lipschitz spinors $S\lip$ into $\clifn{0,n+1}$. And, in fact, they are the components of a spinor themselves.
\begin{lemma}
\label{lem:ProductRatioSpinor}
 Let $\kappa_1 = (\xi_1,\eta_1)$, $\kappa_2 = (\xi_2,\eta_2) \in S\lip$. Then $\left(\langle \kappa_1,\kappa_2\rangle, \{\kappa_1,\kappa_2\}\right)$ is a Lipschitz spinor.
\end{lemma}
\noindent\textit{Proof.} Both the forms are elements of $\lip^{\triangleright}$ (see Lemma \ref{lem:HermRestrictsLip} and Lemma \ref{lem:BracketRestrictsLip}), so all we need to check is that $\langle \kappa_1,\kappa_2 \rangle \ol{\{\kappa_1,\kappa_2\}} \in \para$.
\begin{align*}
(\ol{\xi_1}\xi_2 + \ol{\eta_1}\eta_2)\ol{(\xi_1^*\eta_2-\eta_1^*\xi_2)} &= %(\ol{\xi_1}\xi_2 + \ol{\eta_1}\eta_2)(\ol{\eta_2}\xi_1' - \ol{\xi_2}\eta_1')\\
%&= 
\ol{\xi_1}\xi_2\ol{\eta_2}\xi_1' - \ol{\xi_1}\xi_2\ol{\xi_2}\eta_1' + \ol{\eta_1}\eta_2\ol{\eta_2}\xi_1' - \ol{\eta_1}\eta_2\ol{\xi_2}\eta_1'.
\end{align*}
We substitute $\pV_j=\xi_j\ol{\eta_j}$, and $\pW_j=\xi_j^*\eta_j$.
\begin{align}
    \label{eqn:KahlerRatioExpanded}
    \langle \kappa_1,\kappa_2 \rangle \ol{\{\kappa_1,\kappa_2\}} %&= \ol{\xi_1}\pV_2\xi_1' - \ol{\eta_1}\ \ol{\pV_2}\eta_1' + |\eta_2|^2 (\eta_1^*\xi_1)'- |\xi_2|^2(\xi_1^*\eta_1)'\\
    &= \sigma(\ol{\xi_1})(\pV_2) - \sigma(\ol{\eta_1})(\ol{\pV_2}) + (|\eta_2|^2-|\xi_2|^2)\pW_1'.
\end{align}
All these components are paravectors, so $\langle \kappa_1,\kappa_2 \rangle \ol{\{\kappa_1,\kappa_2\}} \in \para$ and the claim follows. \qed
\subsection{The Tangent Bundle of $S\lip$}
As $\clif^2$ is a vector space, its tangent space at any point is naturally isomorphic to $\clif^2$. $S\lip$ is a submanifold of $\clif^2$, and thus the tangent spaces of $S\lip$ at varying points are isomorphic to varying subspaces of $\clif^2$. That is, we can write tangent vectors to $S\lip$ as pairs of elements in the Clifford algebra. 
\begin{theorem}
\label{thm:SpinorsTangentSpace}
 The tangent space $T_{\kappa}S\lip$ at $\kappa = (\xi,\eta) \in S\lip$ can be characterised as
 \begin{equation}
 \label{SpinorTangentChar1}
T_{\kappa}S\lip = \{(a,b) \in \clif^2\ |\ a \ol{\eta} + \xi \ol{b} \in \para\}.
\end{equation}
 It decomposes into real orthogonal subspaces (under $(\cdot | \cdot )$) as
\[T_{\kappa}S\lip = \kappa\R \oplus \kappa\left(\R^{0,n} \oplus {\bigwedge}^{\!2} \R^{0,n} \right) \oplus \check{\kappa}\left(\para\right).\]
 It can be explicitly constructed as a subspace of $\clif^2$; if $\eta \neq 0$, then
 \[T_{\kappa}S\lip = \{(\eta'\pV + \xi H, \eta H)\ |\ \pV \in \para,\ H \in \$\mathfrak{g}_{n+1}\},\]
 and if $\xi \neq 0$ then
 \[T_{\kappa}S\lip = \{(\xi\Xi,\xi'\pU + \eta\Xi)\ |\ \pU \in \para,\ \Xi \in \$\mathfrak{g}_{n+1}\}.\]
\end{theorem}
\noindent \textit{Proof.} Points $\kappa=(\xi,\eta) \in S\lip$ satisfy $\xi\ol{\eta} \in \para$. For the first characterisation, we take a small variation and differentiate. For some $a,b \in \clif^2$ we have:
\[(\xi+ta)\ol{(\eta+tb)} \in \para \implies a\ol{\eta}+\xi\ol{b} \in T_{\xi\ol{\eta}}\para \cong \para.\]
%\implies \xi\ol{\eta}+t(a\ol{\eta}+\xi\ol{b})+t^2a\ol{b} \in &\para\\
This implies (\ref{SpinorTangentChar1}). Consider the subset $\kappa\mathcal{B}^n \in \clif^2$ (recalling $\mathcal{B}^n$ are the bi-paravectors). We check that they satisfy (\ref{SpinorTangentChar1}); for $(\xi B,\eta B) \in \kappa\mathcal{B}^n$:
\[a \ol{\eta} + \xi \ol{b} = \xi B \ol{\eta} + \xi \ol{B}\ol{\eta} = \xi(B+\ol{B})\ol{\eta}.\]
Since bi-paravectors only have terms of degree 2 or less, $B+\ol{B}=2\text{Re}(B)$.
\[a \ol{\eta} + \xi \ol{b} = 2\text{Re}(B)\xi\ol{\eta} \in \para.\]
So $\kappa\mathcal{B}^n$ is contained in $T_{\kappa}S\lip$. The space $\mathcal{B}^n$ can be decomposed into `real' and `non-real' components (`non-real' components are anything of degree 1 or higher):
\[\kappa \mathcal{B}^n = \kappa\R \oplus \kappa\left(\R^{0,n} \oplus {\bigwedge}^{\!2} \R^{0,n} \right),\]
and, because purely imaginary bi-paravectors $B$ satisfy $B+\ol{B} = 0$, we appeal to Lemma \ref{InnerProductPropertiesLemma} iii) to see that the subspaces $\kappa\R$ and $ \kappa\left(\R^{0,n} \oplus {\bigwedge}^{\!2} \R^{0,n} \right)$ are orthogonal. 

Now consider the subspace $\check{\kappa}(\para)$. We check it sits within the tangent space $T_{\kappa}S\lip$ using (\ref{SpinorTangentChar1}):
\[\eta'\pV\ol{\eta}-\xi\ol{\xi'\pV} = \sigma(\eta')(\pV)-\sigma(\xi)(\ol{\pV}) \in \para.\]
It follows from Lemma \ref{InnerProductPropertiesLemma2} ii) that this subspace is orthogonal to $\kappa\mathcal{B}^n$ under $(\cdot | \cdot)$. The total dimension of $\kappa \mathcal{B}^n \oplus \check{\kappa} (\para)$ matches the dimension of $T_{\kappa}S\lip$, so we conclude they are isomorphic and the tangent space decomposes as
 \[T_{\kappa}S\lip = \kappa\R \oplus \kappa\left(\R^{0,n} \oplus {\bigwedge}^{\!2} \R^{0,n} \right) \oplus \check{\kappa}\left(\para\right).\]
For the latter two characterisations we start instead with the pair of conditions
\begin{align}
 \label{eqn:XiProjCond}\xi &= \pV\eta,\ \pV \in \para,\\
 \label{eqn:EtaProjCond}\eta &= \pW\xi,\ \pW \in \para.
\end{align}
The first case applies when $\eta \neq 0$, and the second when $\xi \neq 0$. 

Consider a curve $\kappa_t=(\xi_t,\eta_t)$ in $S\lip$  for $t \in (-\varepsilon,\varepsilon)$ that passes through the point $(\xi_0,\eta_0) = (\xi,\eta) \in S\lip$ at $t=0$. Define $\pV_t = \xi_t\eta^{-1}_t$ (restricting $\varepsilon$ as necessary), as $\eta_t \neq 0$ over a small interval around $t=0$ by continuity. As this curve lies within the manifold of Lipschitz spinors, it satisfies the pair of conditions (\ref{eqn:XiProjCond}), (\ref{eqn:EtaProjCond}). 

Deriving (\ref{eqn:XiProjCond}) gives a condition on the derivative of $\xi$:
\[\dot{\xi_t} = \dot{\pV}_t\eta_t + \pV_t\dot{\eta}_t = \dot{\pV}_t\eta_t + \xi_t\eta^{-1}_t\dot{\eta}_t.\]
At $t=0$, $(\xi_t,\eta_t) = (\xi,\eta)$ and the tangent space will then be $(\dot{\pV}_0\eta + \xi\eta^{-1}\dot{\eta}_0,\dot{\eta}_0)$ for arbitrary choices of $(\dot{\pV}_0,\dot{\eta}_0) \in \left(\para\right) \times T_{\eta}\lip$. If we instead take $\xi \neq 0$ and apply the same arguments to (\ref{eqn:EtaProjCond}), the tangent vectors take the form $(\dot{\xi}_0,\dot{\pW}_0\xi + \eta\xi^{-1}\dot{\xi}_0)$ for arbitrary choices of $(\dot{\pW}_0,\dot{\xi}_0) \in \left(\para\right) \times T_{\xi}\lip$. Renaming the arbitrary derivatives, these arguments imply that tangent vectors can be written in the form
\begin{equation}
\label{eqn:EtaNonzeroTangentSpace}
(\pV\eta + \xi\eta^{-1} G, G) \subset \clif^2,
\end{equation}
for $\eta \neq 0,\ \pV \in \para,\ G \in T_{\eta}\lip$, and
\begin{equation}
\label{eqn:XiNonzeroTangentSpace}
(N,\pU\xi + \eta\xi^{-1}N) \subset \clif^2,
\end{equation}
for $\xi \neq 0,\ \pU \in \para,\ N \in T_{\xi}\lip$. 
 
Because $G = \eta H$ for $H \in \$\mathfrak{g}_{n+1}$ and $N = \xi \Xi$ for $\Xi \in \$\mathfrak{g}_{n+1}$, by Corollary \ref{cor:TangentSpacesLip} we find the form given in the theorem statement. The above arguments imply $T_{\kappa}S\lip$ is a subset of the vector spaces (\ref{eqn:EtaNonzeroTangentSpace}) and (\ref{eqn:XiNonzeroTangentSpace}) wherever those spaces are defined, but as the dimensions agree these spaces characterise the tangent space fully.\qed 
\\\\
The tangent vectors in the kernel of $\{\kappa,\cdot\}$ will prove useful for a later theorem, so we provide the following corollary to the larger characterisation of tangent spaces.
\begin{corollary}
 \label{TangentSpinorLemma}
 Given $\kappa \in S\lip$, if $a \in \clif$ satisfies $\kappa a \in T_{\kappa}S\lip$ then $a \in \mathcal{B}^n$ and $a+\ol{a} = 2\text{Re}(a) \in \R$.
\end{corollary}
\noindent\textit{Proof.} Lemma \ref{InnerProductPropertiesLemma2} ii) tells us the subspaces generated by multiples of $\kappa$ and $\check{\kappa}$ in $\clif^2$ are orthogonal under $(\cdot | \cdot)$; so, if $\kappa a \in T_{\kappa}S\lip$, it cannot be within the $\check{\kappa}$ summand of the orthogonal decomposition of Theorem \ref{thm:SpinorsTangentSpace}, and must lie within the $\kappa\mathcal{B}^n$ summand. This implies $a$ is a bi-paravector, which can be expressed as the sum of monomials of degree at most 2. Elements of degree $1$ and $2$ are sent to their negative by Clifford conjugation, so $a+\ol{a}$ is real. \qed
\subsubsection{Sections of the Tangent Bundle}
\label{Sec:SectionsTangentBundle}
Following \cite{Mathews_Spinors_horospheres}, we define the scaling factor of a conformal map.
\begin{definition}
    Let $V,W$ be real vector spaces equipped with nondegenerate bilinear symmetric real-valued forms  $(\cdot,\cdot)_V$ and $(\cdot,\cdot)_W$.   Let $f:V \to W$ be a conformal linear map. The real constant $K$ such that for all $V_1,V_2 \in V$,
    \[(f(V_1),f(V_2))_W = K(V_1,V_2)_V,\]
    is the {\normalfont scaling factor} of $f$. \null\hfill\defn
\end{definition}
\noindent \label{not:TSLip}We define a family of sections of $TS\lip$, parameterised by paravectors.
\begin{definition}
 For any $\pV \in \para$, the section $s_{\pV}: S\lip \to TS\lip$ of $TS\lip$ is $s_\pV(\kappa) = \check{\kappa}\pV$.
 More generally, the family of such sections forms a map
 \[s: \para \times S\lip \to TS\lip,\ s(\pV,\kappa) = s_\pV(\kappa) = \check{\kappa}\pV.\tag*{\defn}\]
\end{definition}
\noindent These sections behave nicely under the bracket, as $\{\kappa,s_\pV\kappa\} = \{\kappa,\check{\kappa}\pV\} = -|\kappa|^2\pV$. They are real linear; for $a,b \in \R$ and $\pV,\pW \in \para$,
\[s_{a\pV+b\pW}(\kappa) = \check{\kappa}(a\pV+b\pW)=a\check{\kappa}\pV+b\check{\kappa}\pW = as_\pV(\kappa)+bs_\pW(\kappa).\]
As $\check{\kappa}(\para)$ is a real-linear subspace of $\clif^2$, it has an inner product and norm induced by the inner product $(\cdot|\cdot)$ on $\clif^2$. That is, the inner product $(\cdot|\cdot)$ can be applied to this family of sections of $TS\lip$.
\begin{lemma}
 \label{KappaCheckSectionsLemma}
 For any $\kappa \in S\lip$, the map $s(\cdot,\kappa)$ is conformal with scaling factor $|\kappa|^2$. 
\end{lemma}
\noindent\textit{Proof.} Since $s(\cdot,\kappa)$ sends $\pV \in \para$ to $\check{\kappa}\pV$, for all $\pV,\pW \in \para$, $(\check{\kappa}\pV|\check{\kappa}\pW) = |\check{\kappa}|^2\pV\cdot \pW$ by Lemma \ref{InnerProductPropertiesLemma} iv). We also have that $|\check{\kappa}| = |\kappa|$, implying the claimed result. \qed
\\\\
Later, we will consider the isomorphic copy of $\para$ within the decomposition of $T_{\kappa}S\lip$ by quotienting out the other summands to get $\para \cong T_{\kappa}S\lip/\kappa\mathcal{B}^n$. The points in this quotient space are the affine $\left(\tfrac{n^2+n}{2}+1\right)$-planes in $T_{\kappa}S\lip$ of the form $\check{\kappa}\pV+\kappa\mathcal{B}^n$ for $\pV \in \para$. The following characterisation will then prove useful, identifying the set of points forming these affine planes using the bracket.
\begin{lemma}
\label{lem:QuotientBracketChar}
 Let $\kappa \in S\lip$, $\pV \in \para$, and $w \in T_{\kappa}S\lip$. Then $w \in s_\pV(\kappa) + \kappa\mathcal{B}^n$ if and only if $\{\kappa,w\} = -\pV|\kappa|^2$.
\end{lemma}
\noindent\textit{Proof.} We have that $\{\kappa,s_\pV(\kappa)\} = -\pV|\kappa|^2$. If $w \in s_{\pV}(\kappa)+\kappa\mathcal{B}^n$ with representative $s_{\pV}(\kappa)+\kappa B$, then applying properties of the bracket we find:
\[\{\kappa,w\} = \{\kappa, s_\pV(\kappa)\} + \{\kappa,\kappa B\} = \{\kappa,s_\pV\kappa\} = -\pV|\kappa|^2.\] 
The second equality holds for all $B \in \mathcal{B}^n$ by Lemma \ref{SpinorFormAntisymmetryLemma}. 

Conversely, if $\{\kappa,w\} = -\pV|\kappa|^2$ then $\{\kappa,w-s_\pV(\kappa)\} = 0$ by linearity and by Lemma \ref{KernelLemma} $w=s_\pV(\kappa)+\kappa x$ for some $x \in \clif$. But if $w \in TS\lip$ then $\kappa x \in \kappa\mathcal{B}^n$, by Corollary \ref{TangentSpinorLemma}. \qed
\\\\
We can introduce an inner product on the quotient as well, following Definition 3.11.4 of \cite{mathews2024quaternions}. As noted in that work, there are several ways to introduce an inner product, but as the direct sum decomposition is orthogonal and all spaces involved are positive definite they all agree.
\begin{definition}
\label{defn:InnerProductOnQuotient}
 Let $\kappa \in S\lip$. We define a positive definite inner product
 \[(\cdot|\cdot): T_{\kappa}S\lip/\kappa\mathcal{B}^n \times T_{\kappa}S\lip/\kappa\mathcal{B}^n \to \R\]
 as follows. Each element of the quotient is of the form $\check{\kappa}\pV+\kappa\mathcal{B}^n$ for a unique $\pV \in \para$; for $\pV,\pW \in \para$ we define
 \[\left(\check{\kappa}\pV+\kappa\mathcal{B}^n|\check{\kappa}\pW+\kappa\mathcal{B}^n\right) = ( \check{\kappa}\pV|\check{\kappa}\pW),\]
 where $(\cdot|\cdot)$ is the inner product\footnote{It is because of this close association to the inner product on $\clif^2$ that we use the same notation for each.} on $\clif^2$.\null\hfill\defn
\end{definition}
\noindent We see this inner product is positive definite since Lemma \ref{InnerProductPropertiesLemma} iv) tells us $(\check{\kappa}\pV|\check{\kappa}\pW) = |\check{\kappa}|^2\pV \cdot \pW$, and the dot product on $\para$ is positive-definite.

\section{Light-Cones, Multiflags, and Decorated Ideal Points}
\label{sec:Second}
Our first geometric realisation of the Lipschitz spinors $S\lip$ follows the path established by Penrose and Rindler \cite{SpinorsAndSpacetime}, relating the spinor to a null flag structure in Minkowski space. \label{not:GenMinkSpace}We'll begin with some definitions and facts about the \textit{generalised Minkowski spaces} $\R^{1,n+2}$. 
\subsection{Geometry of Generalised Minkowski Space}
We may write points in $\R^{1,n+2}$ either as a vector\footnote{Unless otherwise specified this will be intended ordering of the variables. T and Z are separated by a semicolon to remind us that the other space variables are collected in a paravector $\pX$; essentially, mimicking the matrix presentation.}
\begin{equation}
\label{OrderOfVariables}
 (T,Z;X_0,X_1,...,X_n)_{\R} = (T,Z;\pX)_{\R}
\end{equation}
\noindent with quadratic form $dT^2-dZ^2-...-dX_n^2$, or as a \textit{paravector Hermitian matrix}.
\begin{definition}[Paravector Hermitian Matrices]
\label{defn:ParavectorHermitianMatrices}
 The paravector Hermitian matrices $\mathcal{P}^{1,n+2}$ are the matrices in $\mathcal{M}(2,\lip)$ of the form\footnote{The factor of $\tfrac{1}{2}$ ensures the trace is precisely $T$.}
\[\frac{1}{2}\begin{pmatrix}
 T+Z & \pX \\
 \ol{\pX} & T-Z 
\end{pmatrix}\]
for $T,Z \in \R$ and $\pX = X_0+\sum_{j=1}^nX_ji_j \in \para$ with $X_j \in \R$ for $j \in \{0,1,...,n\}$. They provide an embedding of $\R^{1,n+2}$ into the space $\mathcal{M}(2,\para)$ of $2\times 2$ matrices with $\para$ entries, generalising the standard embedding of $\R^{1,3}$ as the Hermitian matrices. Just like standard Hermitian matrices, they satisfy the defining property
\[P=P^{\dagger},\ P \in \mathcal{P}^{1,n+2},\]
where $P^{\dagger}=(\ol{P})^T$ is the conjugate transpose.\hfill\defn
\end{definition}
\noindent $\R^{1,n+2}$ and $\mathcal{P}^{1,n+2}$ are isomorphic vector spaces; we default to writing $\R^{1,n+2}$ for both, unless we want to draw attention to the matrix form. \label{not:MinkMetric}Denote the generalised Minkowski metric $dT^2-dZ^2-...-dX_n^2$ by $( \cdot | \cdot)_{1,n+2}$, and the norm by $|p|^2=(p | p)_{1,n+2}$ for $p \in \R^{1,n+2}$. 

\label{not:LightCone}The set of points $\{p \in \R^{1,n+2}:\ (p|p)_{1,n+2} = 0\}$ forms the \textit{light-cone} $L$, while the restricted set with positive $T$ coordinate form the positive light-cone $L^+ = \{p = (T,Z;\textbf{X}) \in \R^{1,n+2}:\ |p|^2 = 0 \text{ and } T>0\}$.

\label{not:ParaMatrixPt}For a point $p=(T,Z;\pX) \in \R^{1,n+2}$ corresponding to a paravector Hermitian matrix $S_p \in \mathcal{P}^{1,n+2}$, we have the following properties:
\[\text{Tr }S_p = T,\quad 4\,\text{pdet }S_p = |p|^2.\]
\label{not:CelestialSphere}We define the \textit{celestial sphere}, or the \textit{sphere at infinity}, to be the quotient of the positive light-cone $L^+$ by the relation $p \sim k p, k \in \R^+$. \label{not:ell}That is, the celestial sphere is the space of light-like directions (or rays) $\ell$ from the origin with positive $T$ coordinate. Denote the celestial sphere by $\mathscr{S}^+$. It can be modelled concretely by taking a slice of $L^+$ at a fixed time $T$, as each point of the slice will be a representative of an equivalence class in $L^+/{\sim}$; we call such a model the \textit{celestial sphere at constant $T$}, and denote it by $\mathscr{S}^+_T$. The celestial sphere and the models at constant $T$ are naturally diffeomorphic, and both are diffeomorphic to the hypersphere $S^{n+1}$.
\subsubsection{The Tangent Spaces to Light-Cones}
Penrose and Rindler's null flag lives in the tangent bundle to the light-cone. To extend it, we need a description of the tangent space to $L^+$ in $\R^{1,n+2}$.
\begin{lemma}
\label{SphereTangentSpaceLemma}
 In a quadratic vector space $V$ with metric $( \cdot,\cdot)$, let $L_K$ be the manifold of points $x \in V$ satisfying $( x,x ) = K$ for fixed $K\in \R$. At a point $p \in L_K$ where $L_k$ is locally a smooth manifold, the tangent space is $p^{\perp}$, the space of vectors in $V$ orthogonal to $p$.
\end{lemma}
\noindent\textit{Proof.} Let $N_p \subset L_K$ be a neighbourhood of $p$. At $p$, all tangent vectors can be constructed as the derivative of some smooth curve $p_t \subset N_p \subset L_K$, $t \in (-\varepsilon,\varepsilon)$ for some $\varepsilon>0$ satisfying $p_0=p$. Because $p_t \in L_K$ it satisfies $(p_t,p_t) = K$, and we differentiate this relation to find $2( p_t,\dot{p}_t ) = 0$, which, at $t=0$, implies the tangent vector $\dot{p}_0$ is orthogonal to the position vector $p_0=p$, implying $T_pL_K \subseteq p^{\perp}$. But the dimensions agree, so $T_pL_K = p^{\perp}$. \qed
\\\\
Accordingly, we have $T_pL^+ = p^{\perp}$. Following \cite{mathews2024quaternions}, we parlay this into a description of the tangent space of the celestial sphere by taking a quotient. If $p \in L^+$ has $T$ coordinate $T_0$ and $\ell = p\R$ is the associated ray in $L^+$, \label{not:projectivisation}projectivisation (that is, taking the space of 1 dimensional vector subspaces) yields an isomorphism $T_{\ell}\mathscr{S}^+ \cong T_p\mathscr{S}_{T_0}^+$. The light-cone has product structure $L^+\cong S^{n+1} \times \R$, and accordingly we get an orthogonal decomposition of the tangent space:
\[T_pL^+ = p^{\perp} = \ell^{\perp} = p\R \oplus T_p\mathscr{S}_{T_0}^+ = \ell \oplus T_p\mathscr{S}_{T_0}^+,\]
and a corresponding description of the tangent space of the celestial sphere:
\[T_p\mathscr{S}_T^+ \cong T_{\ell}\mathscr{S}^+ \cong \frac{T_pL^+}{p\R} = \frac{p^{\perp}}{p\R} = \frac{\ell^{\perp}}{\ell}.\]
\subsubsection{Orientations in Minkowski Space}
\label{sec:OrientationConventions}
We use the following orientation conventions (which are a variation on those in \cite{mathews2024quaternions}; note that here we append vectors to the \textit{beginning} of bases instead of the end):
\begin{itemize}
 \item \textit{Orientations of codimension-1 subspaces.} Given a codimension-1 subspace $W$ of an oriented vector space $V$, a transverse vector $v$ to $W$ induces an orientation on $W$; an ordered basis $B$ for $W$ is positively oriented iff appending $v$ to the start of $B$ yields a positively oriented basis for $V$.
 \item \textit{Orientations of codimension-1 quotient spaces.} Given a 1-dimensional subspace $W=\R w$ of an oriented vector space $V$, we orient $V/W$ as follows. An ordered basis $B$ for $V/W$ is positively oriented if and only if taking $w$ followed by an ordered set of coset representatives from $B$ yields a positively oriented basis for $V$.
\end{itemize}
In $\R^{1,n+2}$ we have unit (norm $\pm 1$) vector fields $\partial T,\partial Z, \partial X_0,...,\partial X_n$ in the coordinate directions. We give $\R^{1,n+2}$ an orientation by declaring this ordered basis to be positively oriented. We also define two radial vector fields
\[\partial_r = T\partial T + Z \partial_Z + ... + X_n \partial X_n,\ \partial_{-} = Z\partial_Z + X_0\partial X_0 + ... + X_n \partial X_n.\]
The field $\partial_r$ is non-zero everywhere in $\R^{1,n+2}$ except at $p=0$, and $\partial_{-}$ is non-zero except along the $T$-axis. The field $\partial_{-}$ is also space-like, pointing along $(n+2)$-planes with constant $T$. On each space-like $(n+2)$-plane $\Pi_{T_0}$ (defined by $T=T_0$ for $T_0 \in \R$), the vector field $\partial_T$  is normal to $\Pi_{T_0}$. Following the conventions above we endow $\Pi_{T_0}$ with the induced orientation, which will be the same as coordinatising the plane as $\R^{n+2}$ with coordinates $(Z,X_0,...,X_n)$.

We define two orientations, the \textit{outward} and \textit{inward} orientations, on $\mathscr{S}_{T_0}^+$, the celestial sphere at constant $T_0$. This is a codimension 1 submanifold of $\Pi_{T_0}$ with outwards normal vector field $\partial_{-}$. The outward orientation is the one induced by $\partial_{-}$, and the inward orientation is induced by $-\partial_{-}$. Any basis of a tangent space to $\mathscr{S}^+$ is, accordingly, either oriented inward or outward. Explicitly, at $p \in \mathscr{S}_{T_0}^+$, an ordered basis $(B_1,...,B_{n+1})$ of $T_p\mathscr{S}_{T_0}^+$ has outward orientation if and only if $(\partial_{-},B_1,...,B_{n+1})$ is a positively oriented basis of $\Pi_{T_0}$, if and only if $(\partial_T,\partial_{-},B_1,...,B_{n+1})$ is a positively oriented basis of $\R^{1,n+2}$.

But there is another way to define an orientation on the tangent space. $L^+$ obtains an orientation as a codimension-1 submanifold of $\R^{1,n+2}$, with the transverse vector field $\partial_{-}$ pointing out of the solid cone of time-like vectors. At any $p \in L^+$, the $(n+2)$-plane $T_pL^+ = p^{\perp}$ also obtains an orientation. The quotient $p^{\perp}/p\R = T_\ell\mathscr{S}^+$ then obtains an orientation as a codimension-1 quotient. With this orientation, an ordered basis $(\underline{B}_1,...,\underline{B}_{n+1})$ of $p^{\perp}/p\R$ is positively oriented if and only if $(\partial_r,B_1,...,B_{n+1})$ forms a positively oriented basis of $p^{\perp}$ (where each $B_m$ is a representative of $\underline{B}_m$), if and only if $(\partial_{-},\partial_r,B_1,...,B_{n+1})$ forms a positively oriented basis of $\R^{1,n+2}$. This corresponds to the inward orientation defined above, because the swap $(\partial_T,\partial_{-}) \to (\partial_{-},\partial_r)$ is orientation reversing.

\subsection{The Basepoint Map}
In \cite{Mathews_Spinors_horospheres}, the relationship between Hermitian matrices and $\R^{1,3}$ is exploited to map the complex spinors into Minkowski space. Given a complex spinor $(\xi,\eta) \in \C^2$,
\[\begin{pmatrix}
 \xi \\
 \eta
\end{pmatrix}
\begin{pmatrix}
 \ol{\xi} & \ol{\eta}
\end{pmatrix}
= \begin{pmatrix}
 |\xi|^2 & \xi\ol{\eta}\\
 \eta\ol{\xi} & |\eta|^2
\end{pmatrix}
= \frac{1}{2}
\begin{pmatrix}
 T+Z & X+Yi \\
 X-Yi & T-Z
\end{pmatrix}.\]
The determinant of the outer product is zero, and the right-hand side has determinant proportional to the Minkowski metric, so this maps complex spinors to the light-cone $L$ in $\R^{1,3}$. Because the trace satisfies $|\xi|^2+|\eta|^2 = T>0$, it in fact maps complex spinors to $L^+$.

We have the Clifford conjugation on Lipschitz spinors, so generalisation of this map is straightforward. Given $\kappa \in S\lip$, we think of mapping it into $\R^{1,n+2}$ by thinking of $\kappa$ as a column vector and taking an outer product with its conjugate transpose to get a paravector Hermitian matrix. We recall the conjugate transpose $\ol{\kappa}^T$ of a column spinor $\kappa= (\xi,\eta)^T$ is notated $\kappa^{\dagger}$.
\begin{definition}
\label{defn:BasepointMap}
The map $\phi_1: S\lip \to \R^{1,n+2}$ is given by
\[\phi_1(\kappa) = \kappa\kappa^{\dagger} = \begin{pmatrix}
 \xi \\
 \eta
\end{pmatrix}
\begin{pmatrix}
 \ol{\xi} & \ol{\eta}
\end{pmatrix}
= \frac{1}{2}
\begin{pmatrix}
 T+Z & X_0 + \sum_{j=1}^nX_ji_j \\
 X_0 - \sum_{j=1}^nX_ji_j & T-Z
\end{pmatrix}.\tag*{\raisebox{-0.9em}{\defn}}\]
\end{definition}
\noindent It is evident that $\text{Im}(\phi_1) \subseteq L^+$ (where Im$(\cdot)$ indicates the image, not the imaginary component) as the pseudo-determinant of the outer product is zero and the trace is positive. In fact, the image is precisely the positive light-cone.
\begin{lemma}
\label{Phi1SurjectLightConeLemma}
 The image of $\phi_1$ is the positive light-cone $L^+$.
\end{lemma}
\noindent\textit{Proof.} $\text{Im}(\phi_1) \subseteq L^+$ as pdet$(\phi_1(\kappa))=0$ and $|\xi|^2+|\eta|^2 = T>0$. To show surjectivity, consider an arbitrary point $p \in L^+$:
\[p = \frac{1}{2}
\begin{pmatrix}
 T+Z & \pX \\
 \ol{\pX} & T-Z
\end{pmatrix}.\]
By definition, points of $L^+$ satisfy $T^2-Z^2-|\pX|^2=0,\ T>0$. Taken together, these conditions imply $T \geq Z$; it is then clear that the spinor
\[\frac{1}{\sqrt{2}}\begin{pmatrix}
 \frac{\pX}{\sqrt{T-Z}} \\
 \sqrt{T-Z}
\end{pmatrix}\]
is defined and maps to $p$.
\qed
\\\\
It is helpful to express the metric on $\R^{1,n+2}$ using these $\mathcal{P}^{1,n+2}$ presentations of points. We may also occasionally have cause to treat the space as $\R^{n+3}$ and consider the space equipped with the Euclidean metric $dT^2+dZ^2+...+dX_n^2$.
\begin{proposition}
 \label{MatrixMetricLemma}
 Given two points $P,Q \in \mathcal{P}^{1,n+2}$,
\[P = \begin{pmatrix}
 P_{11} & P_{\pX} \\
 \ol{P_{\pX}} & P_{22}
\end{pmatrix} = \frac{1}{2}
\begin{pmatrix}
    T_1+Z_1 & \pX_1 \\
    \ol{\pX_1} & T_1-Z_1
\end{pmatrix}
,\ Q = \begin{pmatrix}
 Q_{11} & Q_{\pX} \\
 \ol{Q_{\pX}} & Q_{22}
\end{pmatrix} =
\frac{1}{2}\begin{pmatrix}
    T_2+Z_2 & \pX_2 \\
    \ol{\pX_2} & T_2-Z_2
\end{pmatrix},\]
their generalised Minkowski inner product is given by 
\begin{equation}
\label{MatrixMinkowskiMetric}
(P|Q)_{1,n+2} = 2(P_{11}Q_{22}+P_{22}Q_{11}-P_{\pX}\ol{Q_{\pX}}-Q_{\pX}\ol{P_{\pX}}).
\end{equation}
Their Euclidean inner product is
\begin{equation}
\label{EuclideanInnerProduct}
(P|Q)_{n+3} = 2(P_{11}Q_{11}+Q_{22}P_{22}+P_{\pX}\ol{Q_{\pX}}+Q_{\pX}\ol{P_{\pX}}).
\end{equation}
\end{proposition}
\noindent \textit{Proof.} The dot product (see Definition \ref{defn:ParavectorInnerProduct}) on paravectors $\pX_1$ and $\pX_2$ is given by
\[\frac{\pX_1\ol{\pX_2}+ \pX_2\ol{\pX_1}}{2}=2(P_{\pX}\ol{Q_{\pX}}+Q_{\pX}\ol{P_{\pX}}).\]
We can wrangle the $T$ and $Z$ terms from the diagonal, noting they're real and therefore commute.
\begin{align*}
(T_1T_2-Z_1Z_2) &= \frac{(T_1+Z_1)(T_2-Z_2)+(T_2+Z_2)(T_1-Z_1)}{2}\\
&= 2(P_{11}Q_{22}+P_{22}Q_{11}).
\end{align*}
In total, we find the Minkowski metric is given by 
\[( P,Q )_{1,n+2} = 2(P_{11}Q_{22}+P_{22}Q_{11}-P_{\pX}\ol{Q_{\pX}}-Q_{\pX}\ol{P_{\pX}}).\] 
To instead get the Euclidean inner product we compute
\begin{align*}
(T_1T_2+Z_1Z_2) &= \frac{(T_1+Z_1)(T_2+Z_2)+(T_1-Z_1)(T_2-Z_2)}{2}\\
&= 2(P_{11}Q_{11}+P_{22}Q_{22}),
\end{align*}
which gives us
\[(P|Q)_{n+3} = 2(P_{11}Q_{11}+P_{22}Q_{22}+P_{\pX}\ol{Q_{\pX}}+Q_{\pX}\ol{P_{\pX}}). \tag*{\qed}\]

\subsection{Tangent Data}
The kernel of $D\phi_1$ is large; in fact $\phi_1$ is a fibre bundle map over $L^+$ with fibres homeomorphic to Spin$(n+1)$. We find an indication of this when examining the preimages $\phi_1^{-1}$ of points in the light-cone, as they resemble the diagonal map inclusion of Spin$(n+1)$ into $\clif^2$; to be precise, they are an inclusion $s \to (s,\pX s)$ or $s \to (\pY s,s)$ of Spin$(n+1)$ into $S\lip$ for some elements $\pX,\pY \in \para$. 
\begin{proposition}
\label{phi1FibreLemma}
Given a point $p \in L^+$ and a single preimage point $\kappa \in \phi^{-1}(p)$,
\begin{enumerate}[label=\roman*)]
\item the entire preimage is given by $\kappa s,\ s \in \text{Spin}(n+1)$, where the spin group Spin$(n+1)$ is realised as $\{s \in \lip\ |\ N(s)=1\}$.
\item Alternatively, the preimage $\phi_1^{-1}(p)$ for a point
\[p = \begin{pmatrix}
 p_{11} & p_{\pX} \\
\ol{p_{\pX}} & p_{22}
\end{pmatrix} \in L^+\]
is given when $p_{11} > 0$ by $(s,p_{11}^{-1}\ol{p_{\pX}}s)$ for $\{s \in \lip\ |\ N(s) = p_{11}\}$.

When $p_{22} > 0$, it can instead be given by $(p_{22}^{-1}p_{\pX}s,s)$ for $\{s \in \lip\ |\ N(s) = p_{22}\}$.
\end{enumerate}
\end{proposition}
\noindent\textit{Proof.} Consider a point $p$ on the light-cone. The preimage will be the set of $(\xi,\eta) \in S\lip$ which satisfy
\[ 
\begin{pmatrix}
 |\xi|^2 & \xi\ol{\eta} \\
 \eta\ol{\xi} & |\eta|^2
\end{pmatrix} = 
\begin{pmatrix}
p_{11} & p_{\pX} \\
\ol{p_{\pX}} & p_{22}
\end{pmatrix}.\]
If $\eta \neq 0$, then $\xi=\pV\eta$ for some $\pV \in \para$ and we can decompose the above matrix:
\[p=|\eta|^2\begin{pmatrix}
 |\pV|^2 & \pV \\
 \ol{\pV} & 1
\end{pmatrix}.\]
Our choice of $\pV$ and $|\eta|^2$ is fixed by the equalities $|\eta|^2=p_{22}$ and $\pV=p_{\pX}/|\eta|^2$. This gives us freedom to choose any $\eta$ satisfying
\begin{equation}
\label{eqn:BasepointPreimageSpinGrp}
\{\eta \in \lip\ |\ N(\eta) = p_{22}\}. 
\end{equation}
Taking $\xi \neq 0$ instead, by the analogous argument we get a preimage with a choice of $\{\xi \in \lip\ |\ N(\xi) = p_{11}\}$, which demonstrates ii). 

Now consider i). From Corollary \ref{cor:ParavSpinGrp}, we know the group $\text{Spin}(n+1)$ can be defined as $\{s \in \$\Gamma_{n+1,0}\ |\ N(s) = 1 \}$.  Comparing this with (\ref{eqn:BasepointPreimageSpinGrp}), we notice that the preimage $\phi^{-1}(p)$ is almost exactly Spin$(n+1)$, except rescaled so $N(\eta) = p_{22}$. If we take an element satisfying (\ref{eqn:BasepointPreimageSpinGrp}) and multiply by $s \in \text{Spin}(n+1)$, then $\eta s \ol{s}\, \ol{\eta} = \eta\ol{\eta}$ and we still satisfy (\ref{eqn:BasepointPreimageSpinGrp}). Every element that satisfies (\ref{eqn:BasepointPreimageSpinGrp}) can be generated this way, for if $s \in \lip$ and $N(s) = p_{22} = N(\eta)$ then $N(s^{-1}\eta) = 1$ and $s^{-1}\eta \in \text{Spin}(n+1)$. And, of course, $ss^{-1}\eta = \eta$. The analogous argument holds if we take $\xi \neq 0$ instead, which completes the proof of i).\qed
\\\\
This shows that $\phi_1$ loses a lot of information. To counter this, we can add data via derivatives. Following Section 4.4 of \cite{mathews2024quaternions}, the derivative at a point $\kappa \in S\lip$ in the direction $v$ can be found as follows:
\[\phi_1(\kappa+tv) = (\kappa+tv)(\kappa+tv)^{\dagger} = \kappa\kappa^{\dagger}+\left(\kappa v^{\dagger}+v\kappa^{\dagger}\right)t + vv^{\dagger}t^2,\ t \in \R.\]
Taking the derivative and evaluating at $t=0$, we find
\begin{equation}
\label{DerivEquation}
(D_{\kappa}\phi_1)(v) = \kappa v^{\dagger} + v \kappa^{\dagger}. 
\end{equation}
This derivative respects the orthogonal decomposition of the spinor tangent space (see Theorem \ref{thm:SpinorsTangentSpace}); in particular, each summand has a nice geometric interpretation.
\begin{theorem}
\label{OrthogonalDecompPreservationTheorem}
 Let $\kappa = (\xi,\eta) \in S\lip$ and $\phi_1(\kappa) = p = (T,Z; \pX)_{\R}$.
 \begin{enumerate}[label=\roman*)]
 \item $D_{\kappa}\phi_1$ maps $\kappa \R$ isomorphically onto $p\R$,
 \item $D_{\kappa}\phi_1$ maps $\check{\kappa}(\para)$ isomorphically onto $T_p\mathscr{S}_T^+$,
 \item The kernel of \ $D_{\kappa}\phi_1$ is $\kappa(\R^{0,n} \oplus \bigwedge^2\R^{0,n})$.
 \end{enumerate}
\end{theorem}
\noindent\textit{Proof.} We compute i) directly. For $a \in \R$:
\[(D_{\kappa}\phi_1)(a\kappa) = 2a\kappa\kappa^{\dagger} = 2a\phi_1(\kappa)=2ap.\]
For ii), Let $\pV \in \para$ and compute
\begin{equation}
\label{eqn:DerPhi1KappaCheck}
(D_{\kappa}\phi_1)(\check{\kappa}\pV) = 
\begin{pmatrix}
 \xi\ol{\pV}\eta^* + \eta'\pV\ol{\xi} & -\xi\ol{\pV}\xi^* + \eta'\pV\ol{\eta} \\
 \eta\ol{\pV}\eta^* -\xi'\pV\ol{\xi} & -\eta\ol{\pV}\xi^* - \xi'\pV\ol{\eta}
\end{pmatrix}.
\end{equation}
This derivative lives in the tangent space to $\R^{1,n+2}$ (which we've identified with $\mathcal{P}^{1,n+2}$), which is canonically isomorphic to $\R^{1,n+2}$ itself. Thus the diagonal entries are real, and invariant under conjugation. We compute the $T$ component, applying some conjugations to cancel terms:
\[(D_{\kappa}\phi_1)(\check{\kappa}\pV)_T = \left(\xi\ol{\pV}\eta^* + \eta'\pV\ol{\xi} -(\eta\ol{\pV}\xi^* + \xi'\pV\ol{\eta})^*\right) = 0.\]
As it lives in the tangent space to $L^+$ at $p$, we know $(D_{\kappa}\phi_1)(\check{\kappa}\pV) \in p^{\perp}$. Having no change in the $T$ coordinate means it stays tangent to the fixed $T$-slice $\mathscr{S}_T^+$ that models the celestial sphere. We conclude
\[(D_{\kappa}\phi_1)(\check{\kappa}\pV) \in T_p\mathscr{S}_T^+.\]
To show surjectivity, we appeal to the rank-nullity theorem and show that $D_{\kappa}\phi_1$ has trivial kernel on $\check{\kappa}(\para)$. As the components of $\check{\kappa}$ and $\pV$ are elements of $\lip^{\triangleright}$ they contain no zero divisors, and as both $\check{\kappa}$ and $\pV$ are non-zero $\check{\kappa}\pV \neq 0$. Suppose that there exists some $\pV \neq 0$ such that $D_{\kappa}\phi_1(\check{\kappa}\pV) = 0$. Then, by (\ref{DerivEquation}),
\begin{equation}
\label{eqn:TauKappaApplic}
\kappa(\check{\kappa}\pV)^{\dagger} = - \check{\kappa}\pV\kappa^{\dagger}.
\end{equation}
By Lemma \ref{TauKernelLemma}, there exists a row vector $\tau$ such that precisely one of $\tau \kappa$, $\tau \check{\kappa}$ is zero. This implies precisely one of $\tau \kappa(\check{\kappa}\pV)^{\dagger}$ and $-\tau\check{\kappa}\pV\kappa^{\dagger}$ is zero, but these are the same matrix by (\ref{eqn:TauKappaApplic}) and we have a contradiction.

Finally, we compute iii). For $B \in \mathcal{B}^n$ a bi-paravector with no real component:
\[(D_{\kappa}\phi_1)(\kappa B) = \kappa\ol{B}\kappa^{\dagger} + \kappa B\kappa^{\dagger} = \kappa( B + \ol{B})\kappa^{\dagger}=0.\]
The final equality is because $B+\ol{B}=2\text{Re}(B)$ for bi-paravectors. As the other dimensions of $T_{\kappa}S\lip$ are accounted for, this must comprise the entire kernel. \qed
\\\\
In \cite{Mathews_Spinors_horospheres}, a function $Z$ is chosen to coordinatise parts of the tangent space in a convenient way at $\kappa$; this is precisely the role of our complementary map $\check{\cdot}$, the analogue of a complex structure. In general there are $n$ tangent directions to capture, directions spanned by the basis elements $\check{\kappa}i_j,\ 1 \leq j \leq n$. This (along with $\check{\kappa}$ itself) forms an orthogonal basis for the $\check{\kappa}(\para)$ summand of the tangent space.

We can also capture $\check{\cdot}$ via a `complex structure' matrix.
\begin{definition}
 \[J = \begin{bmatrix}
 0 & 1 \\
 -1 & 0 \\
 \end{bmatrix}.\]
 It satisfies $J\kappa' = \check{\kappa}$.\null\hfill\defn
\end{definition}
\subsection{Decorated Ideal Points}
\label{sec:DecIdealPts}
In Sections \ref{sec:DecIdealPts} and \ref{sec:Multiflags} we present two ways of capturing the degrees of freedom within the kernel of $D_{\kappa}\phi_1$. The first, following \cite{mathews2024quaternions}, are \textit{decorated ideal points}; essentially, these involve identifying a basis for $\para$ in the tangent spaces to $\mathscr{S}^+$. We do this using the various derivatives $(D_{\kappa}\phi_1)(\check{\kappa}i_j)$ to make an identification of $\para$ with the tangent space. In particular, this construction clearly shows conformality of the correspondence. It is a generalisation of the `tangent vector to $\mathscr{S}^+$' construction of Penrose and Rindler \cite{SpinorsAndSpacetime}.
\begin{definition}
 Let $\ell \in \mathscr{S}^+$ be a point in the sphere at infinity (an `ideal point'). \label{not:decIdealPt}A {\normalfont decoration} on $\ell$ is a conformal orientation-preserving $\R$-linear isomorphism
 \[\psi:\ \para \to \frac{\ell^{\perp}}{\ell}\]
 with respect to the outward orientation on $\ell^{\perp}/\ell$, conformal with respect to the norm $N$ on $\para$ and the norm induced on $\ell^{\perp}/\ell$ by the generalised Minkowski metric. 
 \label{not:SpaceDecIdeal}The set of all decorated ideal points $(\ell,\psi)$ is denoted $\mathscr{S}^{+D}$.\null\hfill\defn
\end{definition}
\noindent Informally, a decoration is a choice of orthonormal basis (plus a global scaling factor) in the tangent space $T_{\ell}\mathscr{S}^+$ that identifies it with $\para$. The family of oriented orthonormal bases can be traversed by taking a starting basis and allowing all rotations, so the space of decorations at $\ell \in \mathscr{S}^+$ is really just $SO(n+1) \times \R$.

\subsubsection{The Action of $SL(2,\lip)$ on $TS\lip$}
Elements $A \in SL(2,\lip)$ act real linearly on $S\lip$, and also act real linearly on its tangent spaces, sending $T_{\kappa}S\lip \to T_{A\kappa}S\lip$ by left-multiplication. We recall from Theorem \ref{thm:SpinorsTangentSpace} that the tangent space $T_{\kappa}S\lip$ has a decomposition $\kappa\mathcal{B}^n \oplus \check{\kappa}\left(\para\right)$,
where $\mathcal{B}^n$ are the bi-paravectors. Acting by $A \in SL(2,\lip)$, the tangent space is sent to 
\[A.T_{\kappa}S\lip = A\kappa\mathcal{B}^n \oplus A\check{\kappa}\left(\para\right).\]
We can also apply this decomposition to the target tangent space:
\[T_{A\kappa}S\lip = A\kappa\mathcal{B}^n \oplus \widecheck{(A\kappa)}\left(\para\right).\]
There is an immediate problem; the first summand is preserved under the action, but the second is not. Not all is lost, however, as the second summand of both $A.T_{\kappa}S\lip$ and $T_{A\kappa}S\lip$ is isomorphic to $\para$, just identified differently.

To focus on the second summand, we can quotient out the first (as the first is preserved by the action it can safely be ignored) to get an action by $A$ as an $\R$-linear map $\tfrac{T_{\kappa}S\lip}{\kappa\mathcal{B}^n} \to \tfrac{T_{A\kappa}S\lip}{A\kappa\mathcal{B}^n}$. These quotients are isomorphic to the vector spaces $\check{\kappa}\para$ and $\widecheck{(A\kappa)}\para$ respectively. The points of $T_{\kappa}S\lip/\kappa\mathcal{B}^n$ are affine $\left(\tfrac{n^2+n}{2}+1\right)$-planes (which we shorthand as \textit{higher affine planes}) in $T_{\kappa}S\lip$ of the form $\check{\kappa}\pV+\kappa\mathcal{B}^n = s_{\pV}(\kappa)+\kappa\mathcal{B}^n$ over all $\pV \in \para$; recall that $s_{\pV}(\kappa)$ are our sections of the tangent bundle, as defined in Section \ref{Sec:SectionsTangentBundle}. We can characterise what points lie in these higher affine planes using the bracket operation, as shown in Lemma \ref{lem:QuotientBracketChar}. We can then describe the action of $SL(2,\lip)$ on the quotient.
\begin{lemma}
 \label{lem:SL2QuotientAction}
 Let $\kappa_0 \in S\lip$, $A \in SL(2,\lip)$, and $\kappa_1 = A\kappa_0$. For each $\pV \in \para$, the action of $A$ restricts to a map
 \[|\kappa_0|^{-2}s_{\pV}(\kappa_0) + \kappa_0\mathcal{B}^n \to |\kappa_1|^{-2}s_{\pV}(\kappa_1) + \kappa_1\mathcal{B}^n,\]
 which are higher affine planes in $T_{\kappa_0}S\lip$ and $T_{\kappa_1}S\lip$, respectively. 
\end{lemma}
\noindent\textit{Proof.} By Lemma \ref{lem:QuotientBracketChar}, the first set is the higher affine plane of $w \in T_{\kappa_0}S\lip$ such that $\{\kappa_0,w\} = -\pV$, while the second is the higher affine plane of $m \in T_{\kappa_1}S\lip$ such that $\{\kappa_1,m\} = -\pV$. By Lemma \ref{SesquimorphismLemma}, for any $w \in T_{\kappa_0}S\lip$ we have $\{\kappa_0,w\} = \{A\kappa_0,Aw\}$, which implies $\{\kappa_0,w\} = \{\kappa_1,Aw\} = -\pV$. So $Aw \in T_{\kappa_1}S\lip$, and as the action is invertible this map is bijective.\qed
\\\\
Because $SL(2,\lip)$ maps higher affine planes to higher affine planes, it induces a well-defined map on the quotient. This map is conformal under the inner product $(\cdot|\cdot)$ of Definition \ref{defn:InnerProductOnQuotient}.
\begin{lemma}
 Let $\kappa_0 \in S\lip$, $A \in SL(2,\lip)$, and $\kappa_1 = A\kappa_0$. The map
 \[\frac{T_{\kappa_0}S\lip}{\kappa_0\mathcal{B}^n} \to \frac{T_{\kappa_1}S\lip}{\kappa_1\mathcal{B}^n}\]
 induced by the action of $A$ is an orientation-preserving conformal $\R$-linear isomorphism under $(\cdot|\cdot)$, the metric of Definition \ref{defn:InnerProductOnQuotient}. 
\end{lemma}
\noindent\textit{Proof.} Let $\pV,\pW \in \para$. By Lemma \ref{lem:SL2QuotientAction}, $A$ sends the plane with representative $\check{\kappa}_0\pV = s_{\pV}(\kappa_0)$ to the plane represented by $\check{\kappa}_1\pV|\kappa_0|^2|\kappa_1|^{-2} = s_{\pV}(\kappa_1)|\kappa_0|^2|\kappa_1|^{-2}$. We thus have
\[(A\check{\kappa}_0\pV+\kappa_1\mathcal{B}^n | A\check{\kappa}_0\pW + \kappa_1\mathcal{B}^n) = \left( \check{\kappa}_1|\kappa_0|^2|\kappa_1|^{-2}\pV\ \middle|\ \check{\kappa}_1|\kappa_0|^2|\kappa_1|^{-2}\pW\right) = \frac{|\kappa_0|^4}{|\kappa_1|^4}(\check{\kappa}_1\pV|\check{\kappa}_1\pW).\]
We note that $SL(2,\lip)$ is connected and contains the identity matrix, which it is easy to show acts trivially and is thus orientation-preserving. We conclude the action of the connected component $SL(2,\lip)$ is also orientation-preserving, and $A$ is an orientation-preserving conformal map. 

Linearity then follows because the action of $SL(2,\lip)$ is by matrices, and it is an isomorphism because we can act by $A^{-1}$ (which exists as $SL(2,\lip)$ is a group) to reverse the action of $A$.\qed

\subsubsection{Conformality of $D\phi_1$}
From Theorem \ref{OrthogonalDecompPreservationTheorem} we have that, at $\kappa \in S\lip$, the derivative $D_{\kappa}\phi_1$ restricts to an isomorphism $\check{\kappa}(\para) \to T_{\phi_1(\kappa)}\mathscr{S}_T^+$. We also have a map $s(\cdot,\kappa): \para \to \check{\kappa}(\para)$ from our map of sections of $TS\lip$. We get a composition of linear isomorphisms
\[\para \overset{s(\cdot,\kappa)}{\longrightarrow} \check{\kappa}(\para) \overset{D_{\kappa}\phi_1}{\longrightarrow}T_{\phi_1(\kappa)}\mathscr{S}_T^+.\]
The first and second spaces have positive-definite inner products, and Lemma \ref{KappaCheckSectionsLemma} shows $s(\cdot,\kappa)$ is conformal. The third space $T_{\phi_1(\kappa)}\mathscr{S}_T^+$ is a space-like subspace of $\R^{1,n+2}$, and inherits a negative-definite metric by restriction of the Minkowski metric. We show that the map $D_{\kappa}\phi_1: \check{\kappa}(\para) \to T_{\phi_1(\kappa)}\mathscr{S}_T^+$ in the sequence above is conformal.
\begin{lemma}
\label{phiDerivConformalLemma}
 For $\pV \in \para$ and $\kappa \in S\lip$ we have $\textnormal{pdet}(D_{\kappa}\phi_1(\check{\kappa}\pV)) = -|\pV|^2|\kappa|^4$.
\end{lemma}
\noindent\textit{Proof.} From (\ref{eqn:DerPhi1KappaCheck}) we have the matrix form of $D_{\kappa}\phi_1(\check{\kappa}\pV)$:
\[D_{\kappa}\phi_1\left(\check{\kappa}\pV\right) = \begin{pmatrix}
 \xi\ol{\pV}\eta^* + \eta'\pV\ol{\xi} & -\xi\ol{\pV}\xi^* + \eta'\pV\ol{\eta} \\
 \eta\ol{\pV}\eta^* -\xi'\pV\ol{\xi} & -\eta\ol{\pV}\xi^* - \xi'\pV\ol{\eta}
\end{pmatrix}.\]
The pseudo-determinant can then be computed, and 
$\textnormal{pdet}(D_{\kappa}\phi_1(\check{\kappa}v)) %= &- |\pV|^2\left(|\xi|^4 + |\eta|^4 + 2|\xi|^2|\eta|^2\right)\\
%& + \xi\ol{\pV}\xi^*\eta\ol{\pV}\eta^* - \xi\ol{\pV}\xi^*\eta\ol{\pV}\eta^*  + \eta'\pV\ol{\xi}\eta'\pV\ol{\xi} - \eta'\pV\ol{\xi}\eta'\pV\ol{\xi}\\
%= &-|\pV|^2\left(|\xi|^2+|\eta|^2\right)^2 
= -|\pV|^2|\kappa|^4$
as claimed. \qed
\begin{proposition}
 \label{prop:Dkphi1conformal}
 For any $\pV,\pW \in \para$ and $\kappa \in S\lip$,
 \[(D_{\kappa}\phi_1(\check{\kappa}\pV)|D_{\kappa}\phi_1(\check{\kappa}\pW))_{1,n+2} = -4|\kappa|^4\pV\cdot \pW = -4|\kappa|^2(\check{\kappa} \pV | \check{\kappa} \pW),\]
 applying the generalised Minkowski inner product, the dot product (\ref{eqn:DotProduct}) on $\clif$, and the inner product $(\cdot|\cdot)$ on $\clif^2$ (Definition \ref{defn:CliffSquaredInnerProduct}) respectively.
\end{proposition}
\noindent\textit{Proof.} Recall\footnote{See the discussion of paravector Hermitian matrices below Definition \ref{defn:ParavectorHermitianMatrices}.} that if a matrix $S_p \in \mathcal{P}^{1,n+2}$ corresponds to a point $p \in \R^{1,n+2}$ then $4\text{pdet}(S_p) = |p|^2$. If we let $S_D$ be the paravector Hermitian matrix corresponding to $D_{\kappa}\phi_1(\check{\kappa}\pV)$, Lemma \ref{phiDerivConformalLemma} then tells us
\begin{equation}
 \label{TechnicalEq1}
 |S_D|^2 = 4\text{pdet}\left(D_{\kappa}\phi_1(\check{\kappa}\pV)\right)= -4|\pV|^2|\kappa|^4.
\end{equation}
We can polarise on both norms (that is, the Minkowski norm and the norm $N$ on $\para$) to find a relationship between the inner products. The norm on the left-hand side of (\ref{TechnicalEq1}) polarises as
\[4( D_{\kappa}\phi_1(\check{\kappa}\pV)|D_{\kappa}\phi_1(\check{\kappa}\pW))_{1,n+2} = |D_{\kappa}\phi_1(\check{\kappa}(\pV+\pW))|^2-|D_{\kappa}\phi_1(\check{\kappa}(\pV-\pW))|^2,\]
while the norm on paravectors polarises as $4\pV \cdot \pW = |\pV + \pW|^2 - |\pV - \pW|^2$. Apply these to (\ref{TechnicalEq1}) for the first equality. The second follows from Lemma \ref{InnerProductPropertiesLemma} iv).\qed
\\\\
With tools in hand, we now show the derivative map on the quotient is conformal. 
\begin{lemma}
\label{lem:QuotientMapConformal}
 The quotient map 
 \[D_{\kappa}\phi_1: \frac{T_{\kappa}S\lip}{\kappa \mathcal{B}^n} \to \frac{\ell^{\perp}}{\ell}\] 
 from sections $s(\cdot,\kappa)$ to $T_p\mathscr{S}_T^+$ is conformal with scaling factor $-4|\kappa|^2$. It is orientation-preserving with respect to the outward orientation on $\mathscr{S}^+$.
\end{lemma}
\noindent\textit{Proof.} In Section \ref{Sec:SectionsTangentBundle},
every element of the quotient $T_{\kappa}S\lip/\kappa\mathcal{B}^n$ is seen to have a unique representative that is a section $s_{\pV}(\kappa)$ for $\pV \in \para$. For $\pV,\pW \in \para$, 
\[(s_{\pV}(\kappa) + \kappa\mathcal{B}^n | s_{\pW}(\kappa) + \kappa\mathcal{B}^n) = (s_{\pV}(\kappa) | s_{\pW}(\kappa)) = |\kappa|^2\pV \cdot \pW,\]
where the left inner product is on $\tfrac{T_{\kappa}S\lip}{\kappa \mathcal{B}^n}$, the central is the inner product on $\clif^2$, and the rightmost is the dot product on $\clif$. Conformality now follows from Proposition \ref{prop:Dkphi1conformal}:
\[(D_{\kappa}\phi_1(s_{\pV}(\kappa)) | D_{\kappa}\phi_1(s_{\pW}(\kappa)))_{1,n+2} = -4|\kappa|^2(s_{\pV}(\kappa) | s_{\pW}(\kappa)) = -4|\kappa|^2(s_{\pV}(\kappa) + \kappa\mathcal{B}^n | s_{\pW}(\kappa) + \kappa\mathcal{B}^n).\]
To show the map is orientation-preserving, note that at each $\kappa \in S\lip$ the map $D_{\kappa}\phi_1$ sends a standard oriented basis $s_1(\kappa),s_{i_1}(\kappa),...,s_{i_n}(\kappa)$ to a basis of $T_p\mathscr{S}_T^+$. The space $S\lip$ is connected (except for the case of $n=0$, but there the decoration becomes trivial), so we can check the orientation for a single $\kappa$ and continuously vary to argue the map is orientation-preserving (or reversing) equally over all of $S\lip$. 

The spinor $\kappa_0 = (1,0)$ with tangent space $T_{\kappa_0}S\lip = \kappa_0\mathcal{B}^n\oplus \check{\kappa}_0\para$ is sent to $p_0 = (1,1;0,...,0)_{\R}$, with tangent space $T_{p_0}L^+ = p_0^{\perp}$. The tangent space $p_0^{\perp}$ is spanned by $p_0$ and the basis $\partial X_0,...,\partial X_n$ of $T_{p_0/{\sim}}\mathscr{S}^+$. We find the derivative maps the tangent vectors as follows:
\begin{equation}
\label{eqn:Derivs}
D_{\kappa_0}(s_1(\kappa_0)) = 2\partial_{X_0},\ D_{\kappa_0}(s_{i_j}(\kappa_0)) = 2\partial_{X_j},
\end{equation}
and so the oriented basis $s_1(\kappa_0),s_{i_1}(\kappa_1),...,s_{i_n}(\kappa_n)$ is mapped to the basis $2\partial X_0,...,2\partial X_n$ of $T_{p_0}\mathscr{S}^+$, which is outward oriented. We conclude the map is orientation-preserving.\qed

\subsection{Multiflags}
\label{sec:Multiflags}
The missing degrees of freedom lost by $\phi_1$ can also be recovered via a flag structure, generalising the `null flag' picture of Penrose and Rindler \cite{SpinorsAndSpacetime} in the complex spinor case. 
\begin{definition}[Flag]
    Given a vector space $V^n$, a flag of $V^n$ is a strictly ascending chain of vector subspaces $V_0 \subset V_1 \subset ... \subset V_r$. Because the inclusions are strict, 
    \[0 \leq \text{dim}(V_0) < \text{dim}(V_1) < ... \text{dim}(V_r) \leq n,\]
    where dim$(V_j)$ is the dimension of $V_j$. A flag's {\normalfont signature} is the sequence $(\text{dim}(V_0),...,\text{dim}(V_r))$.
    
    An {\normalfont oriented flag} is a flag with an orientation on each quotient space $V_{j+1}/V_j$, where $V_j \subset V_{j+1}$ are neighbouring subspaces in the flag.\null\hfill\defn
\end{definition}
\noindent The \textit{null flag} is a particular type of flag structure living on the light-cone in Minkowski space, so named because the light-cone is also the space of `null directions' in Minkowski, the vectors of length zero. 
\begin{definition}[Null Flag]
    Consider the light-cone $L^+ \subset \R^{1,n+2}$. \label{not:FlagSingle}Given a point $p \in L^+$ and a vector $v \in T_pL^+$ linearly independent of $p$, let $F_{p,v}$ be the span of $\R p$ and $v$ and denote the flag $0 \subset \R p \subset F_{p,v}$ by $[[p,v]]$. We call $p$ the {\normalfont base-point} and $F_{p,v}$ the {\normalfont flag plane}, and say the flag is {\normalfont based} at $p$. The {\normalfont flagpole} $p\R$ is chosen to be future-oriented, and $F_{p,v}/p\R$ is oriented by the equivalence class of the vector $v$. Such a flag is called a {\normalfont null flag}. \hfill\defn
\end{definition}
\noindent Following \cite{Mathews_Spinors_horospheres}, all flags are taken to be oriented and of signature (1,2). That is, all flags are of the form $\{0\} =V_0 \subset V_1 \subset V_2$, where dim $V_1=1$, dim $V_2=2$, and $V_1/V_0 = V_1$ and $V_2/V_1$ are endowed with orientations.

The fibres of $\phi_1$ are (up to scaling) the spin group $Spin(n+1)$, and it is this group we seek to reconstruct with flags. Given a particular Lipschitz spinor $\kappa$, it should map to a point of the light-cone and some collection of flags with an orthogonality condition upon them, to recover something like a spin group. 

But orthogonal in what sense? Where should our flags live? Our copy of $Spin(n+1)$ in the fibres should act on an $(n+1)$-dimensional vector space, so we should look for a `natural' choice of such a space within $\R^{1,n+2}$. Similar to the decorated ideal points, we choose to act on $T_{\ell}\mathscr{S}^+$, where $\ell = \phi_1(\kappa)\R$. Our flags will all share a flagpole (given by the span of $\phi_1(\kappa)$), and the flag planes will be determined by various derivatives in the tangent space to $L^+$.

The sphere at infinity $\mathscr{S}^+$ can be thought of as the set of directions available along the light-cone; each ray of the light-cone becomes a point in the sphere. Thus, a flag $[[p,v]]$ would have its flagpole mapped to the equivalence class of $p$ in $\mathscr{S}^+$. We can also imagine mapping the tangent space of $L^+$ at $p$ to the tangent space of $\mathscr{S}^+$ by quotienting out $\R p$, the ray pointing along $p$. This would send the flag $[[p,v]]$ to the equivalence class of $p \in \mathscr{S}^+$ with tangent vector $v + \R p$. 

However, $p$ is orthogonal to itself under the Minkowski metric. By the bilinearity of the metric, then, orthogonality of these quotients in the tangent space $T_p\mathscr{S}^+$ is equivalent to orthogonality of the flag plane vectors $v_j$. That is, if we take two vectors $v_1,v_2 \in T_pL^+$ and two arbitrary real numbers $r_1,r_2 \in \R$, to have orthogonality we require
\[(v_1+r_1p|v_2+r_2p)_{1,n+2} = (v_1|v_2)_{1,n+2}+r_2(p|v_2)_{1,n+2} + r_1(v_1|p)_{1,n+2}+r_1r_2(p|p)_{1,n+2} = 0.\]
Then, because $v_1,v_2,p \in p^{\perp}$, this implies
\[(v_1+r_1p|v_2+r_2p)_{1,n+2} = (v_1|v_2)_{1,n+2} = 0.\]
It follows that, rather than dealing with quotients, we can consider the orthogonality of flags simply by comparing the tangent vectors defining the flag plane.
\begin{definition}[Orthogonality of Flags]
\label{defn:OrthoOfFlags}
 Consider two flags $F_1 = [[p,v_1]]$ and $F_2 = [[p,v_2]]$ that share the same flagpole vector. These flags are orthogonal if $(v_1|v_2)_{1,n+2} = 0$. \null\hfill\defn
\end{definition}
\noindent All of this suggests the correct `space of flags' should be those with orthogonal flag planes.
\begin{definition}
\label{def:Multiflag}
The {\normalfont null multiflag} (or simply {\normalfont multiflag}) $[[p;v_1,v_2,...,v_n]]$ is the ordered set (or tuple) of null flags $([[p,v_1]],\ [[p,v_2]],...,[[p,v_n]])$, each based at $p$ and bearing a `flag plane vector' $v_j \in T_pL^+$. The flag plane vectors $v_j$ must be mutually orthogonal, and the collection $\{v_1,v_2,...,v_n,p\}$ must form a linearly independent set. The individual constituent flags $[[p,v_j]]$ are called {\normalfont subflags}. \null\hfill\defn 
\end{definition}
\begin{lemma}
 The subflags of $[[\phi_1(\kappa);(D_{\kappa}\phi_1)(\check{\kappa}i_1),(D_{\kappa}\phi_1)(\check{\kappa}i_2),...,(D_{\kappa}\phi_1)(\check{\kappa}i_n)]]$ are orthogonal under the Minkowski metric; that is, they do indeed collect to form a proper multiflag.
\end{lemma}
\noindent \textit{Proof.} The vectors $i_j$ are orthogonal, and remain orthogonal under the conformal map $D_{\kappa}\phi_1$ (see Lemma \ref{lem:QuotientMapConformal}). This implies the flags themselves are orthogonal, by Definition \ref{defn:OrthoOfFlags}.\qed
\\\\
As an appropriate target space for our map from $S\lip$, we define a bundle of multiflags over the light-cone.
\begin{definition}
\label{def:MultiflagSpace}
The space of null multiflags $\mathcal{MF}^n_p$ at a point $p \in L^+$ consists of all $[[p;v_1,v_2,...,v_n]]$ based at $p$ with flag plane vectors $v_j$ that are mutually orthogonal under $(\cdot|\cdot)_{1,n+2}$, and form a positively oriented frame. The bundle of null multiflags $\mathcal{MF}^n$ is then the fibre bundle with fibre $\mathcal{MF}^n_p$ at each $p \in L^+$. \hfill\defn
\end{definition}
\noindent The space of multiflags $\mathcal{MF}^n_p$ based at $p$ is diffeomorphic to the space of positively oriented orthonormal frames at $p$, which is diffeomorphic to $SO(n+1)$. $\mathcal{MF}^n$ is then a frame bundle over $L^+ \cong S^{n+1} \times \R$.
\begin{definition}
\label{defn:Phi1}
 The map $\Phi_1$ from non-zero spinors to multiflags is given by
 \[\Phi_1: S\lip \to \mathcal{MF}^n,\ \Phi_1(\kappa) = [[\phi_1(\kappa);(D_{\kappa}\phi_1)(\check{\kappa}i_1),(D_{\kappa}\phi_1)(\check{\kappa}i_2),...,(D_{\kappa}\phi_1)(\check{\kappa}i_n)]].\]
 We refer to the individual flag $[[\phi_1(\kappa),(D_{\kappa}\phi_1)(\check{\kappa}i_j)]]$ as the $i_j$ flag, or $i_j$ subflag. \null\hfill\defn
\end{definition}
\noindent Multiflags and decorated ideal points carry the same information, as the following proposition (an extension of Proposition 4.8.2 of \cite{mathews2024quaternions}) shows. Recall that $\mathscr{S}^{+D}$ is the space of decorated ideal points, pairs $(\ell,\psi)$ where $\ell \in \mathscr{S}^+$ is a point in the celestial sphere and the decoration $\psi$ is a conformal orientation-preserving $\R$-linear isomorphism from $\para \to \ell^{\perp}/\ell$ (informally, an orthonormal basis up to a global scale factor for the tangent space $T_{\ell}\mathscr{S}^+$).
\begin{proposition}
\label{prop:MFlagsAreDecIdealPoints}
 There is a smooth bijective correspondence $\Psi:\ \mathcal{MF}^n \to \mathscr{S}^{+D}$, as follows.
 \begin{enumerate}[label=\roman*),leftmargin=0.9cm]
 \item Given a multiflag $[[p;v_1,...,v_n]] \in \mathcal{MF}^n$, let $p$ have $T$ coordinate $T_0$, and let $a_1,a_2,...,a_n > 0$ be chosen such that $a_1v_1,a_2v_2,...,a_nv_n$ have generalised Minkowski norm $-4T_0^2$. Then
 \[\Psi([[p;v_1,v_2,...,v_n]]) = (p\R,\psi),\]
 where $\psi$ is uniquely defined by $\psi(i_j) = a_jv_j + p\R$.
 \item Given $(\ell,\psi) \in \mathscr{S}^{+D}$, let $K < 0$ be the conformal scale factor of $\psi$. Then
 \[(\ell,\psi) \to [[p;\Tilde{\psi}(i_1),...,\Tilde{\psi}(i_n)]],\]
 where $p$ is the point on $\ell$ whose $T$-coordinate $T_0$ satisfies $K=-4T_0^2$, and $\Tilde{\psi}(i_j) \in \ell^{\perp}$ for $j \in \{1,...,n\}$ are arbitrary lifts of $\psi(i_j) \in \ell^{\perp}/\ell$ to $\ell^{\perp} = T_pL^+$.
 \end{enumerate}
\end{proposition}
\noindent\textit{Proof.} Consider the multiflag $[[p;v_1,...,v_n]]$, and let $\ell = p\R$. Flags at $p$ are subsets  of $T_pL^+ = \ell^{\perp}$, and correspond to oriented lines in $\ell^{\perp}/\ell$. The equivalence classes of the $v_j$ in $\ell^{\perp}/\ell$ orient the lines corresponding to the $i_j$ flags. The $v_j$ are space-like, as are their equivalence classes in $\ell^{\perp}/\ell$. 

As $T_0$ is strictly positive and the $v_j$ are space-like (and thus have negative Minkowski norm), by linearity there are a unique set of $a_j > 0$ such that $a_jv_j$ have generalised Minkowski norm $-4T_0^2$, and therefore a unique set of $a_j > 0$ such that $a_jv_j+\ell \in \ell^{\perp}/\ell$ also have norm $-4T_0^2$. Note that $(\cdot |\cdot)_{1,n+2}$ descends to the quotient $\ell^{\perp}/\ell$ because it is zero on $\ell$.

As $\ell^{\perp}/\ell$ is oriented, $n+1$-dimensional, and negative definite under the norm descending from $(\cdot |\cdot)_{1,n+2}$, there exists a unique $v_0$ of norm $-4T_0^2$ such that $(v_0+\ell,a_1v_1+\ell,...,a_nv_n+\ell)$ forms an oriented orthogonal basis. There is a unique orientation-preserving conformal $\psi$ such that $\psi(i_j) = a_jv_j+\ell$, and because we chose $v_0$ with norm $-4T_0^2$, it also satisfies $\psi(1)=v_0$. Since $\psi$ sends the generators $1,i_1,...,i_n$ (all elements of unit $N$-norm) to elements of norm $-4T_0^2$, $\psi$ has scale factor $K=-4T_0^2$.

Conversely, from $(\ell,\psi)$ we can set $p$ as the point on $\ell$ with $T$-coordinate $T_0 > 0$ such that the scale factor $K$ of $\psi$ is $K=-4T_0^2$. Then $\psi(i_j) \in \ell^{\perp}/\ell = T_pL^+/p\R$ gives us the $i_j$ flag plane vector. Taking arbitrary lifts $\Tilde{\psi}(i_j) \in \ell^{\perp}$ we find the $i_j$-flag is $[[p;\Tilde{\psi}(i_j)]]$, and as $\psi$ is conformal and the $i_j \in \para$ are orthogonal the $\psi(i_j)$ are also orthogonal, which implies the flags they define are orthogonal, and we have a multiflag as claimed.

To see that these correspondences are inverses of one another, suppose $[[p;v_1,...,v_n]]$ maps to $(\ell,\psi)$, so $\ell = p\R$, $\psi$ has scaling factor $K=-4T_0^2$, and $\psi(i_j) = a_jv_j+p\R$ as above. Let $(\ell,\psi)$ map to $[[q;\Tilde{\psi}(i_1),...,\Tilde{\psi}(i_n)]]$. Since $q$ lies on $p\R$ and has $T$-coordinate $T'$ such that $K=-4T_0^2=-4T'^2$ and $T_0,T'>0$, we conclude $p=q$. Since $\psi(i_j)=a_jv_j+p\R$ and $a_j>0$, any lift $\Tilde{\psi}(i_j) \in \ell^{\perp}$ yields the same flag as $v_j$. Thus $[[p;v_1,...,v_n]] = [[q;\Tilde{\psi}(i_1),...,\Tilde{\psi}(i_n)]]$.

Similarly, suppose $(\ell,\psi)$ maps to $[[p;v_1,...,v_n]]$ where $\psi$ has scaling factor $K$, the $T$-coordinate $T_0$ of p satisfies $K=-4T_0^2$, and the $v_j \in \ell^{\perp}$ generate the same flags as $\psi(i_j) \in \ell^{\perp}/\ell$. Let $[[p;v_1,...,v_n]]$ map to $(\ell',\psi')$. Then $\ell' = p\R$, so $\ell' = \ell$. $\psi'$ has scaling factor $-4T_0^2$, so sends each $i_j$ to an element of $\ell^{\perp}/\ell$ of norm $-4T_0^2$, generating the flags $[[p,v_1]]$,...,$[[p,v_n]]$, just like $\psi$. Thus $\psi=\psi'$ and so $(\ell,\psi) = (\ell',\psi')$. The correspondences are inverses of one another, and clearly smooth.\qed

\subsection{The Actions of $SL(2,\lip)$}
\label{sec:ActionsofSL2OnMinkowski}
To relate spinors and geometric objects in a useful way, our correspondences should be equivariant under the action of $SL(2,\lip)$. To have equivariance, though, we need an $SL(2,\lip)$ action on each of our spaces. We have an action by M\"{o}bius transformations on the paravectors $\para$ (\ref{ParavectorMobiusAction}), and on $S\lip$ we have an action by matrix multiplication that is equivalent to M\"{o}bius transformations under projectivisation (see Lemma \ref{lem:SL2DefinesSpinorAction}). 
\begin{definition}
\label{defn:PHermAction}
For a paravector Hermitian matrix $S_p \in \mathcal{P}^{1,n+2}$ representing $p \in \R^{1,n+2}$, the action of $A \in SL(2,\lip)$ is defined to be $A.S_p = AS_pA^{\dagger}$. \null\hfill\defn
\end{definition}
\noindent This is generalising the standard conjugation action of $SL(2,\C)$ on $\R^{1,3}$.
\begin{lemma}
    Let $A \in SL(2,\lip)$, and $S_p \in \mathcal{P}^{1,n+2}$. The action $A.S_p = AS_pA^{\dagger}$ of $A$ on $S_p$ is a map $\R^{1,n+2} \to \R^{1,n+2}$.
\end{lemma}
\noindent\textit{Proof.} It is easy to show this action preserves the light-cone; by Lemma \ref{Phi1SurjectLightConeLemma} every point in $L^+$ is the image $\phi_1(\kappa)$ of some $\kappa \in S\lip$, and 
\[A.\phi_1(\kappa) = A\phi_1(\kappa)A^{\dagger} = (A\kappa)(\kappa^{\dagger}A^{\dagger}) = \phi_1(A\kappa).\]
We have already seen that $SL(2,\lip)$ has a well-defined action on $S\lip$, so $A\kappa \in S\lip$ and $\phi_1(A\kappa) \in L^+$. 

We can then choose a basis $\textbf{B}$ for $\R^{1,n+2}$ from $L^+$, because the span of $L^+$ is $\R^{1,n+2}$. For example,
\begin{align*}
B_{+} &= \phi_1(1,0) = \begin{pmatrix}
    1 & 0 \\
    0 & 0
\end{pmatrix},\ 
B_{-} = \phi_1(0,1) = \begin{pmatrix}
    0 & 0 \\
    0 & 1
\end{pmatrix},\\
B_0 &= \phi_1(1,1) = \begin{pmatrix}
    1 & 1 \\
    1 & 1
\end{pmatrix},\ 
B_j = \phi_1(i_j,1) = \begin{pmatrix}
    1 & i_j \\
    -i_j & 1
\end{pmatrix},
\end{align*}
for $1 \leq j \leq n$. We can recover the standard basis for $\R^{1,n+2}$ as follows:
\begin{align*}
(1,0;0)_{\R} &= \tfrac{1}{2}(B_+ + B_-),\ (0,1;0)_{\R} = \tfrac{1}{2}(B_+ - B_-),\\ (0,0;1)_{\R} &= \tfrac{1}{2}(B_0 - B_+ - B_-),\ (0,0;i_j)_{\R} = \tfrac{1}{2}(B_j - B_+ - B_-).
\end{align*}
As the action is a combination of left and right matrix multiplication, the action of $A$ is a real-linear map $L^+ \to L^+$. A point of $\R^{1,n+2}$ written as a linear combination of elements of $\textbf{B}$ in $L^+$ will be sent to another linear combination of points in $L^+$, and as the the span of $L^+$ is $\R^{1,n+2}$ we conclude the image is also in $\R^{1,n+2}$.
\qed
\\\\
This action is defined on both $\phi_1(\kappa)$ and its derivatives (both are in generalised Minkowski space and thus $\mathcal{P}^{1,n+2}$), so it also acts on multiflags. We denote the action on a multiflag $[[p;v_1,...,v_n]]$ by 
\[A.[[p;v_1,...,v_n]]=[[A.p;A.v_1,...,A.v_n]].\]
\vspace{-1em}
\begin{lemma}
\label{ActionIsometryMink}
The action of $SL(2,\lip)$ on $\R^{1,n+2}$ is by orientation-preserving isometries.
\end{lemma}
\noindent\textit{Proof.} Represent a point $p \in \R^{1,n+2}$ as a matrix $S_p \in \mathcal{P}^{1,n+2}$: 
\[S_p = \frac{1}{2}\begin{pmatrix}
 T+Z & \pX \\
 \ol{\pX} & T-Z
\end{pmatrix}.\]
The norm on $S_p \in \mathcal{P}^{1,n+2}$ is given by the pseudo-determinant $\text{pdet}(S_p)$ (up to a constant factor). \label{not:GL2LipMonoid}Define $GL(2,\lip)^{\triangleright}$ by extending the pseudo-determinant condition of $GL(2,\lip)$ to allow any value in $\R$; this will allow us to consider $L^+ \subset GL(2,\lip)$. Action by $A \in SL(2,\lip)$ gives $AS_pA^{\dagger}$, but the pseudo-determinant is multiplicative on elements of $GL(2,\lip)^{\triangleright}$, and it is easy to check that $S_p \in GL(2,\lip)^{\triangleright}$ also. This implies
\[\text{pdet}(AS_pA^{\dagger})=\text{pdet}(A)\text{pdet}(S_p)\text{pdet}(A^{\dagger}) = \text{pdet}(S_p).\]
$SL(2,\lip)$ then acts via isometries as it preserves the norm.

Theorem B of \cite{Ahlfors1985} tells us the group $SL(2,\lip)$ is the double cover of the connected orientation-preserving group of M\"{o}bius transformations $M_+(\widehat{\R}^n)$, which implies $SL(2,\lip)$ is also connected; so, if one element of $SL(2,\lip)$ has an orientation-preserving action, they all do. But the identity matrix is an element of $SL(2,\lip)$, and its action is orientation-preserving.\qed
\subsection{$SL(2,\lip)$ Equivariance of $\Phi_1$}
\label{sec:EquivarianceOfPhi1}
The map $\phi_1$ is equivariant with respect to the $SL(2,\lip)$-action. For $A \in SL(2,\lip)$:
\begin{equation}
\label{eqn:EquivarOfphi1}
\phi_1(A\kappa) = (A\kappa)(\kappa^{\dagger}A^{\dagger})=A\phi_1(\kappa)A^{\dagger}=A.\phi_1(\kappa).
\end{equation}
Equivariance also holds for the derivatives:
\[A.(D_{\kappa}\phi_1)(v)=A.(\kappa v^{\dagger}+v\kappa^{\dagger})=A\kappa v^{\dagger}A^{\dagger}+Av\kappa^{\dagger}A^{\dagger} =(A\kappa)(Av)^{\dagger}+(Av)(A\kappa)^{\dagger} = (D_{A\kappa}\phi_1)(Av).\]
This equivariance also extends to $\Phi_1$. We first generalise Lemma 2.8 of \cite{Mathews_Spinors_horospheres} (and Lemma 4.6.5 of \cite{mathews2024quaternions}) to the Lipschitz spinor case.
\begin{lemma}
\label{NuZjLemma}
 Given a spinor $\kappa$ and a set of tangent vectors $\nu_j,\ 1\leq j \leq n$ at $\kappa \in S\lip$ (that is, $\nu_j \in T_{\kappa}S\lip$), the following are equivalent:
 \begin{enumerate}[label=\roman*)]
 \item $\{\kappa,\nu_j\}$ is a (non-zero) negative multiple of $i_j$ for all $j$,
 \item $\nu_j = \kappa a_j + b_j\check{\kappa}i_j$ for $a_j \in \mathcal{B}^n$, $b_j\in \R^+$ for all $j$,
 \item $[[\phi_1(\kappa);(D_{\kappa}\phi_1)(\nu_1),...,(D_{\kappa}\phi_1)(\nu_n)]]=\Phi_1(\kappa)$.
 \end{enumerate}
\end{lemma}
\noindent
\textit{Proof}.\\
\noindent
$i) \to ii)$: $\check{\kappa}$ is defined so $\{\kappa,\check{\kappa}i_j\}$ is a negative multiple of $i_j$. We can construct any other negative multiple of $i_j$ scaling by some $b_j \in \R^+$. Thus, if $\{\kappa,\nu_j\}$ is a negative multiple of $i_j$, there is some choice of $b_j \in \R^+$ so that $\{\kappa,\nu_j\} = b_j\{\kappa,\check{\kappa}i_j\} = \{\kappa,b_j\check{\kappa}i_j\}$. By linearity $\{\kappa,\nu_j-b_j\check{\kappa}i_j\}=0$, so $\nu_j = b_j\check{\kappa}i_j+D$ for some $D$ in the kernel of the map $\{\kappa,\cdot\}$. From Lemma \ref{KernelLemma}, we know the kernel is all elements of the form $\kappa a_j$, $a_j \in \clif$, and therefore $\nu_j = \kappa a_j + b_j\check{\kappa}i_j$. But because $\nu_j-b_j\check{\kappa}i_j=\kappa a_j$ is a tangent vector, Corollary \ref{TangentSpinorLemma} implies $a_j \in \mathcal{B}^n$.
\\\\
$ii \to i)$: We compute $\{\kappa,\nu_j\}$ using properties of the bracket from Section \ref{sec:TheBracket}:
\[\{\kappa,\nu_j\} = \{\kappa,\kappa a_j\} + b_j\{\kappa,\check{\kappa}\}i_j = -b_j|\kappa|^2i_j.\]
As $b_j>0$ this is a non-zero negative multiple of $i_j$.
\\\\
\noindent $ii) \to iii)$: Consider (\ref{DerivEquation}). If we apply this to $\nu_j = \kappa a_j + b_j\check{\kappa}i_j$ and expand we find:
\begin{align*}
%(D_{\kappa}\phi_1)\left(\kappa a_j + b_j\check{\kappa}i_j\right) &= \kappa \left(\kappa a_j+ b_j\check{\kappa}i_j\right)^{\dagger} + \left(\kappa a_j + b_j\check{\kappa}i_j\right)\kappa^{\dagger}\\
%&= \kappa \ol{a}_j\kappa^{\dagger} + b_j\kappa (\check{\kappa}i_j)^{\dagger} + \kappa a_j \kappa^{\dagger}+ b_j(\check{\kappa}i_j)\kappa^{\dagger}\\
%&=\kappa (a_j+\ol{a}_j) \kappa^{\dagger} + b_j (\kappa (\check{\kappa}i_j)^{\dagger}+(\check{\kappa}i_j)\kappa^{\dagger})\\
&\text{Re}(a_j)(D_{\kappa}\phi_1)(\kappa) + b_j (D_{\kappa}\phi_1)(\check{\kappa}i_j).
\end{align*}
The derivative $\text{Re}(a_j)(D_{\kappa}\phi_1)(\kappa) = 2\text{Re}(a_j)\kappa\kappa^{\dagger}$ is proportional to $\phi_1(\kappa)$, so we can rewrite this as: 
\[(D_{\kappa}\phi_1)\left(\kappa a_j + b_j\check{\kappa}i_j\right)= 2\text{Re}(a_j)\phi_1(\kappa) + b_j (D_{\kappa}\phi_1)(\check{\kappa}i_j).\]
That is, the derivative $(D_{\kappa}\phi_1)\left(\kappa a_j + b_j\check{\kappa}i_j\right)$ is a linear combination of $\phi_1(\kappa)$ and $(D_{\kappa}\phi_1)(\check{\kappa}i_j)$, for some $b_j>0$. This implies the subflags $[[\phi_1(\kappa),(D_{\kappa}\phi_1)(\nu_j)]]$ and $[[\phi_1(\kappa),(D_{\kappa}\phi_1)(\check{\kappa}i_j)]]$ are the same for all $j$, and therefore the multiflags are equal. 
\[[[\phi_1(\kappa);(D_{\kappa}\phi_1)(\nu_1),...,(D_{\kappa}\phi_1)(\nu_n)]]=\Phi_1(\kappa).\]
$iii) \to ii)$: The equality of multiflags is equivalent to the equality of their subflags for all $j$. If the $i_j$ subflags are equal, then the flag plane vectors $(D_{\kappa}\phi_1)(\nu_j)$ must be a linear combination of the vectors spanning the $i_j$ subflag in $\Phi_1(\kappa)$:
\[(D_{\kappa}\phi_1)(\nu_j) = a_j\phi_1(\kappa) + b_j(D_{\kappa}\phi_1)(\check{\kappa}i_j),\]
where $a_j \in \R$ and $b_j \in \R^+$. Substituting $D_{\kappa}\phi_1(\kappa) = 2\phi_1(\kappa)$, we find
\[(D_{\kappa}\phi_1)(\nu_j) = \frac{a_j}{2}D_{\kappa}\phi_1(\kappa) + b_j(D_{\kappa}\phi_1)(\check{\kappa}i_j),\]
which by the linearity of the derivative implies
\[\nu_j - \frac{a_j}{2}\kappa - b_j\check{\kappa}i_j \in \text{Ker}(D_{\kappa}\phi_1) = \kappa\left(\R^{0,n} \oplus \bigwedge^{2} \R^{0,n}\right),\]
with the last equality following from Theorem \ref{OrthogonalDecompPreservationTheorem} iii). This implies $\nu_j - \frac{a_j}{2}\kappa - b_j\check{\kappa}i_j = \kappa c_j$ for some $c_j \in \mathcal{B}^n$, and so $\nu_j = \kappa\left(c_j+\frac{a_j}{2}\right) + b_j\check{\kappa}$, satisfying ii).
\qed
\\\\
With this in hand, we can prove the equivariance of $\Phi_1$.
\begin{theorem} 
\label{Phi1Equivariance}
 The map $\Phi_1$ is equivariant with respect to the $SL(2,\lip)$ actions on $S\lip$ and $\mathcal{MF}^n$.
\end{theorem}
\noindent \textit{Proof.}
Given $\Phi_1(\kappa)=[[\phi_1(\kappa);(D_{\kappa}\phi_1)(\check{\kappa}i_1),(D_{\kappa}\phi_1)(\check{\kappa}i_2),...,(D_{\kappa}\phi_1)(\check{\kappa}i_n)]]$, act on it with an element $A \in SL(2,\lip)$:
\begin{align*}
A.\Phi_1(\kappa) &= [[A.\phi_1(\kappa);A.(D_{\kappa}\phi_1)(\check{\kappa}i_1),A.(D_{\kappa}\phi_1)(\check{\kappa}i_2),...,A.(D_{\kappa}\phi_1)(\check{\kappa}i_n)]]\\
&=[[\phi_1(A\kappa);(D_{A\kappa}\phi_1)(A(\check{\kappa}i_1)),(D_{A\kappa}\phi_1)(A(\check{\kappa}i_2)),...,(D_{A\kappa}\phi_1)(A(\check{\kappa}i_n))]].
\end{align*}
As $A$ preserves the bracket, 
\begin{equation}
\label{eqn:AKappaBracket}
\{A\kappa,A(\check{\kappa}i_j)\}=\{\kappa,\check{\kappa}i_j\}=-|\kappa|^2i_j.
\end{equation}
Now we apply Lemma \ref{NuZjLemma}; taken together with (\ref{eqn:AKappaBracket}), it implies the $A\check{\kappa}$ will generate the same multiflag as the $\widecheck{A \kappa}$.
\begin{align*}
[[\phi_1(A\kappa);(D_{A\kappa}\phi_1)(A(\check{\kappa}i_1)),...,(D_{A\kappa}\phi_1)(A(\check{\kappa}i_n))]] &= [[\phi_1(A\kappa);(D_{A\kappa}\phi_1)(\widecheck{A\kappa}i_1),...,(D_{A\kappa}\phi_1)(\widecheck{A\kappa}i_n)]]\\
&=\Phi_1(A\kappa).\tag*{\qed}
\end{align*}
With the domain and codomain of $\Phi_1$ properly understood, we can consider injectivity and surjectivity. 
\begin{theorem}
\label{thm:Phi1SmoothDoubleCover}
 The map $\Phi_1: S\lip \to \mathcal{MF}^n$ is a smooth double cover.
\end{theorem}
\noindent\textit{Proof.} Given a point $F \in \mathcal{MF}^n$, $F = [[p;v_1,v_2,...,v_n]]$, we have by definition of $\mathcal{MF}^n$ that $p \in L^+$, but $L^+ = \text{Im}(\phi_1)$ by Lemma \ref{Phi1SurjectLightConeLemma}, so there is a non-empty set of spinors $\kappa_p$ which map to $\Phi_1(\kappa_p) = [[p;...]]$. Consider $\Phi_1(\kappa)$ for $\kappa = (\xi,0)$, $|\xi|=1$. All such $\kappa$ have base-point $\phi_1(\kappa)=(1,1;0)_{\R} \in L^+$, and they can all be found by acting on a base spinor $\kappa_0 = (1,0)$.
\[\kappa = A\kappa_0 = \begin{pmatrix}
    \xi & 0 \\
    0 & \xi'
\end{pmatrix}
\begin{pmatrix}
    1 \\
    0
\end{pmatrix},\ A \in SL(2,\lip).\]
$\kappa_0$ is sent to the multiflag $\Phi_1(\kappa_0) = [[(1,1;0)_{\R};\partial_{X_1},\partial_{X_2},...,\partial_{X_n}]]$ (the computation follows from (\ref{eqn:Derivs})), and by the equivariance of $\Phi_1$ shown in Theorem \ref{Phi1Equivariance} we have 
\[\Phi_1(\kappa) = A.\Phi_1(\kappa_0) = [[(1,1;0)_{\R};A.\partial_{X_1},A.\partial_{X_2},...,A.\partial_{X_n}]].\] 
If we check how $A$ acts on these vector fields, it is by the $\sigma(\cdot)$ action, which is just a rotation:
\[A.\partial_{X_j} = A\partial_jA^{\dagger} 
= \begin{pmatrix}
    0 & \xi i_j\xi^* \\
    -\xi'i_j\ol{\xi} & 0
\end{pmatrix} = 
\begin{pmatrix}
    0 & \sigma(\xi)(i_j) \\
    \ol{\sigma(\xi)(i_j)} & 0
\end{pmatrix}.\]
$\sigma(\cdot)$ is a homomorphism of $\lip$ onto the group of rotations of $\para$, with a fibre $\R^+$ that can be partially recovered via the norm $N(\cdot)$. As $|\xi|=1$, the only two preimages of a particular rotation will be $\xi,-\xi$. 

The set of $\xi \in \lip$ satisfying $|\xi|=1$ is just the spin group $Spin(n+1)$; all rotations $SO(n+1)$ can be given as $\sigma(x)$ for $x \in Spin(n+1)$, so all positively oriented orthonormal frames can be reached from $\kappa_0$, and thus all multiflags based at $(1,1;0)_{\R}$ are mapped to by some spinor $\kappa$ by rotating $\kappa_0$ appropriately. 

As $SL(2,\lip)$ acts transitively on $L^+$ (and, as we've just seen, on the space of multiflags based at $(1,1;0)_{\R}$), it acts transitively on $\mathcal{MF}^n$, we can send any base-point to any other, and any flag based at one point to any other, implying $\Phi_1$ is surjective. $\Phi_1(\kappa_1) = \Phi_1(\kappa_2)$ if and only if $\kappa_2 = \pm\kappa_1$, implying the map is a double cover. Smoothness follows because the various maps are either defined in terms of polynomial functions, or derivatives of smooth functions.\qed
\\\\
Topologically, the map $\Phi_1:\ S\lip \to \mathcal{MF}^n$ is a double covering map from a fibre bundle locally modelled on $Spin(n+1)\times S^{n+1}\times \R \to SO(n+1) \times S^{n+1} \times \R$.

\subsection{Example Computation of a Multiflag}
\label{ExampleComputation}
The exponential map sends $\R^{0,n}$ to the unit sphere in $\para$ (see Proposition \ref{EulerGeneralisationLemma}), which sheds some light on why paravectors are a natural object to consider; just as the vectors $\R^{0,n}$ are distinguished in the Lie algebra $\clif$, their exponentials (the unit paravectors) are distinguished in the Lie group $\clif^{\times}$. 

Take $\kappa = (e^{V \theta},0)$ for $V \in \R^{0,n}$ of unit $N$-norm, and $\theta \in \R$. 
%\[\phi_1(\kappa) = \begin{pmatrix}
%1 & 0 \\
%0 & 0
%\end{pmatrix} = \frac{1}{2}
%\begin{pmatrix}
% T+Z & \pX \\
% \ol{\pX} & T-Z 
%\end{pmatrix},\]
%or equivalently $\phi_1(\kappa)$ is the point\footnote{Note that not all flags based at $(1,1;0)_{\R}$ are given by a $\kappa$ of this form.} $(1,1;0)_{\R} \in \R^{1,n+2}$, with $T=Z=1$ and the other components all zero. 
The various derivatives that define our flag planes can be computed to have the following matrix form: %are
\[(D_{\kappa}\phi_1)(\check{\kappa}i_j) = \begin{pmatrix}
 0 & i_j-2V_j\sin \theta e^{V \theta}\\
 -i_j-2V_j\sin \theta e^{-V \theta} & 0
\end{pmatrix}.\]
We can easily read off that $T=Z=0$, but to compute the remaining terms in the standard coordinates $(T,Z;X_0,X_1,...,X_n)_{\R}$ we expand the $(1,2)$ entry out further.
\[i_j-2V_j\sin \theta e^{V \theta}= -2V_j\sin \theta \cos \theta - 2V_j\sin^2\theta V + i_j.\]
Expand $V $ as $\sum_{k=1}^nV_ki_k$ and apply a few double angle formulae.
\[i_j-2V_j\sin \theta e^{V \theta} = -V_j\sin(2 \theta) + V_j(\cos(2\theta)-1) \left(\sum_{k=1}^nV_ki_k\right) + i_j.\]
Clearly the term $X_0 = -V_j\sin(2\theta)$, while for all other $k \in \{1,...,n\},\ k \neq j$ we have $X_k = V_jV_k(\cos(2\theta)-1)$. The remaining term is $X_j = V_j^2(\cos(2\theta)-1)+1$, and so our flag plane vectors $(D_{\kappa}\phi_1)(\check{\kappa}i_j)$ can be explicitly written as
\[(0,0;-V_j\sin(2\theta),V_1V_j(\cos(2\theta)-1),V_2V_j(\cos(2\theta)-1),...,V_j^2(\cos(2\theta)-1)+1,...,V_nV_j(\cos(2\theta)-1))_{\R}.\]
As with Penrose and Rindler's null flags \cite{SpinorsAndSpacetime}, rotating the spinor by $\theta$ rotates the subflags by $2\theta$. 

\section{From Multiflags to Horospheres}
\label{sec:Third}
The multiflag (or, equivalently, decorated ideal point) structure is one way to represent a Lipschitz spinor as a geometric object; it is the natural extension of the null flag of Penrose and Rindler (respectively, their picture of a spin-vector as a tangent vector to $\mathscr{S}^+$). But we can go further, and following \cite{Mathews_Spinors_horospheres}, \cite{mathews2024quaternions} extend the map into hyperbolic space. 

The obvious model to employ is the hyperboloid model of hyperbolic space living `inside' the light-cone. Following \cite{Mathews_Spinors_horospheres} and \cite{PennerPunctured} we map flagpole vectors to horospheres, which in the hyperboloid model are given by the intersection of particular hyperplanes in $\R^{1,n+2}$ with the hyperboloid. We can then intersect the multiflag with the horospheres (or more precisely their tangent spaces); these intersections generate oriented line fields in the Euclidean geometry of the horosphere. These line fields have some nice properties, namely being parallel and orthogonal. 

These constructions exploit the close relationship between the hyperboloid model and the light-cone, but our final destination is the upper half-space model; it is there that the nicest explicit forms for the correspondence are found.
\begin{definition}[Coordinate conventions]
\label{defn:HypCoords}
The hyperboloid model $\hyp^{n+2}$ lives in the generalised Minkowski space $\R^{1,n+2}$. Similar to the coordinate notation introduced at the top of the previous section, we can write coordinates
\[(T,Z;X_0,X_1,...,X_n)_{\hyp} \text{ or } (T,Z;\pX)_{\hyp}, \text{ with } T^2-Z^2-|\pX|^2 = 1,\] 
for the hyperboloid. As this is a submanifold of $\R^{1,n+2}$ we could simply use the subscript $\R$ as in the prior section, but the subscript $\hyp$ helps recall our focus on hyperbolic space, and connects more naturally to the variant subscripts $\partial \hyp$ and $T\hyp$ defined below.

\label{not:BallModel}The ball model $\B^{n+2}$ can be identified (as a topological space) with the unit\footnote{under the metric $N(\pV) = \pV\ol{\pV}$.} ball of $\$\R^{n+2}$, just as the disc model can be identified with the unit disc in $\C$ by applying a new metric. For later convenience we also identify the natural copy of $\para \subset \$\R^{n+2}$ found by setting the $i_{n+1}$ coefficient to zero, and coordinatise $\$\R^{n+2}$ as $y_0+\sum_{j=1}^ny_ji_j + wi_{n+1}$; \label{not:BallCoords}we then use the coordinates
\[(w;y_0,y_1,...,y_n)_{\mathbb{B}} \text{ or } (w;\textbf{y})_{\mathbb{B}},\ \textbf{y}\in\para,\ w^2+|\textbf{y}|^2 < 1 \text{ for the ball model.}\]
\label{not:UpperHalf2}Finally, the upper half-space model is identified as $\U^{n+2} \subset \$\R^{n+2}$, where the boundary plane is identified with $\para$. \label{not:UpperHalfCoords}We coordinatise $\$\R^{n+2}$ as $x_0 + \sum_{j=1}^nx_ji_j + zi_{n+1}$, and we write
\[(z;x_0,x_1,...,x_n)_{\U} \text{ or } (z;\textbf{x})_{\U},\ z>0, \text{ for the upper half-space model.}\]
\label{not:BdryCoords}We use the same coordinate layout with subscript $\partial\hyp$, $\partial\mathbb{B}$, or $\partial\U$ for the boundary of the hyperboloid, ball, and upper half-space models respectively; the boundary of the hyperboloid is given by $\mathscr{S}^+$ (the space of rays on the light-cone), the boundary of $\B^{n+2}$ is given by the unit length paravectors under the positive-definite metric $N(\pV) = \pV\ol{\pV}$ on $\$\R^{n+2}$, and the boundary in $\U^{n+2}$ is given by the points $(0;\textbf{x})_{\U} \cup \{\infty\}$. As the boundary of $\U$ can be identified with $\para \cup \{\infty\}$, we often simply write a paravector $x_0 + \sum_{j=1}^n x_ji_j$ for a coordinate in $\partial \U$.

\label{not:TanCoords}We also define coordinates with subscripts $T\hyp$, $T\mathbb{B}$, or $T\U$ for the tangent spaces of each model. The arrangement of coordinates is identical to their manifold counterparts.\null\hfill\defn
\end{definition}
\subsection{Mapping Spinors to the Hyperboloid Model}
Define $\hyp^{n}$ to be the positive sheet of the hyperboloid given by points $x \in \R^{1,n}$ satisfying $(x|x)_{1,n} = 1$, and endow it with a metric by restricting the Minkowski metric. This gives a model of hyperbolic space called the \textit{hyperboloid model}.
\begin{definition}
    A {\normalfont horosphere} in hyperbolic $n$-space is the limit as $r \to \infty$ of a series of $(n-1)$-spheres of radius $r$ which are tangent at a fixed point, and on a fixed side, of a given hyperplane. When $n=2$, we may also call this a {\normalfont horocycle}. \hfill\defn
\end{definition}
\noindent In the hyperboloid model, horospheres are given by the intersection of the hyperboloid with a hyperplane that has null normal vector;\footnote{A normal vector of norm zero under $(\cdot|\cdot)_{1,n}$.} points of the light-cone parameterise this space, as the light-cone is the space of null vectors (see \cite{HorospheresInHypGeo} for more discussion on this). The first stage of our map will send flagpoles to a hyperplane they define as a normal vector, and then to horospheres defined by the intersection of this hyperplane with the hyperboloid model.
\begin{definition}
 \label{not:HorN}Denote the space of horospheres in $n$-dimensional hyperbolic space as Hor$^n$, regardless of model. \label{not:phi2}The map $\phi_2:\ L^+ \to$ Hor$^n$ sends $p \in L^+$ to the horosphere defined by the intersection of the hyperplane $\Pi_p = \{x \in \R^{1,n+2}\ |\ (x|p)_{1,n} = 1\}$, and the hyperboloid $\hyp^{n+2} = \{x \in \R^{1,n+2}\ |\ (x|x)_{1,n} = 1\}$. 
 \label{not:phi}We also define $\phi = \phi_2 \circ \phi_1$.
 
 \label{not:phi2partial}The map $\phi_2^{\partial}:\ L^+ \to \partial \hyp^n$ instead sends $p$ to the point at infinity of $\phi_2(p)$ (also called the {\normalfont centre} of the horosphere). \label{phipartial}We also define $\phi^{\partial} = \phi_2^{\partial} \circ \phi_1$. \hfill\defn
\end{definition}
\noindent Visualising the boundary at infinity is tricky in the hyperboloid model. However, we already have a good representation; the boundary can be defined as the space of rays on the light-cone, which is our celestial sphere $\mathscr{S}^+$. The centre of the horosphere $\phi_2(p)$ is then represented by the ray $\R^+ p$ through the point $p$ (see \cite{PennerPunctured} for details). 

To consider questions of equivariance, we need to define an action of $SL(2,\lip)$ on the hyperboloid. The obvious choice is to inherit the action on generalised Minkowski space (see Section \ref{sec:ActionsofSL2OnMinkowski}).
\begin{definition}
\label{def:ActionOnHyperbolic}
The action of $SL(2,\lip)$ on $\hyp^{n+2}$ is the restriction of the conjugation action defined on $\mathcal{P}^{1,n+2}$. When we introduce other models of hyperbolic space (say a model $M^{n+2}$), the action on $M^{n+2}$ is then inherited through an isometry $\pi: \hyp^{n+2} \to M^{n+2}$. That is, the action of $A \in SL(2,\lip)$ on some $x \in M^{n+2}$ is defined as $\pi\left(A.\pi^{-1}(x)\right)$. It is equivariant by definition. As it acts on $\R^{1,n+2}$ by isometries (and, as we shall see, maps $\hyp^{n+2} \to \hyp^{n+2}$), this action also maps null planes to null planes and thus horospheres to horospheres, and so defines an action on Hor$^{n+2}$. \null\hfill\defn
\end{definition}
\noindent We must check that this action does restrict to an action on the hyperboloid model. We know from Lemma \ref{ActionIsometryMink} that it preserves the metric, so it does map the hyperboloid to itself; but it must also preserve the positive sheet.
\begin{lemma}
 The action of $SL(2,\lip)$ on $\R^{1,n+2}$ restricts to an action by orientation-preserving isometries on the hyperboloid model $\hyp^{n+2}$.
\end{lemma}
\noindent\textit{Proof.} All we need is to show the action restricts to a map $\hyp^{n+2} \to \hyp^{n+2}$. To show this, we use the connectedness of $SL(2,\lip)$; because the identity matrix maps the positive sheet to itself, all elements of the connected component containing the identity map also preserve the positive sheet. But the connected component containing the identity is the whole of $SL(2,\lip)$.\qed
\\\\
With an action defined, we can check the equivariance of $\phi_2$. 
\begin{lemma}
\label{phi2equivariant}
 The maps $\phi_2$ and $\phi_2^{\partial}$ are equivariant under the action of $SL(2,\lip)$.
\end{lemma}
\noindent\textit{Proof.} $A \in SL(2,\lip)$ acts as an isometry on the hyperboloid, as implied by Lemma \ref{ActionIsometryMink}. Explicitly, given points $x,y \in \hyp^{n+2}$, $(Ax|Ay)_{1,n+2} = (x|y)_{1,n+2}$. A horosphere $\phi_2(p)$ is defined by the hyperplane $\{x \in \R^{1,n+2}\ |\ (x|p)_{1,n+2} = 1\}$ where it intersects the hyperboloid. The action of $A$ then gives us the horosphere $A.\phi_2(p)$ whose points $x \in \hyp^{n+2}$ satisfy $(A^{-1}x|p)_{1,n+2}=1$. Because $A$ is an isometry, $A.\phi_2(p)$ is defined equivalently by the hyperplane with defining condition
\[1 = (A^{-1} x| p)_{1,n+2} = (AA^{-1} x| Ap)_{1,n+2} = ( x|Ap)_{1,n+2},\]
which is the defining hyperplane for $\phi_2(Ap)$. 

Similarly, $\phi_2^{\partial}$ is equivariant as we can just drop all information about the horospheres except their centres. \qed

\subsubsection{Flags to Line Fields}
Consider a multiflag as extending through Minkowski space and intersecting with our newly-found horosphere. These intersections will be straight lines in the Euclidean geometry of the horosphere. Foliating Minkowski space by parallel copies of the flag-planes, this will generate an oriented line field on the horosphere. 

Such a line field, strictly speaking, lives in the tangent bundle to the horosphere, so we need a description of the tangent spaces of horospheres. The tangent space to a horosphere is a subspace of both the tangent space of the hyperboloid model, and the tangent space of the defining hyperplane for that horosphere, so it lives in the intersection of those two tangent spaces.

We know from Lemma \ref{SphereTangentSpaceLemma} that the tangent space at a point $q$ on the hyperboloid is $q^{\perp}$, so we look for the tangent space to the hyperplane.
\begin{lemma}
\label{lem:HyperplaneTangentSpace}
 In a quadratic vector space $V$ with metric $( \cdot,\cdot)$, the tangent space to the hyperplane $\{x \in V\ |\ (x,p) = k\}$ where $p \in V$, $k \in \R$ is given by $p^{\perp}$.
\end{lemma}
\noindent\textit{Proof.} Similar to Lemma \ref{SphereTangentSpaceLemma}, we construct a tangent vector as a derivative of some curve $q_t$ in the hyperplane. $(q_t,p) = K \implies (\dot{q}_t,p) = 0$, so $\dot{q}_t \in p^{\perp}$, and as the dimensions of the tangent space and the orthogonal complement agree they are in fact equal. \qed
\\\\
The tangent space at the point $q$ of the horosphere $H_p$ defined by the hyperplane 
\[\Pi_p = \{x \in \R^{1,n+2}\ |\ (x|p)_{1,n+2}=1\}\] 
will be an $n+1$-dimensional subspace of the intersection $q^{\perp} \cap p^{\perp}$, which implies dim$(q^{\perp} \cap p^{\perp}) \geq n+1$. Both $q^{\perp}$ and $p^{\perp}$ are $n+2$-dimensional, but $q^{\perp}$ is spacelike, while $T_q\Pi_p$ contains the light-like direction $\R p$ by Lemma \ref{lem:HyperplaneTangentSpace}. So the intersection must be precisely $n+1$-dimensional, and the tangent space of the horosphere at $q$ is exactly $q^{\perp} \cap p^{\perp}$. 

We can now define line fields by intersecting the flag planes of $\Phi_1(\kappa)$ with the tangent spaces of $\phi(\kappa)$. The intersection of the $i_j$ subflag will define the $i_j$ \textit{line field}.
\begin{lemma}
\label{lem:FlagsToLineFields}
The $n$ subflags of the multiflag $\Phi_1(\kappa)$ intersect the horosphere $\phi(\kappa)$ to form $n$ linearly independent oriented lines. By taking the intersection with the family of parallel translates of the multiflag across $\R^{1,n+2}$, we instead get $n$ linearly independent oriented line fields.
\end{lemma}
\noindent\textit{Proof.} As we know the subflags only intersect one another along the flagpole $p\R$, the multiflag spans a subspace of dimension $n+1$ (including the light-like flagpole); call this subspace $F_n$. The tangent space at a point $q$ on the horosphere $\phi(\kappa)$ is given by $T_q(\phi(\kappa))=q^{\perp} \cap \phi_1(\kappa)^{\perp}$. 

$F_n$ lies within the tangent space to the light-cone, so $F_n \subset \phi_1(\kappa)^{\perp}$. The intersection of the horosphere with the subspace $F_n$ is thus
\[T_q(\phi_2(p)) \cap F_n = q^{\perp}\cap \phi_1(\kappa)^{\perp} \cap F_n = q^{\perp} \cap F_n.\]
This intersection of vector spaces is at most $n+1$ dimensional, the dimension of $F_n$; but at minimum it is $n$ dimensional, because $q^{\perp}$ is $n+2$ dimensional and $F_n$ is $n+1$ dimensional, while they are within an $n+3$ dimensional vector space $\R^{1,n+2}$. There is at least one light-like vector $\phi_1(\kappa)$ contained in $F_n$ that does not sit within the $q^{\perp}$, as $q^{\perp}$ is necessarily spacelike (since the metric restricts to be positive definite on the hyperboloid). This implies the intersection is precisely $n$ dimensional. All the subflags must therefore intersect the horosphere in at least a 1 dimensional subspace; but as each subflag has one light-like dimension, the intersection is precisely 1 dimensional. 

Say a subflag is parameterised as $f_j = p\R + v_j\R$, where $p\R$ is the flagpole. Via Section \ref{sec:OrientationConventions}, given one of these flag-planes an orientation is defined on $f_j/p\R$, which then defines an orientation on the corresponding line. Across the $n$ subflags, we then have $n$ linearly independent oriented lines. By taking parallel copies of the multiflag we foliate $\R^{1,n+2}$, and so we define a line field at every point of the horosphere.\qed

\subsubsection{The Map $\Phi_2$}
The full map from $\mathcal{MF}^n$ to hyperbolic space will send a multiflag $[[p;v_1,v_2,...,v_n]]$ to the horosphere $\phi_2(p)$ with line fields in its Euclidean geometry defined by taking all parallel translates of the subflags $[[p,v_j]]$, and intersecting them with the horosphere. Note that, when describing a line field by a vector pointing along the lines, scalar multiples of the same vector $v_j$ denote the same field; we can always normalise the vectors describing these line fields where necessary. We say such a vector \textit{directs} or \textit{specifies} the line field.
\begin{definition}
\label{def:LHor}
The space $L$Hor$^{n+2}$ is the space of horospheres in $\hyp^{n+2}$ carrying an ordered set of $n$ linearly independent oriented line fields $(L_1,...,L_n)$. The individual line field $L_j$ is referred to as the {\normalfont $i_j$ line field}. 

\label{Phi2}The map $\Phi_2: \mathcal{MF}^n \to L\text{Hor}^{n+2}$ sends a multiflag  $[[p;v_1,v_2,...,v_n]]$ to the horosphere $\phi_2(p)$ with the line field defined at each point $q$ by $T_q\phi_2(p) \cap [[p,v_j]]$ and the induced orientations. \label{not:Phi}We also define the map $\Phi = \Phi_2 \circ \Phi_1$. \null\hfill\defn 
\end{definition}
\noindent We will later find that the image of $\Phi_2$ only includes certain rather nice line fields, so the codomain will be restricted; but, for now, all we know is that it maps multiflags to horospheres with $n$ linearly independent oriented line fields defined upon them.

We define an action of $SL(2,\lip)$ on $L$Hor$^{n+2}$ by restricting its action on $\hyp^{n+2}$ (which is equivalent to an action by $SO(1,n+2)^+$). The action on hyperbolic space is by orientation-preserving isometries and sends horospheres to horospheres, and its derivative sends line fields on horospheres to line fields on horospheres.

\subsubsection{Parallel Line Fields}
The line fields defined by these intersections have another property, namely being \textit{parallel} (as demonstrated later, in Corollary \ref{cor:LineFieldsParallel}). This means nothing other than being parallel in the Euclidean sense, but it is worth making a precise definition for this context.

Isometries that generate translations along a horosphere form a subgroup of the parabolic transformations in $\hyp^n$. However, as noted in Section 3.9 of \cite{mathews2024quaternions}, in dimension 4 and above there are parabolic transformations that act on a horosphere, but do not act by translations. Combining and extending Definition 3.9.3 and Lemma 3.9.1 of that paper, we can define the translations appropriately.
\begin{definition}
\label{def:ParabolicTranslation}
 We call a matrix $A \in SL(2,\lip)$ that satisfies any of the following conditions:
 \begin{enumerate}[label=\roman*)]
 \item $A \neq 1$ and $(A-1)^2=0$,
 \item $A$ is conjugate to $P_0 = \begin{pmatrix}
 1 & 1 \\
 0 & 1 
 \\
 \end{pmatrix}$,
 \item $A = \begin{pmatrix}
 1-\xi \eta^* & \xi\xi^* \\
 -\eta \eta^* & 1+\eta\xi^*
 \end{pmatrix}$ for some $(\xi,\eta) \in S\lip$,
 \end{enumerate}
 a {\normalfont parabolic translation}. The set of parabolic translation matrices is denoted $P$. 
 
 A line field on the horosphere $H \in \text{Hor}^{n+2}$ is {\normalfont parallel} if it is invariant under the action of all parabolic translations in $SL(2,\lip)$ fixing $H$ (also referred to as {\normalfont translations on the horosphere}). \null\hfill\defn
\end{definition}
\noindent We demonstrate the 3 conditions of Definition \ref{def:ParabolicTranslation} are, in fact, equivalent. First, though, a technical lemma.
\begin{lemma}
\label{lem:VariantSLConds}
    For the Clifford matrix
    \[A = \begin{pmatrix}
        a & b \\
        c & d
    \end{pmatrix} \in SL(2,\lip),\]
    the following variations on the pseudo-determinant condition hold:
    \begin{equation}
    \label{eqn:AlternatePseudoDets}
    ad^*-bc^* = da^*-cb^* = d^*a-b^*c = 1.
    \end{equation}
\end{lemma}
\noindent\textit{Proof.} Taking the original condition on the pseudo-determinant that holds for $SL(2,\lip)$, we manipulate it to get the new forms. 
\[|d|^2\cdot 1 = |d|^2(a^*d-c^*b) = d(a^*d-c^*b)\ol{d} = da^*|d|^2-dc^*b\ol{d}.\]
By Corollary \ref{cor:xistarEta} we have $dc^*=cd^*$, and as $b\ol{d} \in \para$ we have $b\ol{d} = (b\ol{d})^* = d'b^*$.
\[|d|^2 = da^*|d|^2-cd^*d'b^* = |d|^2(da^*-cb^*).\]
This either tells us $da^*-cb^*=1$, or $|d|^2=0$. Varying the argument a little, we find
\[|c|^2 \cdot 1 = |c|^2 (a^*d-c^*b) = c'(a^*d-c^*b)c^* = c'a^*dc^*-|c|^2bc^* = |c|^2(ad^*-bc^*),\]
implying either $|c|^2 = 0$ or $ad^*-bc^*=1$. As the pseudo-determinant $a^*d-c^*b \neq 0$ we must have at least one of $c,d$ non-zero, so at least one of $da^*-cb^* = 1, \text{ or } ad^*-bc^*=1$ must hold. But if $ad^*-bc^*=1$ then $(ad^*-bc^*)^* = da^*-cb^* =1$,
and similarly if $da^*-cb^*=1$ then $1 = (da^*-cb^*)^* = ad^*-bc^*$. We conclude both `variant pseudo-determinants' are equal to $1$. The final equality of (\ref{eqn:AlternatePseudoDets}) then holds because the pseudo-determinant $a^*d-c^*b=1$ is invariant under $\cdot^*$. \qed
%\\\\
%Now to prove the conditions of Definition \ref{def:ParabolicTranslation} are equivalent.
\begin{lemma}
    The following conditions on a matrix $A \in SL(2,\lip)$ are equivalent.
    \begin{enumerate}[label=\roman*)]
 \item $A \neq 1$ and $(A-1)^2=0$,
 \item $A$ is conjugate to $P_0 = \begin{pmatrix}
 1 & 1 \\
 0 & 1 
 \\
 \end{pmatrix}$,
 \item $A = \begin{pmatrix}
 1-\xi \eta^* & \xi \xi^* \\
 -\eta \eta^* & 1+\eta \xi^*
 \end{pmatrix}$ for some $(\xi,\eta) \in S\lip$.
 \end{enumerate}
\end{lemma}
\noindent\textit{Proof.} $ii) \to iii)$: Given a Clifford matrix
\[B = \begin{pmatrix}
    a & b \\
    c & d \\
\end{pmatrix} \in SL(2,\lip),\]
we compute
\begin{align*}
B\begin{pmatrix}
1 & 1 \\
0 & 1 
\end{pmatrix}B^{-1} %= 
%\begin{pmatrix}
%    a & b \\
%    c & d \\
%\end{pmatrix}
%\begin{pmatrix}
%1 & 1 \\
%0 & 1 
%\end{pmatrix}
%\begin{pmatrix}
%    d^* & -b^* \\
%    -c^* & a^*
%\end{pmatrix} &=
%\begin{pmatrix}
%    a & b \\
%    c & d \\
%\end{pmatrix}
%\begin{pmatrix}
%    d^*-c^* & -b^* + a^* \\
%    -c^* & a^*
%\end{pmatrix}
%\\
%= \begin{pmatrix}
%    a(d^*-c^*)-bc^* & a(-b^*+a^*)+ba^* \\
%    c(d^*-c^*)-dc^* & c(-b^*+a^*) + da^*
%\end{pmatrix} 
&= \begin{pmatrix}
    ad^*-bc^*-ac^* & aa^*+ba^*-ab^* \\
    -cc^*+cd^*-dc^* & da^*-cb^* + ca^*
\end{pmatrix}.
\end{align*}
From Lemma \ref{lem:VariantSLConds} we have that $ad^*-bc^*=da^*-cb^*=1$, and we also know $ab^*=ba^*$ and $cd^*=dc^*$ from Corollary \ref{cor:xistarEta}. Upon simplifying:
\[
B\begin{pmatrix}
1 & 1 \\
0 & 1 
\end{pmatrix}B^{-1} =
\begin{pmatrix}
    1 - ac^* & a^* \\
    -cc^* & 1+ca^*
\end{pmatrix}.\]
As $B,P_0 \in SL(2,\lip)$ the rows and columns are spinors by Proposition \ref{prop:AlternateSLCharacterisation}, so $(a,c) \in S\lip$ and ii) implies iii).
\\\\
$iii) \to ii)$: If $A$ satisfies condition iii) then, by Proposition \ref{prop:KappaSLMatrix}, there is a matrix $B \in SL(2,\lip)$ with first column $(\xi,\eta) \in S\lip$, and conjugation of $P_0$ by the matrix $B$ gives $A$ by the computation above. 
\\\\
$ii) \to i)$: The matrix $A \in SL(2,\lip)$ satisfies i) if and only if any conjugate of $A$ by $SL(2,\lip)$ satisfies i). To see this, let $B \in SL(2,\lip)$:
\begin{align*}
(BAB^{-1}-1)^2 = 0 &\Leftrightarrow BA^2B^{-1}-2BAB^{-1}+1 = 0 %&\Leftrightarrow B(A^2-2A)B^{-1}=-1\\
%&\Leftrightarrow A^2-2A+1=0\\ 
\Leftrightarrow (A-1)^2=0.
\end{align*}
The matrix $P_0$ satisfies i), so ii) implies i).
\\\\
$i) \to ii)$: Suppose the matrix $A = \begin{pmatrix}
    \alpha & \beta \\
    \gamma & \delta
\end{pmatrix}$ satisfies i). We show $A$ is conjugate to either an upper or lower triangular matrix with both diagonal entries $1$.

If $\gamma = 0$ then, as $A \in SL(2,\lip)$, the restriction on the pseudo-determinant implies $\delta = \alpha^{*-1}$. The $(1,1)$ entry of the equation $(A-1)^2=0$ is $(\alpha-1)^2=0$, so $\alpha=\delta=1$ and $A$ is an upper triangular matrix with $1$ along the diagonal as claimed. Note that, as the second column must form a Lipschitz spinor, we have $\beta \in \para$.

If $\gamma \neq 0$, then by \cite{Ahlfors_1985B} equation (15) we can conjugate $A$ into the `normal form'
\[N_A = \begin{pmatrix}
    \pV\gamma & \pV\gamma \pV -\gamma^{*-1} \\
    \gamma & \gamma \pV
\end{pmatrix},\]
where $\pV \in \para$. As above, $A$ satisfies i) if and only if any conjugate of $A$ does; we compute the lower left entry of
\[(N_A-1)^2 = \begin{pmatrix}
    \pV\gamma -1& \pV\gamma \pV -\gamma^{*-1} \\
    \gamma & \gamma \pV -1
\end{pmatrix}^2,\]
and use $(N_A-1)^2=0$ to find the equation $\gamma(\pV\gamma-1) + (\gamma\pV -1)\gamma = 2\gamma(\pV\gamma - 1) = 0$. We've chosen $\gamma \neq 0$ so this implies $\pV\gamma = 1$ and $\pV = \gamma^{-1}$, so $\gamma$ is a paravector and
\[N_A = \begin{pmatrix}
    1 & 0 \\
    \gamma & 1
\end{pmatrix}.\]
$A$ is conjugate to a lower triangular matrix with diagonal entries $1$, as claimed. 

This shows an $A$ satisfying i) is conjugate to one of the two following forms:
\[N_A = \begin{pmatrix}
    1 & \beta \\
    0 & 1
\end{pmatrix},\ \beta \in \para \text{  or  }
N_A = \begin{pmatrix}
    1 & 0 \\
    \gamma & 1
\end{pmatrix},\ \gamma \in \para.\]
We know $\lip$ acts on $\para$ by $\sigma$ as the group of rotations plus a homothety, so we can send $1$ to any other paravector in $\para$. Thus there is either an element $\rho_{\beta} \in \lip$ such that 
\begin{equation}
\label{eqn:SqrtParavector}
\sigma(\rho_{\beta})(1) = \rho_{\beta}\rho_{\beta}^* = \beta
\end{equation}
in the first case, or an element $\rho_{\gamma} \in \lip$ such that $\sigma(\rho_{\gamma})(1) = \rho_{\gamma}\rho_{\gamma}^* = -\gamma$ in the second. This lets us write $N_A$ as either
\[N_A = \begin{pmatrix}
    1 & \rho_{\beta}\rho_{\beta}^* \\
    0 & 1
\end{pmatrix} \text{  or  }
N_A = \begin{pmatrix}
    1 & 0 \\
    -\rho_{\gamma}\rho_{\gamma}^* & 1
\end{pmatrix},\]
which we recognise as the form of iii) for either the spinor $(\rho_{\beta},0) \in S\lip$ or $(0,\rho_{\gamma}) \in S\lip$. Then, by the previously proven equivalence of iii) with ii), this implies $N_A$ is conjugate to $P_0$, and as $A$ is conjugate to $N_A$ it is also conjugate to $P_0$ as claimed.
\qed
\\\\
Now we find these parabolic translations in an explicit form for a given spinor $\kappa$.
\begin{lemma}
 \label{ParabolicSubgroupDefinition}
The subgroup of parabolic translations fixing a Lipschitz spinor $\kappa = (\xi,\eta)$ is given by
\begin{equation}
\label{eqn:ParabolTrans}
P^{\kappa} = \left\{
\begin{pmatrix}
 1-\xi\pV\eta^* & \xi\pV\xi^* \\
 -\eta\pV\eta^* & 1+\eta\pV\xi^* 
\end{pmatrix}\ \middle|\ \pV \in \para \right\}.
\end{equation}
For a fixed $\pV \in \para$, the corresponding group element is denoted $P^{\kappa}_{\pV }$. This provides an explicit isomorphism between $P^{\kappa}$ and $\para$.
\end{lemma}
\noindent\textit{Proof.} 
The subgroup of parabolic translations fixing $(1,0)$ takes the form
\[P^{(1,0)} = \left\{\begin{pmatrix}
    1 & \pV \\
    0 & 1
\end{pmatrix}\ \middle| \ \pV \in \para\right\}.\]
It is simple to check these matrices fix $(1,0)$; for the converse, take an element $A = \begin{pmatrix}
    a & b \\
    c & d
\end{pmatrix} \in SL(2,\lip)$ fixing $(1,0)$:
\[\begin{pmatrix}
    a & b \\
    c & d
\end{pmatrix} \begin{pmatrix}
    1 \\
    0
\end{pmatrix} = \begin{pmatrix}
    1 \\
    0
\end{pmatrix},\]
and we read off the equations $a=1,\ c=0$. The pseudo-determinant condition for $SL(2,\lip)$ then imposes $a^*d-c^*b = d = 1$.
But $(a,b) = (1,b)\in S\lip$ by Proposition \ref{prop:AlternateSLCharacterisation}, which implies $b \in \para$ and $A \in P^{(1,0)}$. This implies $P^{(1,0)}$ is the maximal subgroup fixing $(1,0)$, and therefore the maximal subgroup of parabolic translations fixing $(1,0)$.

For an arbitrary spinor $\kappa = (\xi,\eta) \in S\lip$ we conjugate by the matrix of (\ref{eqn:FirstSpinorSLEmbed}) to send $\kappa$ to $(1,0)$, and act by $P^{(1,0)}$.
\[\begin{pmatrix}
 \xi & -\frac{1}{2}\eta^{*-1} \\
 \eta & \frac{1}{2}\xi^{*-1}
\end{pmatrix}
\begin{pmatrix}
 1 & \pV \\
 0 & 1
\end{pmatrix}
\begin{pmatrix}
 \frac{1}{2}\xi^{-1} & \frac{1}{2}\eta^{-1} \\
 -\eta^* & \xi^*
\end{pmatrix}
\begin{pmatrix}
 \xi \\
 \eta
\end{pmatrix} =
\begin{pmatrix}
\xi \\
\eta
\end{pmatrix}.\]
The associated matrix subgroup $P^{\kappa}$ is then
\begin{equation}
\label{eqn:ParabolicTranslationsGrp}
\begin{pmatrix}
 \xi & -\frac{1}{2}\eta^{*-1} \\
 \eta & \frac{1}{2}\xi^{*-1}
\end{pmatrix}
\begin{pmatrix}
 1 & \pV \\
 0 & 1
\end{pmatrix}
\begin{pmatrix}
 \frac{1}{2}\xi^{-1} & \frac{1}{2}\eta^{-1} \\
 -\eta^* & \xi^*
\end{pmatrix} = 
\begin{pmatrix}
 1-\xi\pV\eta^* & \xi\pV\xi^* \\
 -\eta\pV\eta^* & 1+\eta\pV\xi^* 
\end{pmatrix}.
\end{equation}
If there were an element $k \notin P^{\kappa}$ fixing $\kappa$, then by conjugation we would find an element $\tilde{k}$ fixing $(1,0)$ not in $P^{(1,0)}$, a contradiction. 

It is clear that the map $\pV \to P^{\kappa}_{\pV}$ is surjective, and a simple matrix multiplication shows it is a homomorphism, so to prove this is an isomorphism we only need consider injectivity. Let $P^{\kappa}_{\pV_1} = P^{\kappa}_{\pV_2}$, and assume $\xi \neq 0$. Then the $(1,2)$ entry of equating these matrices gives $\xi\pV_1\xi^* = \xi\pV_2\xi^*$, but as $\xi \neq 0$ both $\xi$ and $\xi^*$ are elements of $\lip$ and are thus invertible, so $\pV_1 = \pV_2$. If $\xi = 0$ then $\eta \neq 0$ and we can make a similar argument from equating the $(2,1)$ entries instead.\qed
\\\\
\begin{remark}
Parabolic translation matrices can be characterised by Definition \ref{def:ParabolicTranslation} iii):
    \[A = \begin{pmatrix}
        1-\xi\eta^* & \xi\xi^* \\
        -\eta\eta^* & 1+\eta\xi^*
    \end{pmatrix} \text{ for some } (\xi,\eta) \in S\lip.\]
    Compared to this, the matrices
    \[P^{\kappa} = \left\{
\begin{pmatrix}
 1-\xi\pV\eta^* & \xi\pV\xi^* \\
 -\eta\pV\eta^* & 1+\eta\pV\xi^* 
\end{pmatrix}\ \middle|\ \pV \in \para \right\}\]
of (\ref{eqn:ParabolTrans}) (a special family of parabolic translation matrices) seem to have an extra parameter. Of course, for $P^{\kappa}$ we have $\xi,\eta$ fixed, and by a similar argument to that preceeding (\ref{eqn:SqrtParavector}) we will have $\pV = \alpha\alpha^*$ for some $\alpha \in \lip$. Then the matrices of $P^{\kappa}$ will take the form
\[\begin{pmatrix}
 1-(\xi\alpha)(\eta\alpha)^* & (\xi\alpha)(\xi\alpha)^* \\
 -(\eta\alpha)(\eta\alpha)^* & 1+(\eta\alpha)(\xi\alpha)^* 
\end{pmatrix},\]
which satisfies Definition \ref{def:ParabolicTranslation} iii) for the spinor $(\xi \alpha, \eta \alpha) \in S\lip$.
\end{remark}
\noindent A horosphere $H$ can be parameterised by taking a point $q \in H$ and looking at its image under the action of the subgroup of parabolic translations fixing $H$. We can show this for the case of $\kappa=(1,0)$ and handle all other cases by conjugation as above.
\begin{proposition}
 \label{ParabolicLemma}
 The group $P^{(1,0)}$ of parabolic translations fixing $(1,0)$ acts simply transitively on the horosphere $H_0=\phi(1,0)$.
\end{proposition}
\noindent \textit{Proof.} To show transitivity, it is enough to show every point $h \in H_0$ can be sent to the point $q_0=(1,0;0,...,0)_{\hyp} \in H_0$; we can then invert and multiply elements of $P^{(1,0)}$ to make a map sending any point to any other. 

Let
\[h = \frac{1}{2}\begin{pmatrix}
T+Z & \pX \\
\ol{\pX} & T-Z
\end{pmatrix} \in H_0.\]
By definition it satisfies $(h|\phi_1(1,0))_{1,n+2} = T-Z = 1$. We apply the parabolic transformation with $\pV =-\ol{X}$:
\begin{align*}
P^{(1,0)}_{-\pX}.h %&=\frac{1}{2}\begin{pmatrix}
% 1 & -\pX \\
% 0 & 1
%\end{pmatrix}
%\begin{pmatrix}
%T+Z & \pX \\
%\ol{\pX} & 1
%\end{pmatrix}
%\begin{pmatrix}
% 1 & 0 \\
% -\ol{\pX} & 1
%\end{pmatrix}\\
&= \frac{1}{2}
\begin{pmatrix}
(T+Z)-|\pX|^2 & 0 \\
0 & 1
\end{pmatrix}.
\end{align*}
As the point $h$ is in the hyperboloid model and $T-Z=1$ we know $|h|^2 = 4\text{pdet}(h) = (T+Z)-|X|^2=1$, which tells us $h$ was sent to $q_0$, as desired.
\[P^{(1,0)}_{-\pX}.h = \begin{pmatrix}
 1 & 0 \\
 0 & 1
\end{pmatrix}=q_0.\]
To prove the action is free, consider the action of an arbitrary element $P^{(1,0)}_{\pV } \in P^{(1,0)}$ on $h \in H_0$: 
\begin{align*}
P^{(1,0)}_{\pV }.h %&= \frac{1}{2}\begin{pmatrix}
% 1 & \pV \\
% 0 & 1
%\end{pmatrix}
%\begin{pmatrix}
%T+Z & \pX \\
%\ol{\pX} & 1
%\end{pmatrix}
%\begin{pmatrix}
% 1 & 0 \\
% \ol{\pV } & 1
%\end{pmatrix}\\
&= \frac{1}{2}
\begin{pmatrix}
(T+Z)+\pX\ol{\pV } + \ol{\pX\ol{\pV }} + |\pV |^2 & \pX+\pV \\
\ol{\pX} + \ol{\pV } & 1
\end{pmatrix}.
\end{align*}
For $h$ to be a fixed point of the transformation, we require
\[\begin{pmatrix}
T+Z & \pX \\
\ol{\pX} & T-Z
\end{pmatrix} = \begin{pmatrix}
(T+Z)+\pX\ol{\pV } + \ol{\pX\ol{\pV }} + |\pV |^2 & \pX+\pV \\
\ol{\pX} + \ol{\pV } & 1
\end{pmatrix},\]
and equating the $(1,2)$ entries implies $\pV =0$ and $P^{(1,0)}_{\pV }$ is the identity transformation. \qed

\begin{corollary}
    The group $P^{\kappa}$ acts simply transitively on the horosphere $H_{\kappa} = \phi(\xi,\eta)$.
\end{corollary}
\noindent\textit{Proof.} The group $P^{\kappa}$ is just a conjugation of $P^{(1,0)}$. Factored as in the left-hand side of (\ref{eqn:ParabolicTranslationsGrp}), it maps $(\xi,\eta)$ to $(1,0)$, acts by $P^{(1,0)}$, then maps it back to $(\xi,\eta)$. By Lemma \ref{phi2equivariant} and (\ref{eqn:EquivarOfphi1}) the map $\phi = \phi_2 \circ \phi_1$ is equivariant with respect to the $SL(2,\lip)$ actions, so $A$ also maps $\phi(\xi,\eta)$ to $\phi(1,0)$ bijectively. By Proposition \ref{ParabolicLemma} the action of $P^{(1,0)}$ on $\phi(1,0)$ is simply transitive, which taken with the bijective maps between the horospheres implies $P^{\kappa}$ acts simply transitively on $\phi(\xi,\eta)$. \qed

\subsection{The Upper Half-Space Model}
\label{sec:UpperHalfspace}
For explicit computations, the upper half-space model
\[\U^{n+2} = \{(z;x_0,x_1,...,x_n)_{\U} \in \R^{n+2}\ |\ z>0\} \text{ with metric } ds^2 = \frac{dz^2+dx_0^2+dx_1^2+...+dx_n^2}{z^2}\]
gives clean expressions for the spinor-horosphere correspondence and the action of $SL(2,\lip)$. We identify\footnote{Thess identifications follows naturally from \cite{Ahlfors1985}, where he similarly identifies the upper half-space and its boundary with paravector spaces to allow the action of $SL(2,\lip)$ to extend to hyperbolic space.} $\R^{n+2}$ with $\$\R^{n+2}$, the hyperplane $z=0$ with $\para \subset \$\R^{n+2}$ (treating $z$ as the coefficient of $i_{n+1}$), and $\partial\U^{n+2}$ with $\cpara$. The coordinates on the boundary $(x_0,x_1,...,x_n)_{\partial\U}$ can then be given as a paravector $\textbf{x} = x_0+\sum_{j=1}^nx_ji_j \in \para$. 

Horospheres in this model either appear as horizontal hyperplanes if centred at $\infty$ (in which case we refer to the z-coordinate of the hyperplane as the \textit{Euclidean height}) or as $(n+1)$-spheres tangent at points on the boundary; that point of tangency is the centre of the horosphere. Call the maximum value of $z$ on such an $(n+1)$-sphere the \textit{Euclidean diameter} of the horosphere.

\subsubsection{Maps Between Models}
\label{sec:MapsBetweenModels}
We can map the hyperboloid model to the conformal ball model $\mathbb{B}^{n+2}$ via projection:
\[\pi_1:\ \hyp^{n+2} \rightarrow \mathbb{B}^{n+2},\ (T,Z;X_0,X_1,...,X_n)_{\hyp} \to \frac{1}{1+T}(Z;X_0,X_1,...,X_n)_{\mathbb{B}} = (w;y_0,y_1,...,y_n)_{\mathbb{B}},\]
where we introduce coordinates $w,y_0,...,y_n$ on $\B^{n+2}$, thought of as 
\begin{equation}
\label{eqn:BallModelParavCoords}
(w;y_0,y_1,...,y_n)_{\mathbb{B}} \to y_0 + \sum_{j=1}^n y_ji_j + wi_{n+1} \in \$\R^{n+2}.
\end{equation}
We can extend this map to the boundary; the sphere at infinity is given by projectivising the light-cone, which we take advantage of to define a map $\partial\hyp^{n+2} \to \partial\B^{n+2}$ as a map from the light-cone to the boundary sphere of the ball model.\footnote{This map is constant on rays of $L^+$, and so descends to $\mathscr{S}^+$.}
\[\pi_1^{\partial}:\ L^+ \to \partial \mathbb{B}^{n+2},\ (T,Z;X_0,X_1,...,X_n)_{\partial\hyp} \to \frac{1}{T}(Z;X_0,X_1,...,X_n)_{\partial\mathbb{B}} = (w;y_0,y_1,...,y_n)_{\partial\mathbb{B}}.\]
The image of $\pi_1^{\partial}$ satisfies a relation induced by the defining relation $|p|^2=0$ for the light-cone:
\[w^2+y_0^2 + y_1^2+...+y_n^2=1,\]
which confirms that the boundary really does map to the unit hypersphere in $\$\R^{n+2}$. \label{not:pi2partial}We then map to the upper half-space model via stereographic projection on the boundary:
\[\pi_2^{\partial}:\ \partial \mathbb{B}^{n+2} \rightarrow \partial \U^{n+2},\ (w;y_0,y_1,...,y_n)_{\partial\mathbb{B}} \to \frac{1}{1-w}\left(y_0+\sum_{j=1}^ny_ji_j\right) = x_0+\sum_{j=1}^nx_ji_j,\]
where $x_0+\sum_{j=1}^nx_ji_j = \textbf{x}$ are the paravector coordinates on $\partial\U^{n+2}$ introduced above. \label{not:pipartial}Let 
\[\pi^{\partial}: L^+ \to \partial \U^{n+2},\ \pi^{\partial} = \pi_2^{\partial} \circ \pi_1^{\partial}\]
be the boundary map from the hyperboloid model to upper half-space.

$\pi_2^{\partial}$ extends to an isometry $\pi_2$ on the interior, which we describe explicitly below. However, we can also compute what we need on the boundary by considering the action on the endpoints of intersecting geodesics; simply take two straight lines and see where their endpoints on $\partial \U^{n+2}$ get mapped to under the action, and you can track how their intersection moves. 

Geodesics in hyperbolic space can be specified by the two points where they meet the boundary at right angles. Thus, an oriented geodesic $\gamma \in \U^{n+2}$ can be specified by an ordered pair of paravectors $\pX,\pY$ (with the possibility one is infinite) in $\partial\U^{n+2} \cong \cpara$. 
\begin{definition}
\label{def:GeodesicsInHyp}
 The geodesic $\gamma(\pX \to \pY)$ is defined to be the unique geodesic in $\U^{n+2}$ that intersects the boundary at $\pX$ and $\pY$, oriented from $\pX$ to $\pY$. The geodesic $\gamma(w_1;\textbf{x} \to w_2;\textbf{y})$ is similarly the oriented geodesic from $(w_1;x_0,...,x_n)_{\partial \mathbb{B}}$ to $(w_2;y_0,...,y_n)_{\partial \mathbb{B}}$ in the ball model. \null\hfill\defn
\end{definition}
\noindent On the final page of \cite{Ahlfors1985}, Ahlfors defines a generalised Cayley mapping $\mathfrak{S}$ from the upper half-space model to the ball model, after identifying each with a subset of $\$\R^{n+2}$ appropriately. In the case of the ball model, this identification is given by (\ref{eqn:BallModelParavCoords}); the identification for $\U^{n+2}$ is given in the discussion of Section \ref{sec:UpperHalfspace}.
\begin{definition}
\label{defn:GeneralisedCayleyTransform}
For a model $M$ of hyperbolic space, let $\widehat{M} = M \cup \partial M$ be the closure of the model. The generalised Cayley transform $\mathfrak{S}$ is the isometry defined by:
\[\mathfrak{S}:\ \widehat{\U}^{n+2} \to \widehat{\B}^{n+2}, \ \mathfrak{S}(x) = (x-i_{n+1})(x+i_{n+1})^{-1}.\]
It is invertible, with inverse given by
\[\mathfrak{S}^{-1}:\ \widehat{\B}^{n+2} \to \widehat{\U}^{n+2},\ \mathfrak{S}^{-1}(x) = (1-x)^{-1}(1+x)i_{n+1}.\tag*{\defn}\]
\end{definition}
\noindent As written, this map doesn't agree on the boundary with the map $\pi_2^{\partial}$. But, by adjusting it, we can align the maps.
\begin{definition}
\label{def:pi2}
Our map from the ball model to the upper half-space model is given by
 \[\pi_2:\ \widehat{\B}^{n+2} \to \widehat{\U}^{n+2},\ \pi_2(x) =\frac{|1+i_{n+1}\ol{x}|^2}{|1-i_{n+1}\ol{x}|^2}\mathfrak{S}^{-1}(-i_{n+1}\ol{x}).\tag*{\raisebox{-0.8em}{\defn}}\]
\end{definition}
\begin{proposition}
$\pi_2^{\partial}(W) = \pi_2(W)$, for all $W \in \partial\B^{n+2}$.
\end{proposition}
\noindent\textit{Proof.} We apply both sides to the same generic point $W \in \partial \B^{n+2}$:
\begin{equation}
 \label{eqn:WinPartialB}
W = y_0 + \sum_{j=1}^ny_ji_j + wi_{n+1} \in \partial \B^{n+2}.
\end{equation}
Taking our modified Cayley transform $\pi_2$, we substitute the map $\mathfrak{S}^{-1}$ and expand some brackets.
\begin{align*}
\pi_2(W) %&= \frac{|1+i_{n+1}\ol{W}|^2}{|1-i_{n+1}\ol{W}|^2} (1+i_{n+1}\ol{W})^{-1}(1-i_{n+1}\ol{W})i_{n+1}\\
&= \frac{-1}{|1-i_{n+1}\ol{W}|^2} (Wi_{n+1}+i_{n+1}\ol{W})i_{n+1},
\end{align*}
where we use that $|W|=1$ for $W \in \partial\B^{n+2}$. If we expand $W$ using (\ref{eqn:WinPartialB}), we can ``commute'' the $i_{n+1}$ terms around (flipping the signs of $i_j$ with $1 \leq j \leq n$), and simplify further.
\begin{align*}
\pi_2(W) &= \frac{2}{|1-i_{n+1}\ol{W}|^2} \left(y_0 + \sum_{j=1}^ny_ji_j\right).
\end{align*}
We then expand the scaling factor out front and expand $W$ with (\ref{eqn:WinPartialB}), which gives us
%\begin{align*}
%\pi_2(W) &= \frac{2}{(2-i_{n+1}\ol{W}+Wi_{n+1})} \left(y_0 + %\sum_{j=1}^ny_ji_j\right).
%\end{align*}
%After again expanding $W$ with (\ref{eqn:WinPartialB}), this simplifies to
\begin{align*}
\pi_2(W) &= \frac{1}{1 - w} \left(y_0 + \sum_{j=1}^ny_ji_j\right) = \pi_2^{\partial}(W). \tag*{\raisebox{-1em}{\qed}}
\end{align*}
\noindent Because this map agrees with $\pi_2^{\partial}$ on the boundary, and isometries extend to the interior of hyperbolic space uniquely, this map is equivalent to the map given by looking at intersecting geodesics. However, while this is an explicit construction, in practice it often proves easier to apply the boundary map and look at intersecting geodesics to track points in the interior.
\begin{proposition}
\label{prop:PiTwoComponentForm}
    Given a point $W = (w;y_0,y_1,...,y_n) \in \B^{n+2}$, the map $\pi_2$ can be written explicitly in component form as
    \[\pi_2(W) = \frac{(1-|W|^2;2y_0,2y_1,...,2y_n)_{\U}}{1-2w+|W|^2},\]
    or purely in terms of the coordinates:
    \[\pi_2(W) = \frac{\left(1-\left(\sum_{j=0}^ny_j^2 + w^2\right);2y_0,2y_1,...,2y_n\right)_{\U}}{\sum_{j=0}^ny_j^2 + (w-1)^2}.\]
\end{proposition}
\noindent\textit{Proof.} We expand $W$ as 
\begin{equation}
\label{eqn:WBallGeneric}
W = y_0 + \sum_{j=1}^ny_ji_j + wi_{n+1} \in \B^{n+2}, 
\end{equation}
and compute:
\begin{align*}
\pi_2(W) %&= \frac{|1+i_{n+1}\ol{W}|^2}{|1-i_{n+1}\ol{W}|^2} (1+i_{n+1}\ol{W})^{-1}(1-i_{n+1}\ol{W})i_{n+1}\\
&= \frac{1}{|1-i_{n+1}\ol{W}|^2} (1-|W|^2-(Wi_{n+1}+i_{n+1}\ol{W}))i_{n+1}.
\end{align*}
We substitute (\ref{eqn:WBallGeneric}) to allow some manoeuvring, and we find: 
\begin{align*}
\pi_2(W) %&= \frac{1}{|1-i_{n+1}\ol{W}|^2} \left(1-|W|^2-\left(\left(y_0 + \sum_{j=1}^ny_ji_j + wi_{n+1}\right)i_{n+1}+i_{n+1}\ol{\left(y_0 + \sum_{j=1}^ny_ji_j + wi_{n+1}\right)}\right)\right)i_{n+1}\\
&= \frac{1}{|1-i_{n+1}\ol{W}|^2} \left((1-|W|^2)i_{n+1} + 2\left(y_0 + \sum_{j=1}^ny_ji_j\right)\right).
\end{align*}
We can expand and simplify the denominator:
\begin{align*}
|1-i_{n+1}\ol{W}|^2 %&= (1-i_{n+1}\ol{W})(1+Wi_{n+1}) \\ 
%&= (1-i_{n+1}\ol{W}+Wi_{n+1}-i_{n+1}|W|^2i_{n+1})\\ 
&= (1 + |W|^2 - i_{n+1}\ol{W} + Wi_{n+1}).
\end{align*}
Expand $W$ in terms of components once more to find $|1-i_{n+1}\ol{W}|^2 = 1 - 2w +|W|^2$. 
%\begin{align*}
%|1-i_{n+1}\ol{W}|^2 %&= \left(1+|W|^2-i_{n+1}\left(y_0 - \sum_{j=1}^ny_ji_j - wi_{n+1}\right)+\left(y_0 + \sum_{j=1}^ny_ji_j + wi_{n+1}\right)i_{n+1}\right) \\
%&= 1 - 2w +|W|^2.
%\end{align*}
In total, we have
\begin{align*}
\pi_2(W) &= \frac{2\left(y_0 + \sum_{j=1}^ny_ji_j\right) + (1-|W|^2)i_{n+1}}{1 - 2w +|W|^2},
%&= \frac{(1-|W|^2;2y_0,2y_1,...,2y_n)_{\U}}{1-2w+|W|^2}\\
%&= \frac{\left(1-\left(\sum_{j=0}^ny_j^2 + w^2\right);2y_0,2y_1,...,2y_n\right)_{\U}}{\sum_{j=0}^ny_j^2 + (w-1)^2}.\tag*{\qed}
\end{align*}
which can be rearranged into the forms given in the proposition statement. \hfill\qed
\\\\
\label{not:pi}We combine our maps between models to give a map from the hyperboloid directly to the upper half-space model. 
\[\pi: \hyp^{n+2} \to \U^{n+2},\ \pi = \pi_2 \circ \pi_1.\]

\subsubsection{The Action of $SL(2,\lip)$ on Upper Half-Space}
The action of $SL(2,\lip)$ on $\U^{n+2}$ is, by Definition \ref{def:ActionOnHyperbolic}, inherited from the action on the hyperboloid model in Minkowski space. That is, for $p \in \partial\U^{n+2}$ we define $A.p = \pi^{\partial}\left(A.(\pi^{\partial})^{-1}(p)\right)$, and we then extend the map on the boundary to the interior by considering the action on the endpoints of geodesics. But we have another, rather more natural action on $\U^{n+2}$, defined by letting $SL(2,\lip)$ act on the boundary as M\"{o}bius transformations.
\begin{lemma}
\label{UpperHalfspaceMobiusTransformationsLemma}
 The action of $SL(2,\lip)$ on $\partial\U^{n+2}$ inherited from the action on $\partial\hyp^{n+2}$ is equivalent to the standard action by M\"{o}bius transformations.
\end{lemma}
\noindent\textit{Proof.} 
The map $\pi^{\partial}$ is given in coordinates by
\[\pi^{\partial}: L^+ \to \partial \U^{n+2},\ (T,Z;X_0,X_1,...,X_n)_{\partial\hyp} \to \frac{1}{T-Z}\left(X_0+\sum_{j=1}^nX_ji_j\right) = \frac{1}{T-Z}\pX.\]
Consider acting on a point\footnote{This is really a representative for a ray along $L^+$; we can think of quotienting by $\R^+$.} $p = (T,Z;\pX)_{\partial\hyp}$ in the boundary $\mathscr{S}^+$ of the hyperboloid model. Action by the element
\[A = \begin{pmatrix}
 a & b \\
 c & d
\end{pmatrix} \in\ SL(2,\lip)\]
results in the point $A.p$ with coordinates\footnote{We use $\pY$ since $\pX$' would clash with our involution $\cdot'$.} $(T',Z';\pY)$:
\begin{align*}
2T' &= (T+Z)(|a|^2+|c|^2)+(T-Z)(|b|^2+|d|^2)+2\text{Re}(a\pX\ol{b}+c\pX\ol{d}),\\
2Z' &= (T+Z)(|a|^2-|c|^2)+(T-Z)(|b|^2-|d|^2)+2\text{Re}(a\pX\ol{b}-c\pX\ol{d}),\\
\pY &= 2\left((T+Z)a\ol{c} + (T-Z)b\ol{d} + a\pX\ol{d} + b\ol{\pX}\ol{c}\right).
\end{align*}
Feeding this to $\pi^{\partial}$ gives us the image 
\begin{equation}
\label{eqn:UpperHalfMobiusEqui1}
\pi^{\partial}(A.p) = \frac{(T+Z)a\ol{c} + (T-Z)b\ol{d} + a\pX\ol{d} + b\ol{\pX}\ol{c}}{((T+Z)|c|^2+(T-Z)|d|^2+2\text{Re}(c\pX\ol{d}))} \in \partial\U^{n+2}.
\end{equation}
Now we map to $\U^{n+2}$ first,
and apply the M\"{o}bius transformation.
\begin{align*}
A.\pi^{\partial}(p) &= A.\left(\frac{\pX}{T-Z}\right) = \left(a\frac{\pX}{T-Z}+b\right)\left(c\frac{\pX}{T-Z}+d\right)^{-1} \\
&= \left(a|\pX|^2\ol{c}+(T-Z)(a\pX\ol{d}+b\ol{\pX}\ol{c})+(T-Z)^2b\ol{d}\right)\left|c\pX+d(T-Z)\right|^{-2}.
\end{align*}
As $p$ is represented by a point on the light-cone its coordinates satisfy $T^2-Z^2-|\pX|^2=0$, implying $|\pX|^2=(T-Z)(T+Z)$. Making this substitution, our expression then reduces to (\ref{eqn:UpperHalfMobiusEqui1}):
\begin{align*}
A.\pi^{\partial}(p) %&= (T-Z)\left(a(T+Z)\ol{c} + (T-Z)b\ol{d} + a\pX\ol{d} + b\ol{\pX}\ol{c}\right)\left|c\pX + d(T-Z)\right|^{-2} \\
&= (T-Z)\left((T+Z)a\ol{c} + (T-Z)b\ol{d} + a\pX\ol{d} + b\ol{\pX}\ol{c}\right)\left(\left(c\pX + (T-Z)d\right)\left(\ol{\pX}\ol{c} + (T-Z)\ol{d}\right)\right)^{-1} \\
&= \frac{(T+Z)a\ol{c}+(T-Z)b\ol{d}+a\pX\ol{d}+b\ol{\pX}\ol{c}}{(T+Z)|c|^2+(T-Z)|d|^2+2\text{Re}(c\pX\ol{d})}.
\end{align*}
For the final equality we need $c\pX \ol{d} + \ol{c\pX \ol{d}} \in \R$; if $c$ or $d$ is $0$ this is trivial, so assume $c,d \neq 0$. This equality follows from the fact that $(c,d) \in S\lip$ by Proposition \ref{prop:AlternateSLCharacterisation}; this implies $c = \pV d$ for some $\pV \in \para$, and then $c\pX\ol{d} = |d|^2 \pV Ad_d(\pX)$ is of degree $\leq 2$, implying $c\pX\ol{d} + \ol{c\pX\ol{d}} \in \R$.\qed
\\\\
From Section \ref{sec:EquivarianceOfPhi1}, we know $\Phi_1$ is equivariant. We can now extend this equivariance to the entire map $\Phi$.
\begin{theorem}
\label{thm:PhiEquivariant}
The map $\Phi: S\lip \to L\text{Hor}^{n+2}$ is equivariant with respect to the $SL(2,\lip)$-action.
\end{theorem}
\noindent\textit{Proof.} We've already defined an action on both domain and codomain, so all that's left is to show that $\Phi$ commutes with the group action. We have that $\Phi_1$ is equivariant (Theorem \ref{Phi1Equivariance}), as is $\phi_2$ (Lemma \ref{phi2equivariant}), so we focus on $\Phi_2$. 

But the action of $SL(2,\lip)$ on both $\mathcal{MF}^n$ and $L\text{Hor}^{n+2}$ is inherited from the same linear action on $\R^{1,n+2}$, one by restriction and one by taking a derivative. The derivative of a linear map is that very same map, so the action on the multiflags and on line fields is inherited from the same action on $\R^{1,n+2}$. This implies $\Phi_2$, and therefore $\Phi$, are $SL(2,\lip)$ equivariant. \qed

\subsection{Explicit Computation of $\Phi$}
We can calculate the map $\Phi:\ \lip \to L\text{Hor}^{n+2}$ in a simple case, and describe the general case via $SL(2,\lip)$ equivariance.
\subsubsection{The Base Case}
\begin{lemma}
\label{TransitiveExampleLemma}
 $\Phi(1,0)$ is the horosphere $H_0$ in $\hyp^{n+2}$ with point at infinity along the direction $p_0 = (1,1;0,...,0)_{\partial\hyp}$ on $L^+$ and which passes through $q_0 = (1,0;0,...,0)_{\hyp}$. The $n$ oriented line fields are parallel and mutually orthogonal, and at $q_0$ point in the directions $\partial_j = (0,0;...,1,...,0)_{T\hyp}$ where the only non-zero entry is for $X_j$, $1 \leq j \leq n$.
\end{lemma}
\noindent \textit{Proof.} Considering the example in Section \ref{ExampleComputation}, setting $\theta=0$ we see that $\phi_1(1,0) = p_0$ as expected. The associated horosphere is the manifold $H_0 = \{x \in \hyp^{n+2}\ |\ (x|p_0)_{1,n+2} = 1\}$, and $q_0$ satisfies this relation, so it lies on $H_0$. Since the family of horospheres with a given centre has one parameter for radius, this is enough to specify the horosphere. 

We can also compute the flag vectors from the example in Section \ref{ExampleComputation}:
\[(D_{\kappa}\phi_1)(\check{\kappa}i_j) = (0,0;0,...,1,...,0)_{T\hyp},\]
where the only non-zero entry is the $X_j$ component. Therefore the $i_j$ subflag is given by $[[p_0,\partial_j]]$ with induced orientation.

To find the line fields we intersect this flag with $H_0$. As proven in Proposition \ref{ParabolicLemma} the action of the parabolic subgroup on $H_0$ is simple transitive, so the entire horosphere is parameterised by the image of $q_0$ under the group action. Write $q_{\pV }=P^{(1,0)}_{\pV }.q_0$, and then $H_0 = \{q_{\pV }\ |\ \pV \in \para\}$. The action of $P^{(1,0)}_{\pV}$ on a point $p$ in Minkowski (written as a matrix $S_p \in \mathcal{P}^{n+2}$) is:
\begin{align*}
P^{(1,0)}_{\pV}.S_p %&= \frac{1}{2}\begin{pmatrix}
% 1 & \pV \\
% 0 & 1
%\end{pmatrix}
%\begin{pmatrix}
% T+Z & \pX \\
% \ol{\pX} & T-Z 
%\end{pmatrix}
%\begin{pmatrix}
% 1 & 0 \\
% \ol{\pV } & 1
%\end{pmatrix}\\
&=\frac{1}{2} 
\begin{pmatrix}
T+Z+|\pV |^2(T-Z) +\pX\ol{\pV }+\pV \ol{\pX} & \pX + (T-Z)\pV \\
\ol{\pX} + (T-Z)\ol{\pV } & T-Z
\end{pmatrix}.
\end{align*}
This gives us the following transformed coordinates for $P_{\pV}^{(1,0)}.S_p$:
\begin{align*}
T' &= T + \frac{\pX\ol{\pV }+\pV \ol{\pX}}{2}+\frac{|\pV |^2(T-Z)}{2},\\
Z' &= Z + \frac{\pX\ol{\pV }+\pV \ol{\pX}}{2}+\frac{|\pV |^2(T-Z)}{2},\\
\pY &= \pX+(T-Z)\pV,
\end{align*}
which when applied to $S_p=S_{q_0}$ gives $q_{\pV} = \left(1+\frac{|\pV |^2}{2},\frac{|\pV |^2}{2};\pV \right)_{\hyp}$. These $q_{\pV}$ are the points of the horosphere $\phi(1,0)$ in the hyperboloid model. At the point $q_{\pV}$ the $i_j$ oriented line field is given by the intersection of the tangent space of $H_0$ at $q_{\pV}$ (which is $q_{\pV}^{\perp} \cap p_0^{\perp}$) with the $i_j$ subflag: 
\[q_{\pV}^{\perp}\cap [[p_0,\partial_j]] \cap p_0^{\perp} = q_{\pV}^{\perp}\cap [[p_0,\partial_j]],\] 
as the null flags live in the tangent space to $L^+$ at $p_0$ already. The points $x=(T_x,Z_x;\pX_x)_{T\hyp} = (T_x,Z_x;[X_x]_0,[X_x]_1,...,[X_x]_n)_{T\hyp} \in q_{\pV}^{\perp} \cap [[p_0,\partial_j]]$ lie in the hyperplane $q_{\pV}^{\perp}$ and so satisfy the equation
\begin{equation}
 \label{hyperplaneCondition}
(x | q_{\pV})_{1,n+2} = \left(1+\frac{|\pV |^2}{2}\right)T_x- \frac{|\pV |^2}{2}Z_x - \pV \cdot \pX_x = 0.
\end{equation}
But they also lie in the $i_j$ subflag, and so satisfy the pair of conditions
\begin{equation}
\label{flagFirstcondition}
[X_x]_0,[X_x]_1,...,\widehat{[X_x]}_j,...,[X_x]_n=0,
\end{equation}
\begin{equation}
 \label{flagCondition}
 T_x=Z_x,
\end{equation} 
where the hat denotes exclusion. Simply speaking, this says all the coordinates of the paravector $\pX_x$ are zero except the $i_j$ coordinate, and the $T$ and $Z$ coordinates match. 

Combining (\ref{hyperplaneCondition}) and (\ref{flagCondition}), the points $x=(T_x,Z_x;\pX_x)_{T\hyp} \in q^{\perp}_{\pV}\cap [[p_0,\partial_j]]$ of the intersection satisfy
\begin{align*}
0 &= \left(1+\frac{|\pV |^2}{2}\right)T_x- \frac{|\pV |^2}{2}T_x - \pV \cdot \pX_x
\end{align*}
which implies $T_x = \pV \cdot \pX_x$. By (\ref{flagFirstcondition}), $[X_x]_j$ is the only non-zero term of $\pX_x$ and this reduces to $T_x = V_j[X_x]_j$. This gives an intersection parameterised by
\[(V_j[X_x]_j,V_j[X_x]_j;0,...,[X_x]_j,...,0)_{T\hyp} = [X_x]_jV_jp_0 + [X_x]_j\partial_j,\]
and spanned by
\begin{equation}
\label{eqn:ExampleLineFieldSpan}
(V_j,V_j;0,...,1,...,0)_{T\hyp} = V_jp_0 + \partial_j.
\end{equation}
In particular, as the point $q_0$ has $\pV=0$ the oriented line field is directed there by $\partial_j$.

Finally, we show that the line fields are \textit{parallel}; that is, taking the intersection of the flag at other points of the horosphere leads to the same oriented line as applying the parabolic translation taking $q_0$ to that point. Let $\pW = W_0 + \sum_{j=1}^nW_ji_j \in \para$ and apply the parabolic transformation $P^{(1,0)}_{\pW}$ to the vector $V_jp_0 + \partial_j$ directing the line field:
\begin{align*}
P^{(1,0)}_{\pW}.(V_jp_0 + \partial_j) %&= \frac{1}{2}\begin{pmatrix}
% 1 & \pW \\
% 0 & 1
%\end{pmatrix}
%\begin{pmatrix}
% 2V_j & i_j \\
% \ol{i_j} & 0 
%\end{pmatrix}
%\begin{pmatrix}
% 1 & 0 \\
% \ol{\pW} & 1
%\end{pmatrix} \\
&= \frac{1}{2}\begin{pmatrix}
2V_j-2\text{Re}(i_j\pW) & i_j \\
-i_j & 0
\end{pmatrix},
\end{align*}
where $-$Re$(i_j\pW) = W_j$. This gives the defining vector $((V_j+W_j),(V_j+W_j);0,...,1,...,0)_{T\hyp}$ for our oriented line field at the point $q_{\pV + \pW}$. This is precisely the defining vector (\ref{eqn:ExampleLineFieldSpan}) we get from intersecting directly at $q_{\pV+\pW}$, so the oriented line field is parallel as claimed. \qed
\\\\
Here we see the first demonstration of these line fields being parallel! Though this is only an example, it will be proven be true in general in Corollary \ref{cor:LineFieldsParallel}. The line fields of this example are also orthogonal, which will be proven in the general case in Corollary \ref{OrthogonalLineFields}.
\begin{proposition}
\label{BasicHorosphereLemma}
 $\pi \circ \Phi (1,0)$ is the horosphere in $\U^{n+2}$ centred at infinity, and of Euclidean height 1. It bears $n$ mutually orthogonal parallel oriented line fields, with the $i_j$ field at $q_0 = (1;0,...,0)_{\U}$ pointing along the $x_j$ coordinate.
\end{proposition}
\noindent \textit{Proof.} Recall that $\pi_1$, $\pi_2$ are our maps between models of hyperbolic space from Section \ref{sec:MapsBetweenModels}. From Lemma \ref{TransitiveExampleLemma} we have a description of $\Phi(1,0)$ in $\hyp^{n+2}$. Under the map $\pi^{\partial}$ the centre $p_0$ goes to $\infty \in \partial\U^{n+2}$; the point $q_0$ maps to $(0;0,...,0)_{\mathbb{B}}$ in the ball model under $\pi_1$, but as this isn't on the boundary we cannot simply feed it into $\pi_2^{\partial}$. Bypassing $\pi_2$, we instead describe this point as the intersection of the pair of geodesics $A=A(0;-1 \to 0;1)$ and $B=B(-1;0 \to 1;0)$ in $\mathbb{B}^{n+2}$.

Following the endpoints under $\pi_2^{\partial}$, the geodesic $A$ maps to the geodesic $\pi_2^{\partial}A(-1\to 1)$, while $B$ maps to the geodesic $\pi_2^{\partial}B(0 \to \infty)$. Both of these geodesics are contained in the ($x_0$, $z$)-plane, and the intersection (which is on the horosphere) is at the point $(1;0)_{\U}$; as the horosphere is a plane parallel to the boundary $\para$, the Euclidean height is 1.

The line fields at $q_0$ are given in Lemma \ref{TransitiveExampleLemma}, directed by the vectors $\partial_j = (0,0;0,...,1,...,0)_{T\hyp}$, where the only non-zero entry is the $X_j$ component for $1 \leq j \leq n$. The maps $\pi_1$, $\pi_2$ are isometries, which implies that the parallel property of the line fields passes from $\hyp^{n+2}$ to the other models. To find the direction, then, we only need to check the direction of the $i_j$ fields at a single point.

We use the Jacobian matrix to map the tangent vector $\partial_j$ to the ball model.
\[D\pi_1: T\hyp^{n+2} \to T\B^{n+2},\ D\pi_1 = \begin{pmatrix}
 -\frac{Z}{(1+T)^2} & \frac{1}{1+T} & 0 & ... \\
 -\frac{X_0}{(1+T)^2} & 0 & \frac{1}{1+T} & ... \\
 ... & ... & ... & ...\\
 -\frac{X_n}{(1+T)^2} & 0 & ... & \frac{1}{1+T} \\
\end{pmatrix}.\]
The only non-zero entries in row $i$ are the first column and the $(i+1)$th column. Applying this to our $\partial_j$ vectors we find
\[D\pi_1(\partial_j) = \left(0;0,...,\frac{1}{1+T},...,0\right)_{T\mathbb{B}},\]
where the only non-zero entry is in the $y_j$ column. As $T>0$ this defines the same oriented line field as $\partial^{\B}_j = (0;0,...,1,...,0)_{T\mathbb{B}}$, so we rescale and take that to be our representative instead. 

Next, we compute the Jacobian of $\pi_2$ using the component form of Proposition \ref{prop:PiTwoComponentForm}. Not all of the entries of $D\pi_2$ are needed to find the action on the $\partial_j^{\B}$; we can ignore derivatives with respect to $w$ and $y_0$. We compute the other partial derivatives:
\[\frac{\partial(\pi_2(W))_z}{\partial y_j} = \frac{\partial}{\partial y_j} \left(\frac{1-\left(\sum_{\ell=0}^ny_\ell^2 + w^2\right)}{\sum_{\ell=0}^ny_\ell^2 + (w-1)^2}\right) = \frac{4y_j\left(w-1\right)}{\left(\sum_{\ell=0}^n y_\ell^2 + (w-1)^2\right)^2}.\]
For $j \neq k$,
\[\frac{\partial(\pi_2(W))_{x_k}}{\partial y_j} = \frac{\partial}{\partial y_j}\left(\frac{2y_k}{\sum_{\ell=0}^n y_{\ell}^2 +(w-1)^2}\right) = -\frac{4y_ky_j}{\left(\sum_{\ell=0}^n y_{\ell}^2 + (w-1)^2\right)^2}.\]
If instead we have $k=j$, we find:
\[\frac{\partial(\pi_2(W))_{x_j}}{\partial y_j} = \frac{\partial}{\partial y_j}\left(\frac{2y_j}{\sum_{\ell=0}^n y_{\ell}^2 + (w-1)^2}\right) = \frac{2(\sum_{\ell=0}^ny_{\ell}^2 -2y_j^2 +(w-1)^2)}{\left(\sum_{\ell=0}^n y_{\ell}^2 + (w-1)^2\right)^2}.\]
As mentioned above we only need to evaluate these at one point, and the parallel property will define their direction elsewhere. We evaluate the vector $D\pi_2(\partial_j)$ at the point $(0;0,0,...,0)_{\B}$.
\[D\pi_2(\partial_j)|_{(0;0)_{\B}} = (0;0,...,0,2,0,...,0)_{T\U},\]
where the only non-zero entry is the $x_j$ component. The $i_j$ oriented line field thus points in the $x_j$ direction as claimed. 
\qed

\subsubsection{The General Case}
The horosphere of Proposition \ref{BasicHorosphereLemma} is an $(n+1)$-plane; more broadly, the horospheres centred at infinity are simply $(n+1)$-planes parallel to the boundary copy of $\para$ in the upper half-space. But horospheres centred elsewhere appear as spheres tangent to their centres on $\para$, which makes it a little more difficult to describe the line fields of the correspondence. We make a standard choice, the \textit{North Pole specification}, which considers the line fields only in the tangent space to the horosphere at the point of maximal $z$ value in $\U^{n+2}$ (this point is the \textit{north pole}, the point furthest from the boundary $\para$ in the Euclidean metric on $\$\R^{n+2}$).

So we can describe isometries on line fields by M\"{o}bius transformations on the boundary at infinity, we associate to a given line field a tangent geodesic which is similarly oriented. Then, as the endpoints of the geodesic change, we can reconstruct how the geodesic (and thus the tangent oriented line) changes as well, as conformality ensures the tangencies are preserved.
\begin{lemma}
\label{DecorationsTransformLemma}
Consider a horosphere $H$ in $\U^{n+2}$ centred at $\infty$ or $0$, with a parallel oriented line field specified (respectively, North Pole specified) by the unit paravector $\textbf{D}$. The line field transforms as follows under the generators of $SL(2,\lip)$:
 \begin{itemize}
 \item $A_1 = \begin{pmatrix}
 1 & \pV \\
 0 & 1
 \end{pmatrix},\ \pV \in \para$ leaves $\textbf{D}$ unaffected.
 \item $A_2 = \begin{pmatrix}
 0 & -1 \\
 1 & 0 
 \end{pmatrix}$ sends $\textbf{D}$ to $-\ol{\textbf{D}}$, up to scaling.\footnote{Recall line fields are invariant under rescaling the specification vector.}
 \item $A_3 = \begin{pmatrix}
 a & 0 \\
 0 & a^{*-1} 
 \end{pmatrix},\ a \in \lip$ sends $\textbf{D}$ to $a\textbf{D}a^*$.
 \end{itemize}
\end{lemma}
\noindent\textit{Proof.} Let the Euclidean diameter (or equivalently the height) of $H$ be $h$. If $H$ is centred at $0$, the circular geodesic segment $\gamma(-h\textbf{D} \to h\textbf{D})$ is tangent to the north pole specification vector; if centred at $\infty$, this geodesic will also be tangent to the oriented line field of $H$ at $(1;0)_{\U}$. The orientation of the line field also agrees with the orientation of the geodesic. 

The direction of this geodesic can be specified by the `vector' (really just an ordered pair of points on the boundary) $\pX$ pointing from one endpoint of $\gamma$ to another. Alternately, it can be thought of as the projection onto the boundary $\para$ of the geodesic. We write it as a vector in $\para$:
\[\pX = h\textbf{D} - (-h\textbf{D}) = 2h\textbf{D}.\]
Note that this is a vector in the copy of $\para \subset \partial\U^{n+2}$, and not a tangent vector! 

The parabolic transformation $A_1$ sends the endpoints $\pm h\textbf{D}$ to $\pm h\textbf{D} + \pV$. The action on $\pX$ is then trivial, as it affects the endpoints equally:
\[A_1.\pX = h\textbf{D} + \pV - (- h\textbf{D} + \pV)=2h\textbf{D} = \pX.\]
Inversion ($A_2$) sends the endpoints $\pm h\textbf{D}$ to $-\left(\pm h\textbf{D}\right)^{-1} = -(\pm h\ol{\textbf{D}})\left| \pm h\right|^{-2}$. Taking the difference and rescaling, the line segment points in the direction $-\ol{\textbf{D}}$. Finally, $A_3$ sends the endpoints $\pm h\textbf{D}$ to $\pm a(h\textbf{D})a^*$, and the directed line segment between them points along the paravector $a\textbf{D}a^*$, as required. \qed

\begin{corollary}
\label{cor:LineFieldsParallel}
    The $n$ oriented line fields of $\Phi (\kappa)$ are parallel.
\end{corollary}
\noindent\textit{Proof.} From Lemma \ref{TransitiveExampleLemma}, the line fields are parallel in the case of $\kappa = (1,0)$. The action of $SL(2,\lip)$ is equivariant through $\Phi_2$ by Theorem \ref{thm:PhiEquivariant}, and Lemma \ref{DecorationsTransformLemma} tells us the action of $SL(2,\lip)$ on the line fields of $L$Hor$^{n+2}$ is via isometries, which preserve parallel lines.\qed
\\\\
With these lemmas in place, we can use the action of $SL(2,\lip)$ to find the relationship between arbitrary Lipschitz spinors and horospheres bearing line fields. 
\begin{theorem}
\label{HorosphereCoordinateTheorem}
The spinor $\kappa=(\xi,\eta) \in S\lip$ maps under $\Phi$ to a horosphere decorated with $n$ parallel oriented line fields, centred at $\xi\eta^{-1}$ in the upper half-space model (taken to be $\infty$ if $\eta=0$).
\begin{itemize}
 \item If $\eta \neq 0$ the horosphere has Euclidean diameter $|\eta|^{-2}$, and the north pole specification for the $i_j$ direction field is $\eta^{'}i_j\ol{\eta}$.
 \item If $\eta = 0$ then the horosphere has Euclidean height $|\xi|^2$ and the $i_j$-direction field is specified by $\xi i_j\xi^*$.
\end{itemize}
\end{theorem}
\noindent \textit{Proof.} The map $\phi_1$ sends $\kappa$ to a point $(|\xi|^2+|\eta|^2,|\xi|^2-|\eta|^2;2\xi\ol{\eta})_{\R} = (T,Z;\pX)_{\R}$ on the light-cone.
Expand the paravector component as $\xi\ol{\eta} =[\xi\ol{\eta}]_0+\sum_{j=1}^n[\xi\ol{\eta}]_ji_j$. We map the ideal boundary of the hyperboloid first into the ball model:
\[\pi_1^{\partial}(|\xi|^2+|\eta|^2,|\xi|^2-|\eta|^2;2[\xi\ol{\eta}]_0,2[\xi\ol{\eta}]_1,...,2[\xi\ol{\eta}]_n)_{\partial \hyp} = \frac{1}{|\xi|^2+|\eta|^2}(|\xi|^2-|\eta|^2;2[\xi\ol{\eta}]_0,2[\xi\ol{\eta}]_1,...,2[\xi\ol{\eta}]_n)_{\partial\mathbb{B}},\]
and then to the upper half-space model:
\begin{align*}
\pi_2^{\partial}\left[\frac{1}{|\xi|^2+|\eta|^2}(|\xi|^2-|\eta|^2;2[\xi\ol{\eta}]_0,2[\xi\ol{\eta}]_1,...,2[\xi\ol{\eta}]_n)_{\partial\mathbb{B}} \right] &= \frac{1}{1-\frac{|\xi|^2-|\eta|^2}{|\xi|^2+|\eta|^2}}\frac{2\xi\ol{\eta}}{\left(|\xi|^2+|\eta|^2\right)}%\\
%=\frac{|\xi|^2+|\eta|^2}{2|\eta|^2}\frac{2\xi\ol{\eta}}{|\xi|^2+|\eta|^2} = \xi\ol{\eta}|\eta|^{-2} 
= \xi\eta^{-1}.
\end{align*}
To compute the Euclidean diameter (or height), first consider the case where $\eta$=0. Starting at $(1,0)$ we apply the transformation
\[G_1 = \begin{pmatrix}
\xi & 0 \\
0 & \xi^{*-1}
\end{pmatrix},\]
to get all Lipschitz spinors of the form $G_1\begin{pmatrix}
 1 & 0
\end{pmatrix}^T = 
\begin{pmatrix}
 \xi & 0
\end{pmatrix}^T$. This transforms the geodesic $\pi_2^{\partial}A(-1 \to 1)$ of Proposition \ref{BasicHorosphereLemma} into $G_1\pi_2^{\partial}A(-\xi\xi^* \to \xi\xi^*)$, and the geodesic $\pi_2^{\partial}B(\infty \to 0)$ goes to $G_1\pi_2^{\partial}B(\infty \to 0)$, so the intersection of these geodesics will now be at the midpoint of the Euclidean semicircle $G_1\pi_2^{\partial}A$, and thus a full Euclidean radius away from the boundary. The radius is $|\xi\xi^*|=|\xi||\xi^*|=|\xi|^2$, which is the Euclidean height of the horosphere as all points of the plane are at the same height. Lemma \ref{DecorationsTransformLemma} then implies the action of $G_1$ sends the specifying vectors $i_j$ to $\xi i_j \xi^*$.

If $\eta \neq 0$, start by applying the inversion
\[G_2 = \begin{pmatrix}
 0 & -1 \\
 1 & 0 
\end{pmatrix},\]
which sends $(1,0) \to (0,1)$ and sends our geodesics to $G_2\pi_2^{\partial}A(1 \to -1),\ G_2\pi_2^{\partial}B(0 \to \infty)$. That is, the orientations are reversed but the geodesics are otherwise invariant. We then apply the transformation
\[G_3 = \begin{pmatrix}
 \eta^{*-1} & 0 \\
 0 & \eta 
\end{pmatrix},\]
which sends $(0,1) \to (0,\eta)$. This transforms the geodesics as follows:
\begin{align*}
G_2\pi_2^{\partial}A &\to G_3G_2\pi_2^{\partial}A(\eta^{*-1}\eta^{-1} \to -\eta^{*-1}\eta^{-1}),\\
G_2\pi_2^{\partial}B &\to G_3G_2\pi_2^{\partial}B(0 \to \infty),
\end{align*}
and by the same logic as above, the intersection is at Euclidean height $|\eta^{*-1}\eta^{-1}|=|\eta|^{-2}$. As this geodesic goes from the centre of the horosphere to infinity it passes through the north pole, and so this gives the Euclidean diameter. 

Finally, we apply the transformation
\[G_4 = \begin{pmatrix}
 1 & \pV \\
 0 & 1
\end{pmatrix},\]
where $\xi \eta^{-1}=\pV$. This sends $(0,\eta) \to (\xi,\eta)$ and just translates the geodesics equally, which (by Lemma \ref{DecorationsTransformLemma}) doesn't change the point of intersection:
\begin{align*}
G_3G_2\pi_2^{\partial}A &\to G_4G_3G_2\pi_2^{\partial}A(\eta^{*-1}\eta^{-1}+\xi \to -\eta^{*-1}\eta^{-1}+\xi),\\
G_3G_2\pi_2^{\partial}B &\to G_4G_3G_2\pi_2^{\partial}B(\xi \to \infty).
\end{align*}
The Euclidean diameter of $\phi(\xi,\eta)$ is thus $|\eta|^{-2}$. 

Consider acting with these transformations on the oriented line fields directed by the $i_j$; the inversion $G_2$ sends $i_j \to -\ol{i_j}=i_j$, the homothety and rotation $G_3$ sends $i_j \to \eta^{*-1}i_j\eta^{-1}$, and $G_4$ leaves the directing vectors invariant. The final line field is north pole specified by $\eta^{*-1}i_j\eta^{-1}$, which we can rescale to the slightly simpler form $\eta'i_j\ol{\eta}$.\qed
\\\\
This theorem implies the line fields come from taking the orthonormal frame for $\R^{0,n}$ given by the $i_j$, and applying an isometry (plus rescaling) $\sigma(\xi)$ or $\sigma(\eta')$ to the whole frame.
\begin{corollary}
\label{OrthogonalLineFields}
The $n$ line fields of $\Phi(\xi,\eta)$ are mutually orthogonal.
\end{corollary}
\noindent\textit{Proof.} From Theorem \ref{HorosphereCoordinateTheorem} we know the line fields are specified either by $\sigma(\xi)(i_j)$ for every $i_j$, or by $\sigma(\eta')(i_j)$ for every $i_j$. But $\sigma(x)$ is a conformal transformation of $\para$ for every $x \in \lip$ (namely, an isometry plus a homothety). Since the $i_j$ are mutually orthogonal in $\para$, they remain orthogonal after applying an isometry and a homothety.\qed
\\\\
We can now properly restrict the map $\Phi_2$ to the appropriate codomain.
\begin{definition}
 An ordered set of $n$ oriented mutually orthogonal parallel line fields $(L_1,...,L_n)$ on a horosphere is termed a {\normalfont decoration}; the line field $L_j$ may also be referred to as the $i_j$ {\normalfont decoration}. A horosphere in $\hyp^{n+2}$ bearing a decoration is termed a {\normalfont decorated horosphere}. \label{not:DHor}Denote the set of decorated horospheres by $D$Hor$^{n+2}$.\hfill\defn
\end{definition}
\noindent The set of decorated horospheres can be given a manifold structure; in fact, a fibre bundle structure. The space of horospheres Hor$^{n+2}$ can be identified with $\R^+ \times S^{n+1}$, as we can smoothly vary the centre across the boundary of $\hyp^{n+2}$ and then pick a `radius' to capture a unique horosphere. A decoration can be equated with the orthonormal frame in the tangent space to the horosphere whose basis vectors point along the parallel line fields (and from any orthonormal frame we can recover a unique decoration), and so we can identify $D$Hor$^{n+2}$ with the orthonormal frame bundle over Hor$^{n+2}$.

Because the case of $n=0$ is somewhat unique, we define a special decoration to allow a continuous rotation between real spin decorations (defined in the next section).\footnote{See \cite{DescartesTheorem} Section 2.1 for more details.}
\begin{definition}
\label{def:PlanarDecoration}
    The points $(z;x_0,0)_{\U}$ of $\U^3$ can be identified with $\U^2$; the boundary at infinity of $\U^2$ is then given by $\partial \U^2 = \R \cup \{\infty\} \subset \C \cup \{\infty\} = \partial \U^3$. By intersecting $\U^3$ with $\U^2$, we may regard horocycles $H$ of $\U^2$ as arising from horospheres $\widetilde{H}$ of $\U^3$ centred at $\R \cup \{\infty\}$. 
    
    Given a horocycle $H$ in $\U^2 \subset \U^3$, a \emph{planar decoration} on $H$ is a decoration on $\widetilde{H}$ (north pole) specified by the vector $i_1=i$. \null\hfill\defn
\end{definition}
\noindent A spinor $(\xi, \eta)$ describes a planar decoration on a horocycle if and only if $\xi$ and $\eta$ are both real; see \cite{DescartesTheorem}, towards the end of Section 2.1.

\begin{theorem}
\label{thm:Phi2Diffeomorph}
    $\Phi_2$ is a diffeomorphism from $\mathcal{MF}^n$ to $D$Hor$^{n+2}$.
\end{theorem}
\noindent\textit{Proof.} First, we note that $\phi_2$ is a bijection (in fact, a diffeomorphism). Every point $p \in L^+$ will generate a different hyperplane (implying the map is injective), and both $L^+$ and Hor$^{n+2}$ are topologically $S^{n+1} \times \R$. Smoothly varying $p$ leads to smoothly varying hyperplanes, and thus smoothly varying horospheres.

From Corollary \ref{OrthogonalLineFields}, the image $\Phi_2(\kappa)$ for a Lipschitz spinor $\kappa$ with $\phi_1(\kappa) = p$ has orthogonal line fields, and by Corollary \ref{cor:LineFieldsParallel} they are also parallel. So the image of $\Phi_2$ does lie within $D$Hor$^{n+2}$. At a point $p \in L^+$, the set of multiflags is diffeomorphic to $SO(n+1)$. The set of decorations on the horosphere $\phi_2(p)$ is also diffeomorphic to $SO(n+1)$, and the set of multiflags based at $p$ corresponds bijectively to the set of decorations on $\phi_2(p)$ (as both are expressions of an orthonormal basis in terms of the same family of vectors), so $\Phi_2$ is a bijection. As the flags vary smoothly, their intersections with the horosphere will also vary smoothly, so the map is a diffeomorphism.
\qed
\section{Spin decorations, Lambda Lengths, and Ptolemy Relations}
\label{sec:Fourth}
We turn now to the issue of spin, and lift our various $SO(n)$ structures to $Spin(n)$ structures. This will turn our various maps from double covers into genuine homeomorphisms and equivalences, as well as setting us up nicely to apply spinor algebra to our hyperbolic objects. The fruit of this will be a generalisation of Penner's lambda lengths to arbitrary dimensions, and a non-commutative generalisation of the hyperbolic Ptolemy relation in arbitrary dimension.

\label{not:Hypn}In this section, $\hyp^n$ refers to $n$-dimensional hyperbolic space regardless of model. Assume $\hyp^n$ is oriented.
\subsection{Frame Fields and Double Covers}
A horosphere $H \in \hyp^n$ has two choices of normal direction; following \cite{Mathews_Spinors_horospheres}, call the normal pointing towards the centre of the horosphere the \textit{outward} normal, and the one pointing away the \textit{inward} normal. \label{Not:NinNout}These form unit normal vector fields $N^{out}$ and $N^{in}$ on the horosphere, with vectors based at different points related by parallel transport. 

A \textit{frame} is shorthand for a positively oriented orthonormal frame at a point $p$ in $\hyp^n$; that is, an ordered tuple of $n$ orthonormal vectors in $T_p\hyp^n$ which is positively oriented. In \cite{Mathews_Spinors_horospheres} this orientation was defined in terms of the standard cross product on $\R^3$, but in arbitrary dimension defining an orientation takes more work. \label{not:HodgeStar}A common way to define orientations is via the Hodge star operator.
\begin{definition}[Hodge Star for $k$-Vectors]
    Let $(X,g)$ be an oriented, $m$-dimensional smooth manifold $X$ with a (pseudo)-Riemannian metric $g$, and let $p \in X$. For $0 \leq k \leq m$ write $\bigwedge^kT_pX$ for the space of $k$-vectors in the tangent space $T_pX$. Let $(\cdot,\cdot)_g$ be the bilinear form on $\bigwedge^kT_pX$ induced by $g$, defined on simple $k$-vectors $\alpha = \alpha_1 \wedge ... \wedge \alpha_k$, $\beta = \beta_1 \wedge ... \wedge \beta_k$ by the {\normalfont Gram determinant} $(\alpha,\beta)_g = \text{det} \left( g(\alpha_{\ell},\beta_j)\right)$, and extended to the rest of $\bigwedge^kT_pX$ by linearity. The {\normalfont Hodge star operator} is the unique linear function 
    \[\star: \bigwedge^kT_pX \to \bigwedge^{m-k}T_pX\]
    defined by the identity $\alpha \wedge \star \beta = (\alpha, \beta)_g\omega$ for all $\alpha,\beta \in \bigwedge^kT_pX$, where $\omega$ is the unit $m$-vector defined locally by $\omega = e_1 \wedge e_2 \wedge ... \wedge e_m$, for an oriented orthonormal basis $\{e_1,e_2,...,e_m\}$ of $T_pX$. \null\hfill\defn
\end{definition}
\noindent A \textit{positively oriented orthonormal frame} (or just \textit{frame}) at a point $p \in \hyp^n$ then consists of $n$ orthonormal (under the metric induced by $(\cdot|\cdot)_{1,n}$) vectors $(v_1,...,v_n) \in T_p\hyp^n$ such that $\star (v_1 \wedge v_2...\wedge v_{n-1})=v_n$. Because the vectors are orthogonal, the wedge and Clifford products are equivalent and this condition can equivalently be written as $\star (v_1v_2...v_{n-1})=v_n$.

We've established that the decorations of the spinor-horosphere correspondence are orthogonal, so they can be equated with a frame. In the Euclidean geometry of the horosphere, vector fields define oriented line fields and oriented line fields define unit vector fields, so we will use the two somewhat interchangeably. 

\label{not:Fr}The collection of frames forms a principal $SO(n+2)$ bundle over $\hyp^{n+2}$, denoted $\text{Fr}^{n+2} \longrightarrow \hyp^{n+2}$. However, from Proposition \ref{phi1FibreLemma} we know we should really be looking at spin groups, rather than special orthogonal groups. \label{not:SFr}We denote the associated double cover of Fr$^{n+2}$ (which will be a principal $Spin(n+2)$ bundle) by
$S\text{Fr}^{n+2} \longrightarrow \hyp^{n+2}$. Points of this double cover are called \textit{spin frames}, and each point in Fr$^{n+2}$ has two lifts in $S\text{Fr}^{n+2}$. 

How do these frames relate to decorated horospheres?
\begin{definition}[Horosphere Frame Fields]
\label{defn:FrameFields}
 The outwards frame field of a decorated horosphere $H \subset D\text{Hor}^{n+2}$ with unit normals $N^{out}$ and $N^{in}$ and decorations specified by unit vector fields $\widehat{D}_1,\widehat{D}_2,...,\widehat{D}_n$ is\footnote{This is chosen to agree with the convention in \cite{Mathews_Spinors_horospheres}, Definition 4.1. It is slightly different to the convention in \cite{mathews2024quaternions} Definition 6.1.1, as there the normal and the `real paravector 1' generators are swapped in the tuple.}
 \[F^{out} = \left(N^{out},\widehat{D}_1,\widehat{D}_2,...,\widehat{D}_n,\star (N^{out} \wedge \widehat{D}_1 \wedge \widehat{D}_2 \wedge ...\wedge \widehat{D}_n)\right).\]
 Similarly, the inward frame field is given by
 \[F^{in} = \left(N^{in},\widehat{D}_1,\widehat{D}_2,...,\widehat{D}_n,\star (N^{in} \wedge \widehat{D}_1 \wedge \widehat{D}_2 \wedge ...\wedge \widehat{D}_n)\right).\]
\label{not:HodgeOne}For expediency's sake, we label this final Hodge star component by $1^{out}$ and $1^{in}$, respectively. The {\normalfont frame field} of the decorated horosphere is the pair $(F^{out},F^{in})$.\null\hfill\defn
\end{definition}
\noindent Any frame $F$ based at a point $p_F$ in $\hyp^{n+2}$ (that is, any point of Fr$^{n+2}$) will uniquely generate a frame field in a horosphere as follows: call the first component of $F$, $N^{out}$. Choose the horosphere $H_F$ which passes through $p_F$ and is centred at the point of $\partial\hyp^{n+2}$ where the oriented line determined by $N^{out}$ intersects the boundary, in the direction $N^{out}$ points along. Then parallel transport of $F$ across the horosphere will generate a frame field, which will be precisely the $F^{out}$ of the horosphere $H_F$. This gives a surjective map Fr$^{n+2} \to D\text{Hor}^{n+2}$, though the map is not injective as all parallel transports across the horosphere $H_F$ are different points of Fr$^{n+2}$ that generate the same decorated horosphere.  
\subsubsection{Bijections and Covering Spaces}
We know (see Ahlfors \cite{Ahlfors1985}) that the group of orientation-preserving isometries of $\hyp^{n+2}$ is isomorphic to $PSL(2,\lip)$, the quotient $SL(2,\lip)/\{I,-I\}$ identifying elements of $SL(2,\lip)$ with their negative. The action of $PSL(2,\lip)$ can be extended to frames through derivatives, and because hyperbolic space is homogeneous and isotropic this action is simply transitive on the frame bundle Fr$^{n+2}$. This implies Fr$^{n+2}$ is a $PSL(2,\lip)$-torsor.\footnote{Roughly, a $G$-torsor is a set acted on by $G$ that can be identified with $G$ as a group after making an arbitrary choice of identity.}
\begin{proposition}
     $\text{Fr}^{n+2}$ is a $PSL(2,\lip)$-torsor.
\end{proposition}
\noindent\textit{Proof.} We can obtain an explicit identification by choosing a base-point $F_0 \in \text{Fr}^{n+2}$ and making the association $M \to M.F_0$ for $M \in PSL(2,\lip)$. This turns Fr$^{n+2}$ into a group; because the action is simply transitive, the identification is bijective, and because it is obviously a homomorphism we have an isomorphism of groups. \qed
\\\\
We can lift this map to an isomorphism $SL(2,\lip) \cong S\text{Fr}^{n+2}$, choosing a lifted base-point $\Tilde{F}_0 \in S\text{Fr}^{n+2}$ and assigning to $M \in SL(2,\lip)$ the image $M.\Tilde{F}_0$ of the basepoint $\Tilde{F}_0$ under the group action.
\begin{theorem}
\label{DoubleCoverTheorem}
The map $\Phi: S\lip \to D\text{Hor}^{n+2}$ is a smooth double cover.
\end{theorem}
\noindent\textit{Proof.} In Theorem \ref{thm:Phi1SmoothDoubleCover} it was shown that $\Phi_1$ is a smooth double cover, and in Theorem \ref{thm:Phi2Diffeomorph} it was shown $\Phi_2$ is a diffeomorphism. Taken together, we see $\Phi_2 \circ \Phi_1 = \Phi$ is also a smooth double cover. \qed
\begin{remark} We have used the equivariance of $\Phi$ to extend calculations from $\Phi(1,0)$ to other decorated horospheres; but does this allow us to describe all of $D\text{Hor}^{n+2}$? In fact, it does; we can extend the transitive action on $S\lip$ (see Lemma \ref{TransitiveLemma}) through to $D\text{Hor}^{n+2}$, because transitivity of a group action $G$ passes through $G$ equivariant surjective maps; and, if the action is transitive on $D$Hor$^{n+2}$, then we can reach every decorated horosphere from $\Phi(1,0)$. This also lifts to transitivity on $S$Hor$^{n+2}$ under the equivariance of $\Tilde{\Phi}$, terms which which will be defined shortly.
\end{remark}

\subsection{Lifting to Double Covers}
Now we lift everything to the universal cover. Instead of discussing frames and decorations, we now consider spin frames and spin decorations.
\begin{definition}
\label{SpinFrameDefinition}
Consider a horosphere $H \subset \hyp^{n+2}$. An outward (inward) spin frame on $H$ is a continuous lift of an outward (inward) frame on $H$ from Fr$^{n+2}$ to $S\text{Fr}^{n+2}$.
 \begin{itemize}
 \item If $W^{out}$ is an outward spin frame on $H$ lifting an outward frame field, the associated inward spin frame is obtained by rotating $W^{out}$ by an angle of $\pi$ about the oriented geodesic $n$-plane (with orientation chosen by $N^{out}$) spanned by the decorations. More concretely, it is obtained by rotating in the plane spanned by $N^{out}$ and $1^{out}$ by $\pi$, from $N^{out}$ towards $1^{out}$.
 \item If $W^{in}$ is an inward spin frame on $H$ lifting an inward frame field, the associated outward spin frame is obtained by rotating $W^{in}$ by an angle $-\pi$ about the oriented geodesic $n$-plane (with orientation chosen by $N^{out}$) spanned by the decorations. Precisely, it is obtained by rotating in the plane spanned by $N^{in}$ and $1^{in}$ by $\pi$, from $1^{in}$ towards $N^{in}$.\hfill\defn
 \end{itemize}
\end{definition}
\noindent The choice of `associated' spin frame is somewhat arbitrary, but some consistent choice needs to be made. An outward spin frame $W^{out}$ now has an associated inward spin frame $W^{in}$, and the inward frame is associated to $W^{out}$ as well; they form a pair. 
\begin{definition}[Spin Decorations]
 \label{not:SpinDecHoro}A {\normalfont spin decorated horosphere} is a decorated horosphere $H$ equipped with a pair $W=(W^{in},W^{out})$ of an inward and outward spin frame lifting decorations on $H$, each associated to the other. The $i_j$ {\normalfont spin decoration} is then the spin lift of the $i_j$ decoration. \label{not:SHor}The space of spin decorated horospheres $(H,W)$ in $\hyp^{n+2}$ is denoted $S$Hor$^{n+2}$.\hfill\defn
\end{definition}
\subsubsection{Null Flag Bundles}
The flag bundle $\mathcal{MF}^n$ is also insensitive to the spinorial nature of $S\lip$. We can resolve this by considering a map into the double cover $\widetilde{\mathcal{MF}}^n$; choosing an arbitrary lift of $\Phi_1(1,0)$ as a base-point in $\widetilde{\mathcal{MF}}^n$, we obtain a lifted map
\[\Tilde{\Phi}_1: S\lip \to \widetilde{\mathcal{MF}}^n.\]
\noindent We denote points of this new space \textit{spin multiflags}. To differentiate our flags $[[p,v_j]]$ from their lifted counterparts we write $\widetilde{[[p,v_j]]}$ for the lifted flags. We also define a lifted map $\Tilde{\Phi}_2: \widetilde{\mathcal{MF}}^n \to S\text{Hor}^{n+2}$, again choosing an arbitrary lift of $\Phi (1,0)$ as a base-point such that Figure \ref{fig:LiftingComDiagram} commutes.
\begin{figure}[H]
    \centering
\begin{tikzcd}
S\lip \arrow[rr, "\tilde{\Phi}_1"] \arrow[rrdd, "\Phi_1"] &  & \widetilde{\mathcal{MF}} \arrow[rr, "\tilde{\Phi}_2"] \arrow[dd] &  & S\text{Hor}^{n+2} \arrow[dd] \\
                                                          &  &                                                                  &  &                              \\
                                                          &  & \mathcal{MF} \arrow[rr, "\Phi_2"]                                &  & D\text{Hor}^{n+2}           
\end{tikzcd}
    \caption{The various lifts of spaces and maps.}
    \label{fig:LiftingComDiagram}
\end{figure}
\noindent With these lifted functions, we properly have a diffeomorphism $\Tilde{\Phi} = \Tilde{\Phi}_2 \circ \Tilde{\Phi}_1: S\lip \to S\text{Hor}^{n+2}$. Putting together these various correspondences gives us the full Lipschitz spinor-horosphere correspondence.
\begin{theorem}
 There is an explicit, smooth, bijective, $SL(2,\lip)$ equivariant correspondence between the following:
 \begin{enumerate}[label=\roman*)]
 \item Lipschitz spinors $S\lip$,
 \item spin multiflags $\widetilde{\mathcal{MF}}^n$,
 \item The space $S$Hor$^{n+2}$ of spin decorated horospheres in $\hyp^{n+2}$.
\end{enumerate}
\end{theorem}
\noindent\textit{Proof.} We prove ii), iii) are equivalent to i).\\\noindent
i) $\Leftrightarrow$ ii): $\Phi_1$ is a smooth double cover (as shown in Theorem \ref{thm:Phi1SmoothDoubleCover}), which implies it lifts to a diffeomorphism. It is $SL(2,\lip)$ equivariant by Theorem \ref{Phi1Equivariance}, which also lifts to the equivariant map $\Tilde{\Phi}_1$.
\\\\
i) $\Leftrightarrow$ iii): Theorem \ref{DoubleCoverTheorem} tells us $\Phi$ is a smooth double cover of $D\text{Hor}^{n+2}$, which lifts to a diffeomorphism $S\lip \to S\text{Hor}^{n+2}$. It is made explicit (up to spin) by Theorem \ref{HorosphereCoordinateTheorem}, and $SL(2,\lip)$ equivariant by lifting the equivariance of Theorem \ref{thm:PhiEquivariant}.
\qed
\subsubsection{The Case of 2 Dimensions}
Spin decorations in $\hyp^2$ are an outlier, because the space of positively oriented orthonormal frames based at a point on a horocycle consists of a single point. The spin `double cover' is then two points and is thus disconnected, which means there is no path from one spin frame to another. We can remedy this by embedding $\U^2 \subset \U^3$, as was set up in Definition \ref{def:PlanarDecoration}.
\begin{definition}[Planar Spin Decorations]
\label{defn:PlanarSpinDecs}
Embed $\U^2 \subset \U^3$ as in Definition \ref{def:PlanarDecoration}. Given a horocycle $H$ in $\U^2 \subset \U^3$, a \emph{planar spin decoration} on $H$ is a lift of a planar decoration on $H$ to the spin double cover. Each horocycle $H$ of $\U^2$ has two planar spin decorations, related by a $2\pi$ rotation. \null\hfill\defn
\end{definition}
\noindent We can then define frames and spin frames as usual by considering these planar spin decorations as a subset of the spin decorations on $\U^3$.
\subsection{Lambda Lengths}
\label{LambdaLengthsSection}
In \cite{Mathews_Spinors_horospheres}, complex lambda lengths are defined geometrically and shown to have a nice algebraic interpretation as a symplectic form on the space of complex spinors. This relationship can be generalised; but while we already have a natural generalisation of the symplectic form (the bracket defined in Section \ref{sec:TheBracket}), we need to define the appropriate notion of `higher dimensional lambda length' between decorated horospheres. 

Informally, the appropriate geometric picture of a lambda length (or $\lambda$) is given by pushing the inward (resp. outward) spin frame on one horosphere to the outward (resp. inward) spin frame on the other, parallel transporting along the geodesic between their centres. Then, when the frames are based at the same point, there will be some rotation in the (now identified) tangent spaces of the horospheres that aligns the spin decorations. The lambda length is an element of $\lip$ representing this combined translation and rotation.

\begin{definition}
 We define two (closely related) actions of $\lip$ on spin decorations, with the difference being essentially spinorial. Because $\lip$ is defined to act on $\para$, using it to act on the tangent spaces of horospheres requires an identification of the tangent space with $\para$. In one instance we make this identification using the spin frame $W^{in}$ from the spin decoration $W$, choosing $i_j \in \para$ to point along the $i_j$ component of $W$, and $1 \in \para$ to point along $1^{in}$. In the other, we identify $\para$ along the components of $W^{out}$.
 
 For $a \in \lip$, the {\normalfont inwards action} on the spin decorated horosphere $(H,W)$ sends it to $a^{in}.(H,W) = (H,a^{in}.W)$. The spin decoration $a^{in}.W = (a^{in}.W^{in},a^{in}.W^{out})$ is the result of acting on $W^{in}$ and $W^{out}$ with $\Tilde{\sigma}\left(a/|a|\right)$ under the identification of the tangent space along $W^{in}$, where $\Tilde{\sigma}(\cdot)$ is the lift of $\sigma(\cdot)$ from a map $\lip \to SO(n+1)$ to a map $ \lip \to Spin(n+1)$. The {\normalfont outwards action} is almost the same, but instead identifies the tangent space with $\para$ along $W^{out}$.
 \null\hfill\defn
\end{definition}

\begin{definition}[Lambda Length]
\label{defn:LambdaLength}
 Consider a pair of spin decorated horospheres $(H_1,W_1)$ and $(H_2,W_2)$. If the centres of $H_1$, $H_2$ are distinct, the lambda length $\lambda_{12}$ from $H_1$ to $H_2$ is defined to be the element of $\lip$ with magnitude $e^{d/2}$, where $d$ is the oriented hyperbolic distance from $H_1$ to $H_2$ along their common perpendicular geodesic, and whose inwards action on spin decorations sends $W_1^{in}$ to $\Tilde{\sigma}(\lambda_{12})(W_1^{in}) = W_2^{out}$ after parallel transporting $W_1$ along the common perpendicular to be based at the same point as $W_2$, carrying the inwards identification along with it. If $H_1$, $H_2$ share a centre, then $\lambda_{12}=0$. \null\hfill\defn
\end{definition}
\noindent Although this is a somewhat intuitive geometric description, it isn't very explicit. For computations, the following characterisation proves more useful.
\begin{proposition}
 The lambda length $\lambda_{12}$ between two spin decorated horospheres $(H_1,W_1)$ and $(H_2,W_2)$ can be written explicitly in the form
 \[\lambda_{12} = e^{d/2}e^{\theta_k V_k/2}... e^{\theta_2 V_2/2}e^{\theta_1 V_1/2},\]
 where the simple rotations $e^{\theta_jV_j/2}$ are applied right to left. Here, $d$ is the signed hyperbolic distance from $H_1$ to $H_2$ along the geodesic between their centres (as in the definition above), the $\theta_j$ are angles in $[0,4\pi)$, and the $V_j \in \R^{0,n}$ are unit vectors; the $e^{\theta_j V_j/2}$ correspond to an ordered series of rotations by $\theta_j$ within the $2$-planes spanned by $\{\R,V_1\},\{\R,V_2\},...,\{\R,V_k\}$, from $1$ towards $V_j$, under the identification of $\para$ with the tangent space using the spin frame $W^{in}$. These rotations compose to the action $\Tilde{\sigma}(\lambda_{12})$ sending $W_1^{in}$ to $W_2^{out}$.
\end{proposition}
\noindent\textit{Proof.} $e^{d/2}$ is the magnitude of $\lambda_{12}$ by definition. The action on $W_1^{in}$ is by $\Tilde{\sigma}(\lambda_{12})$, but by Corollary \ref{IsometryGeneratedByRotationsCorollary}, the action of any $\sigma(\alpha)$ for $\alpha \in \lip$ can be generated by the product of simple rotations in the 2-planes spanned by $\{\R,V_j\}$ for some $V_j \in \R^{0,n}$. This statement lifts, so any $\Tilde{\sigma}(\alpha)$ can be generated by the product of the lifts of simple rotations.\qed
\\\\
Before moving on, a quick lemma.
\begin{lemma}
 \label{ParavectorsSquareLemma}
 The only paravectors $\pV \in \para$ satisfying $\pV^2=-1$ are the subset of \textit{unit vectors}; that is, the $V \in \R^{0,n} \subset \lip$ satisfying $|V|=1$.
\end{lemma}
\noindent\textit{Proof.} Consider expanding $\pV^2$ for a generic paravector $\pV=V_0+V$ (where $V$ is the vector component):
\[\left(V_0+V\right)\left(V_0+V\right) = V^2+2V_0V+V_0^2=-1.\]
If $V_0V$ is non-zero the left-hand side has complex components, so this cross-term must be zero. This implies $V_0=0$, or $V=0$; but if $V=0$ then $V_0^2=-1$, which is not satisfied for any real $V_0$. We conclude $V_0=0$ and therefore $\pV^2=V^2=-1$, and as $|V|^2=-V^2$ we see $|V|=1$.\qed
\subsubsection{Computing the Lambda Length}
The spin nature of the lambda length can make it difficult to handle signs; to resolve this we compute a simple example, and extend that case to the general one.
\begin{lemma}
 \label{FirstLambdaLengthLemma}
 The lambda length $\lambda_{12}$ from $\Tilde{\Phi} (1,0)$ to $\Tilde{\Phi} (0,1)$ is $1$.
\end{lemma}
\noindent\textit{Proof.} We give the horospheres and spin decorations names:
\begin{align*}
\Tilde{\Phi}(1,0) &= (H_1,W_1),\ \Tilde{\Phi}(0,1) = (H_2,W_2).
\end{align*}
Both horospheres go through the point $(1;0)_{\U}$ (which looks, in the upper half-space, like a point at Euclidean distance 1 above the origin of the boundary), and the two spin decorations project to the same orthonormal frame $(i_1,i_2,...,i_n)$ on $\partial \U^{n+2}$. This implies the spin decorations either align, or differ by a $2\pi$ rotation; if they are aligned then $\lambda_{12}=1$, and if they differ by a $2\pi$ rotation $\lambda_{12}=-1$.

Consider a smooth path of horospheres from $(1,0)$ to $(0,1)$ given by $(\cos\theta,\sin\theta)$ for $\theta \in [0,\pi/2]$. This is equivalently given by acting on $(1,0)$ with
\[M_{\theta}=\begin{pmatrix}
 \cos\theta & -\sin\theta\\
 \sin\theta & \cos\theta
\end{pmatrix} \in SL(2,\lip).\]
This action is an elliptic isometry, a rotation about a geodesic hyperplane. We can see this by examining the fixed points; consider the set\footnote{It is easy to check $\infty$ is not fixed.} of $\pV \in \partial \U^{n+2}$ such that $M_{\theta}.\pV = \pV$.
\begin{align*}
\pV &= \left(\cos\theta\pV - \sin\theta\right)\left(\sin\theta\pV + \cos\theta\right)^{-1} \implies \pV^2 = -1.
\end{align*}
From Lemma \ref{ParavectorsSquareLemma}, this implies $\pV$ is a unit vector in $\R^{0,n} \subset \para$, and the set of these unit vectors form the boundary of a geodesic hyperplane. The path $M_{\theta}$ fixes the points $i_1,-i_1,i_2,-i_2$ on the boundary and therefore also fixes the geodesics $\gamma_1(-i_1\to i_1)$ and $\gamma_2(-i_2 \to i_2)$. The point $(1;0)_{\U}$ is on both geodesics, and must stay on both under the action. As the only point where the geodesics meet is $(1;0)_{\U}$, it must be fixed. Because there is a fixed point in the interior, the action is elliptic.

Geometrically, this action looks like a rotation by $2\theta$ about the oriented geodesic hyperplane $\Pi_{\gamma}$ bounded by the unit $(n-1)$-sphere of pure vectors $ \R^{0,n} \subset \partial\U^{n+2}$. Note that if we restrict to the subspace spanned by the vectors $1$, $i_j$, and $i_{n+1}$ in $\U^{n+2} \subset \$\R^{n+2}$ then this looks like rotation about the oriented geodesic from $(-i_j\to i_j)$, which reduces us to the $n=1$ case given in \cite{Mathews_Spinors_horospheres} p.18. The frame $W_1^{in}$ based at $(1;0)_{\U}$ is thus rotated by $\pi$ about $\Pi_{\gamma}$ to become $W_2^{in}$. 

To arrive at the associated outward spin frame $W_2^{out}$, we then rotate $W_2^{in}$ by $-\pi$ about the oriented hyperplane defined by the decorations (as in Definition \ref{SpinFrameDefinition}). But the decorations for $(H_2,W_2)$ are given by $i_1,i_2,...,i_n$, so this is a rotation about precisely the same hyperplane $\Pi_{\gamma}$! In particular, as the decorations all point along the $i_j$, we get the same orientation as $\Pi_{\gamma}$ also. So rotating to the outwards spin frame means rotating by $-\pi$, back to the start, and $W_1^{in}=W_2^{out}$ at $(0;1)_{\U}$. As the spin frames align, the lambda length $(H_1,W_1) \to (H_2,W_2)$ is $1$.\qed 
\\\\
We can compute the general case via the equivariant action of $SL(2,\lip)$, starting with a slightly more general example. Unless stated otherwise, in the following lemma when we write elements in the tangent spaces to horospheres as elements of $\para$ the identification is always taken to be the generic basis inherited from the boundary $\partial\U^{n+2}$, or equivalently the restriction $z=$ \textit{constant} of $\U^{n+2}$ to the horosphere. 
\begin{lemma}
 \label{LambdaLengthParavectorCaseLemma}
 The lambda length $\lambda_{13}$ from the spin decorated horosphere $(H_1,W_1) = \Tilde{\Phi} (1,0)$ to $(H_3,W_3) = \Tilde{\Phi} (0,\pU)$, $\pU \in \para$ is $\pU$.
\end{lemma}
\noindent\textit{Proof.} We can write $\pU$ as $|\pU|e^{\theta V}$ for a unit vector $V \in \R^{0,n}$, $|V|^2=1$ (Proposition \ref{EulerGeneralisationLemma}). We then consider the magnitude and rotational components of the lambda length independently, a rotation by $\pU/|\pU|$ and magnitude $|\pU|$. 

First we send $(1,0)$ to $(0,1)$, which (as the lambda length between them is 1 by Lemma \ref{FirstLambdaLengthLemma}) matches the inward spin frame of $\Tilde{\Phi}(1,0)$ to the outward spin frame of $\Tilde{\Phi}(0,1)$. If $\pU=|\pU|$ is real and positive then there is no rotation, and the magnitude is a function of the distance from $1$ to $\pU^{-2}$. The signed hyperbolic distance along the geodesic from the centre of $H_1$ (which is $\infty$) to the centre of $H_3$ (which is 0) is then:
\[d = \int^{1}_{|\pU|^{-2}}\frac{dz}{z}=2\ln|\pU|,\]
and $\lambda=e^{d/2}=|\pU|=\pU$. 

We therefore assume $\pU \notin \R^+$, and we'll also assume $|\pU|=1$ as we can simply rescale afterwards. Acting by
\[M_1=\begin{pmatrix}
 \pU^{-1} & 0\\
 0 & \pU
\end{pmatrix} \in SL(2,\lip)\]
sends $(0,1)$ to $(0,\pU)$ and therefore sends $\Tilde{\Phi}(0,1)=(H_2,W_2)$ from Lemma \ref{FirstLambdaLengthLemma} to $(H_3,W_3)$.
We define a path in $SL(2,\lip)$ from the identity to the action $M_1$:
\[M_t = \begin{pmatrix}
 e^{-t\theta V} & 0 \\
 0 & e^{t\theta V}
\end{pmatrix},\]
for $t \in [0,1]$. As shorthand, write $\pU_t=e^{t\theta V}$; note also that, as a unit paravector, $\pU'_t = \ol{\pU}_t = \pU^{-1}_t$. As per Lemma \ref{ParavectorActionFixedSpaceLemma}, the action of $M_t$ as a M\"{o}bius transformation on $\pV \in \para$ will be a rotation $\sigma(\pU^{-1}_t)$ in the plane spanned by $\R$ and $\pU^{-1}_t$; denote the span by $\{\R,\pU^{-1}_t\}$.
\[M_t.\pV = \sigma(\pU^{-1}_t)(\pV) = e^{-t\theta V}\pV e^{-t\theta V}.\]
But this will always be the plane spanned by $\{\R,\pU\}$, since
\[\pU^{-1}_t = \cos(t\theta) - V\sin(t\theta) \in \{\R,\pU\},\]
and therefore $\{\R,\pU^{-1}_t\}$ spans the same plane as $\{\R,\pU\}$ unless $\pU^{-1}_t$ is real. If $\pU^{-1}_t$ is real (implying $\pU^{-1}_t = \cos(t\theta)$), it is simple to check that the derivative lies within the correct plane and is orthogonal to $\pU^{-1}_t$, so the path of isometries $M_t$ really is just a rotation in the $\{\R,\pU\}$-plane. 

If we act on a vector in the plane $\{\R,\pU\}$:
\[\sigma(\pU^{-1}_t)(A+B\pU) = Ae^{-t\theta V}e^{-t\theta V}+Be^{-t\theta V}e^{\theta V}e^{-t\theta V},\]
then, applying the Baker-Campbell-Hausdorff formula, we see all the commutators vanish (as all the exponents are multiples of $V$):
\[\sigma(\pU^{-1}_t)(A+B\pU) = Ae^{-2t\theta V}+Be^{\theta V-2t\theta V}= (A+B\pU)e^{-2t\theta V}.\]
This is precisely a rotation by $-2t\theta$ in the plane spanned by $\R$ and $\pU^{-1}$ (Lemma \ref{ParavectorActionFixedSpaceLemma}), from $\R^+$ towards $\pU$. Thus, the action of $M_1$ sending the spin frame of $\Tilde{\Phi}(0,1)$ to $\Tilde{\Phi}(0,\pU)$ is a rotation in the $\{\R,\pU\}$-plane by $-2\theta \mod 4\pi$ from $\R^+$ towards $\pU$.

But we aren't trying to map the spin frame of $\Tilde{\Phi}(0,1)$ to the spin frame of $\Tilde{\Phi}(0,\pU)$; we want to go from $\Tilde{\Phi}(1,0)$ to $\Tilde{\Phi}(0,\pU)$, and the lambda length is computed in the $W_1^{in}$ identification. As shown in Lemma \ref{FirstLambdaLengthLemma}, the inwards spin frame of $\Tilde{\Phi}(1,0)$ is identified with the outwards spin frame of $\Tilde{\Phi}(0,1)$. But, in our orientation conventions, $1^{in}$ of $W^{in}$ will point in the opposite direction to $1^{out}$ of $W^{out}$. This makes the rotation by $-2\theta$ from $\R^+$ towards $\pU$ a rotation by $-2\theta$ from $\R^-$ towards $\pU$, or a rotation by $2\theta$ from $\R^{+}$ towards $\pU$ by our inwards action, which is the rotation taking the spin frame of $\Tilde{\Phi}(1,0)$ to $\Tilde{\Phi}(0,\pU)$.

To compute the lambda length $\lambda_{13}$ we relax the assumption $|\pU|=1$ and integrate (as above) to find the hyperbolic distance as $2\ln|\pU|$, giving
\[\lambda_{13} = e^{2\ln|\pU|/2}e^{2\theta V/2}= |\pU|e^{\theta V}=\pU. \tag*{\qed}\]
It isn't a huge leap to the next generalisation; because $\lip$ is generated by paravectors, we can factor a general $\eta \in \lip$ into paravectors and apply the rotations one after the other. 
\begin{lemma}
 \label{HalfwayLambdaLengthComputationLemma}
 The lambda length between the spin decorated horospheres $(H_1,W_1) = \Tilde{\Phi} (1,0)$ and $(H_4,W_4) = \Tilde{\Phi} (0,\eta)$ is $\eta$.
\end{lemma}
\noindent\textit{Proof.} Similar to the previous case, we can find the hyperbolic distance between these by integrating along the $z$ axis, giving $d = \int_{|\eta|^{-2}}^{1}\frac{dz}{z} = 2\ln|\eta|$ and $e^{d/2} = |\eta|$. 

Parallel transport the inward frame $W_1^{in}$ along $\gamma$ to be based at the intersection of $\gamma$ with $H_2$. To be precise about spin frames, we need to act by a path of elements in $SL(2,\lip)$. We already know the lambda length from $(1,0)$ to $(0,1)$ involves no distance or rotation, so we consider going from $(0,1)$ to $(0,\eta)$. Action by
\[M_1=\begin{pmatrix}
 \eta^{-1*} & 0\\
 0 & \eta
\end{pmatrix} \in SL(2,\lip)\]
sends $(0,1)$ to $(0,\eta)$ and therefore sends $\Tilde{\Phi}(0,1)=(H_2,W_2)$ from Lemma \ref{FirstLambdaLengthLemma} to $(H_4,W_4)$. We factor $\eta/|\eta|$ as a product of paravectors and write them in exponential form, giving $\eta/|\eta| = e^{\theta_kV_k}...e^{\theta_1V_1}$
where the $V_j$ are unit vectors in $\R^{0,n}$. We can construct a path through $SL(2,\lip)$ from the identity to $M_1$ by defining an appropriately staggered set of parameters $t_j$:
\[t_{j+1} = t_{j+1}(t) = 
\begin{cases}
0 & t < \frac{j}{k}, \\
kt-j & \frac{j}{k} \leq t \leq \frac{j+1}{k}, \\
1 & t > \frac{j+1}{k},
\end{cases}\]
for $j \in \{0,...,k-1\}$, which gives us the path of transformations
\[M_{t} = \begin{pmatrix}
 e^{-t_k\theta_kV_k}...e^{-t_1\theta_1V_1} & 0 \\
 0 & e^{t_k\theta_kV_k}...e^{t_1\theta_1V_1}
\end{pmatrix}\]
from the identity matrix $I$ to $M_1$, for $t \in [0,1]$. The action on $\para$ will be by $\sigma(\cdot)$; for $\pV \in \para$,
\begin{align*}
M_t.\pV &= e^{-t_k\theta_kV_k}...e^{-t_1\theta_1V_1}\pV e^{-t_1\theta_1V_1}...e^{-t_k\theta_kV_k} = \sigma\left(e^{-t_k\theta_kV_k}\right)...\sigma\left(e^{-t_1\theta_1V_1}\right)(\pV).
\end{align*}
From Lemma \ref{ParavectorActionFixedSpaceLemma} and Lemma \ref{LambdaLengthParavectorCaseLemma} we see this is a rotation by $2\theta_1 \mod 4\pi$ in the $\{\R,V_1\}$ plane, followed by a $2\theta_{2} \mod 4\pi$ rotation in the $\{\R,V_2\}$ plane, and so on. By the same argument as Lemma \ref{LambdaLengthParavectorCaseLemma}, this implies a total lambda length of
\[\lambda = e^{2\ln|\eta|/2}e^{2\theta_kV_k/2}...e^{2\theta_1V_1/2} = |\eta|e^{\theta_kV_k}...e^{\theta_1V_1}=\eta. \tag*{\qed}\]
Given a completely general pair of Lipschitz spinors, we can transform them via hyperbolic isometries (which leave the lambda length invariant) to the form of Lemma \ref{HalfwayLambdaLengthComputationLemma}.
\begin{theorem}
\label{thm:LambdaLengthBracket}
 The lambda length between two spin decorated horospheres with spinor coordinates $\kappa_1 = (\xi_1,\eta_1)$ and $\kappa_2 = (\xi_2,\eta_2)$ is given by taking their bracket. Properly, $\lambda_{12} = \{\kappa_1,\kappa_2\}$.
\end{theorem}
\noindent\textit{Proof.} We can apply an isometry $A \in SL(2,\lip)$ to send $\kappa_1$ to $(1,0)$:
\[A.\kappa_1 = \begin{pmatrix}
 \frac{1}{2}\xi_1^{-1} & \frac{1}{2}\eta_1^{-1} \\
 -\eta_1^* & \xi_1^*
\end{pmatrix}
\begin{pmatrix} 
\xi_1 \\
\eta_1
\end{pmatrix}= \begin{pmatrix}
 1 \\
 0
\end{pmatrix}.\]
This isometry affects $\kappa_2$ as follows:
\[A.\kappa_2 = \begin{pmatrix}
 \frac{1}{2}\xi_1^{-1} & \frac{1}{2}\eta_1^{-1} \\
 -\eta_1^* & \xi_1^*
\end{pmatrix}
\begin{pmatrix} 
\xi_2 \\
\eta_2
\end{pmatrix} = 
\begin{pmatrix}
 \frac{1}{2}\xi_1^{-1}\xi_2+\frac{1}{2}\eta_1^{-1}\eta_2 \\
 -\eta_1^*\xi_2+\xi_1^*\eta_2
\end{pmatrix}.\]
We can then apply the appropriate translation $B \in SL(2,\lip)$ to send the centre of $A.\kappa_2$ to 0:
\[B.A.\kappa_2 = 
\begin{pmatrix}
 1 & -(\frac{1}{2}\xi_1^{-1}\xi_2+\frac{1}{2}\eta_1^{-1}\eta_2)(-\eta_1^*\xi_2+\xi_1^*\eta_2)^{-1} \\
 0 & 1 
\end{pmatrix}
\begin{pmatrix}
 \frac{1}{2}\xi_1^{-1}\xi_2+\frac{1}{2}\eta_1^{-1}\eta_2 \\
 -\eta_1^*\xi_2+\xi_1^*\eta_2
\end{pmatrix}
=
\begin{pmatrix}
 0 \\
 \xi_1^*\eta_2-\eta_1^*\xi_2
\end{pmatrix}.\]
This translation leaves $A.\kappa_1$ invariant. The lambda length can now be computed from Lemma \ref{HalfwayLambdaLengthComputationLemma} as
\[\lambda_{12} = \xi_1^*\eta_2 - \eta_1^*\xi_2 = \{\kappa_1,\kappa_2\}.\tag*{\qed}\]
In its definition, the lambda length picks out a canonical basis for the tangent space (given by the decorations) to allow the action of $\lip$ to be defined. But, as it turns out, the lambda length is independent of the coordinates on the tangent space, up to orientation-preserving and orthogonal changes of basis (meaning it maps positively oriented orthonormal bases to positively oriented orthonormal bases).
\begin{proposition}
The lambda length is invariant under orientation-preserving orthonormal changes of basis of the shared tangent space of the horospheres.
\end{proposition}
\noindent\textit{Proof.} Take a pair of spin decorated horospheres $(H_1,W_1)$ and $(H_2,W_2)$ described by Lipschitz spinors $\kappa_1$ and $\kappa_2$ respectively. Consider the lambda length $\lambda_{12}$ from $H_1$ to $H_2$, so $W_1^{in}$ is identified with $\para$. We change the identification of the tangent space with $\para$ through the action $\sigma(b)$ of some unit element $b \in \lip$; as $\sigma$ gives all orientation-preserving isometries, any orientation-preserving orthogonal change of basis can be written in this form.

Via equivariance, we can describe this rotation on the tangent space by acting on $\kappa_1$ and $\kappa_2$ with the element
\[B=\begin{pmatrix}
 b & 0 \\
 0 & b^{*-1}
\end{pmatrix} \in SL(2,\lip).\]
This action on $\kappa_1$ and $\kappa_2$ leaves bracket, and thus the lambda length, invariant, as shown in Lemma \ref{SesquimorphismLemma}.

\subsection{The Generalised Ptolemy Relations}
\label{sec:PtolemyRelns}
In the complex $n=1$ case described in \cite{Mathews_Spinors_horospheres}, a Ptolemy equation relating the lambda lengths of four horospheres is derived from the Pl\"{u}cker relations of a matrix built from the spinor coordinates of the horospheres. This relation can be generalised, but since we're now working over noncommutative algebras we cannot use the commutative Pl\"{u}cker relations. We appeal instead to the theory of quasideterminants (see for example \cite{gelfand2004quasideterminants}), due to Gelfand and Retakh. 
\begin{definition}[Quasideterminant of a Matrix]
\label{def:quasideterminant}
 Given an $m \times m$ matrix $A$ with components $A_{jk}$ over a (possibly noncommutative) ring, if $A$ has a well defined inverse with components $(A^{-1})_{jk}$, $1 \leq j,k \leq m$, the $(j,k)$-quasideterminant of $A$ is defined as $|A|_{jk}=((A^{-1})_{kj})^{-1}$ whenever the element $(A^{-1})_{kj}$ is invertible. \null\hfill\defn
\end{definition}
\noindent Given an invertible $2\times 2$ matrix with possibly noncommutative entries:
\[A = \begin{pmatrix}
a & b \\
c & d
\end{pmatrix}\]
The inverse $A^{-1}$ is the following:
\[A^{-1} = \begin{pmatrix}
(a-bd^{-1}c)^{-1} & (c-db^{-1}a)^{-1} \\
(b-ac^{-1}d)^{-1} & (d-ca^{-1}b)^{-1}
\end{pmatrix},\]
and then the various quasideterminants are just the inverses of these entries. Specifically, the quasideterminants can be arranged into the following matrix:
\[|A| = \begin{pmatrix}
 (a - bd^{-1}c) & (b - ac^{-1}d) \\
 (c - db^{-1}a) & (d - ca^{-1}b)
\end{pmatrix}.\]
\begin{definition}
\label{not:DeleteColj}Given a matrix $A$, the matrix $A^{(j)}$ is $A$ with the $j$th column deleted. \label{not:LeftQuasiPlucker}The left quasi-Pl\"{u}cker coordinates are defined on the $m\times (m+1)$ matrix $A$ as $q_{jk}^{(s)}(A) = |A^{(k)}|_{sj}^{-1}|A^{(j)}_{sk}|$, with $1 \leq j,s \leq m$, $0 \leq k \leq m$. The quasi-Pl\"{u}cker coordinates $q_{jk}^{(s)}(A)$ are actually independent of the choice of $s$ (Proposition 4.2.1 of \cite{gelfand2004quasideterminants}). 

For a more general $m \times \ell$ matrix $B$, $m\leq \ell$, choose a set of indices $1 \leq j,k,j_1,...,j_{m-1} \leq \ell$ with $j \not\in I = \{j_1,j_2,...,j_{m-1}\}$. Let $C$ be the $m \times (m+1)$ submatrix of the columns $j,k,j_1,...,j_{m-1}$ drawn from $B$. The left quasi-Pl\"{u}cker coordinates are $q_{jk}^I(B) = q_{jk}^{(s)}(C)$.\null\hfill\defn
\end{definition}
\noindent To relate this machinery to the bracket and lambda lengths, we recall that a matrix 
\[A_{\ell j} = \begin{pmatrix}
 \xi_{\ell} & \xi_j \\
 \eta_{\ell} & \eta_j
\end{pmatrix}\]
has pseudo-determinant $\Delta_{\ell j} = \xi_{\ell}^*\eta_j - \eta_{\ell}^*\xi_j$.
\begin{lemma}
\label{lem:QuasiPluckerPdet}
Consider a matrix whose columns are four Lipschitz spinors:
\[A = \begin{pmatrix}
 \xi_1 & \xi_2 & \xi_3 & \xi_4 \\
 \eta_1 & \eta_2 & \eta_3 & \eta_4
\end{pmatrix}.\]
The left quasi-Pl\"{u}cker coordinates can be constructed from the pseudo-determinant as $q^{\{k\}}_{\ell j}=\Delta_{k\ell}^{-1}\Delta_{kj}$.
\end{lemma}
\noindent\textit{Proof.} 
Define the submatrices
\[A_{\ell j} = \begin{pmatrix}
 \xi_{\ell} & \xi_j \\
 \eta_{\ell} & \eta_j
\end{pmatrix}.\] 
Computing the left quasi-Pl\"{u}cker coordinates (arbitrarily picking $s=1$) gives us
\[q^{\{k\}}_{\ell j} = |A_{\ell k}|_{1\ell}^{-1}|A_{jk}|_{1j} = (\xi_{\ell} - \xi_k\eta_k^{-1}\eta_{\ell})^{-1}(\xi_j - \xi_k\eta_k^{-1}\eta_j).\]
Now we compute the right-hand side and compare.
\[\Delta_{k\ell}^{-1}\Delta_{kj}=(\xi_k^*\eta_{\ell} - \eta_k^*\xi_{\ell})^{-1}(\xi_k^*\eta_j - \eta_k^*\xi_j).\]
We can insert a factor of $1=\eta_k^*\eta_k^{*-1}$ between the two factors and rearrange terms to find
\[\Delta_{k\ell}^{-1}\Delta_{kj}=(\xi_{\ell}-(\xi_k\eta_k^{-1})^*\eta_{\ell})^{-1}(\xi_j-(\xi_k\eta_k^{-1})^*\eta_j)\]
and because the $(\xi_k\eta_k^{-1})$ terms are paravectors they're invariant under the $\cdot^*$ operation.
\[\Delta_{k\ell}^{-1}\Delta_{kj}=(\xi_{\ell}-\xi_k\eta_k^{-1}\eta_{\ell})^{-1}(\xi_j-\xi_k\eta_k^{-1}\eta_j) = q_{\ell j}^{\{k\}}.\tag*{\qed}\]
With this, we can apply the properties of these coordinates to lambda lengths (as these descend from the pseudo-determinant); one of those properties is a noncommutative Pl\"{u}cker relation.
\begin{theorem}
\label{PtolemyTheorem}
 Given four spin decorated horospheres $(H_j,W_j)$, $j \in \{1,2,3,4\}$ in $\hyp^{n+2}$ with lambda lengths $\lambda_{jk}$ between $(H_j,W_j)$, $(H_k,W_k)$, their lambda lengths satisfy
 \begin{equation}
 \label{PtolemyEquation}
\lambda_{31}^{-1}\lambda_{23}^*\lambda_{42}^{-1}\lambda_{14}^*+\lambda_{31}^{-1}\lambda_{43}^*\lambda_{24}^{-1}\lambda_{12}^*=1.
 \end{equation}
This is the noncommutative counterpart to the Ptolemy equation.
 \end{theorem}
\noindent\textit{Proof.} The noncommutative (2,4) Pl\"{u}cker relations (\cite{gelfand2004quasideterminants} Theorem 4.4.2) are $q^{\{k\}}_{\ell j}q_{j\ell}^{\{l\}}+q_{\ell l}^{\{k\}}q_{l\ell}^{\{j\}}=1$, which we can substitute with pseudo-determinants by Lemma \ref{lem:QuasiPluckerPdet}.
\[\Delta_{k\ell}^{-1}\Delta_{kj}\Delta_{lj}^{-1}\Delta_{l\ell} + \Delta_{k\ell}^{-1}\Delta_{kl}\Delta_{jl}^{-1}\Delta_{j\ell}=1.\]
However, the pseudo-determinant applied to a matrix with two Lipschitz spinors as columns is just the lambda length between them.
\[\lambda_{k\ell}^{-1}\lambda_{kj}\lambda_{lj}^{-1}\lambda_{l\ell}+\lambda_{k\ell}^{-1}\lambda_{kl}\lambda_{jl}^{-1}\lambda_{j\ell}=1.\]
By substituting $\ell,j,k,l=1,2,3,4$ and applying the identity $\lambda_{\ell j}=-\lambda_{j\ell}^*$, we get the claimed result. \qed
\\\\
Rearranging the Ptolemy equation as
\begin{equation}
 \label{LambdaHolonomyEquation}
 \lambda_{kj}\lambda_{lj}^{-1} \lambda_{l\ell}+\lambda_{kl}\lambda_{jl}^{-1}\lambda_{j\ell}=\lambda_{k\ell},
\end{equation}
we get a kind of holonomy relation, where the lambda length from $k$ to $\ell$ is the sum of two paths going $k \to j \to l \to \ell$ and $k \to l \to j \to \ell$.

Another useful property of these quasi-Pl\"{u}cker coordinates is given in Theorem 4.4.1 of \cite{gelfand2004quasideterminants}, the skew-symmetry relation $q^{\{k\}}_{\ell j}q_{jk}^{\{\ell\}}q_{k\ell}^{\{j\}}=-1$. This translates into a relation on lambda lengths.
\begin{proposition}
 Given the lambda lengths $\lambda_{12}$, $\lambda_{13}$, $\lambda_{23}$ between three spin decorated horospheres $(H_j,W_j)$ for $j \in \{1,2,3\}$, the following identity holds:
 \[\lambda_{12}\lambda_{32}^{-1}\lambda_{31}=(\lambda_{12}\lambda_{32}^{-1}\lambda_{31})^*.\]
 That is, $\lambda_{12}\lambda_{32}^{-1}\lambda_{31}$ is invariant under the reversal map.
\end{proposition}
\noindent\textit{Proof.} Using the above skew-symmetry relation, we substitute lambda lengths to find
\[\lambda_{31}^{-1}\lambda_{32}\lambda_{12}^{-1} \lambda_{13}\lambda_{23}^{-1}\lambda_{21}=-1.\]
After some rearranging, we find the claimed identity.\qed
\\\\
We can also rearrange this skew-symmetry relation and find the following:
\[\lambda_{13}\lambda_{23}^{-1}\lambda_{21} + \lambda_{12}\lambda_{32}^{-1}\lambda_{31}=0=\lambda_{11}.\]
This is a special case of the holonomy relation (\ref{LambdaHolonomyEquation}); just set $k=\ell$.

\appendix
\section*{Notation}
\label{sec:Notation}
\small
\begin{tabularx}{\textwidth} { 
   >{\raggedright\arraybackslash}l
   >{\raggedright\arraybackslash}X}
 \hline
 {\small\textbf{General}} & \\
 $\N$ & The natural numbers $\{0,1,2,...\}$ \\
 $\Z$, $\Z_{>0}$ & The integers (resp. the positive integers $\{1,2,3,...\}$) \\
 $\Z_n$ & The ring of integers mod $n$ \\
 $\Q$ & The rationals \\
 $\R$ & The reals \\
 $\R^+$, $\R^-$ & The open interval $(0,\infty)$ (resp. $(-\infty,0)$) \\
 $\C$ & The complex numbers \\
 $\hyp$ & The quaternions (sec. \ref{sec:ClifAlgLadder} only) \\
 $\mathbb{O}$ & The octonions \\
 $I$ & A multi-index $\{j_1,...,j_k\}$  (p. \pageref{not:multiindex})\\
 $e^X$ & The exponential of $X$ (p. \pageref{not:exponentialmap})\\
 $\text{Im}(f)$ & The image of the function $f$ \\
 $\theta,\varphi,\psi,$... & Generic angle variables\\
 $\lfloor \cdot \rfloor$, $\lceil \cdot \rceil$ & The floor and ceiling functions, respectively\\
 $|\cdot|^2$ & Norm or magnitude squared (context dependent, see p. \pageref{defn:CliffordNorm})\\
 $[X_a]_j$ & The $j$ component of an object $X_a$\\
 $x_{\bullet}$ & A generic placeholder for an indexed object $x$\\
\hline

\textbf{Clifford Algebra} & \\
$\R^{p,q}$ & Real vector space with $p$ positive, $q$ negative unit generators (p. \pageref{not:Rpq})\\
 $Q_{p,q}$, $N_{p,q}$ & The quadratic form (resp. norm) on $\R^{p,q}$ (p. \pageref{not:Rpq})\\
$\clifn{p,q}$ & Clifford algebra on $\R^{p,q}$ (p. \pageref{not:Clifpq})\\
$\cdot'$ & Grade involution on a Clifford algebra (ch. 1-4) (p. \pageref{not:GradeInv})\\
$\cdot^*$ & Reversal involution on a Clifford algebra (ch. 1-4) (p. \pageref{not:RevInv})\\
$\ol{\cdot}$ & Clifford conjugation on a Clifford algebra (ch. 1-4) (p. \pageref{not:CliffConj})\\
$\pX \cdot \pY$, $x \cdot y$ & The dot product on $\$\R^{q+1,p}$ and $\clifn{p,q}$ (p. \pageref{defn:ParavectorInnerProduct})\\
$\clifn{p,q}^{\pm}$ & The positive and negative eigenspaces of $\clifn{p,q}$ under $\cdot'$ (p. \pageref{not:Clifpqpm})\\
$X^{\pm}$ & $X \cap \clifn{p,q}^{\pm}$ for a space $X$\\ 
$\clifn{p,q}^{\times}$  & The invertible subgroup of $\clifn{p,q}$ (p. \pageref{not:ClifInvSubgrp})\\
$\$\R^{q+1,p}$, $\para$ & The paravector subspace of $\clifn{p,q}$ (resp. $\clif$) (p. \pageref{not:Paravectors})\\
$\Gamma_{p,q}$, $\Gamma_{n+1}$ & The Lipschitz group of $\clifn{p,q}$ (resp. $\clif$) (p. \pageref{def:LipschitzGrp})\\
$\$\Gamma_{q+1,p}$, $\lip$ & The paravector Lipschitz group of $\clifn{p,q}$ (resp. $\clif$) (p. \pageref{def:ParaLipschitzGrp})\\
$\$\Gamma_{q+1,p}^{\triangleright}$, $\lip^{\triangleright}$ & The paravector Lipschitz monoid of $\clifn{p,q}$ (resp. $\clif$) (p. \pageref{not:LipschitzMonoid})\\
$Spin(p,q)$, $Spin(n+1)$ & The Spin cover of $SO(p,q)$ (resp. $SO(n+1)$) (p. \pageref{def:SpinGrp})\\
$\$pin(p,q)$, $\$pin(n+1)$ & The copies of the spin groups in $\$\Gamma_{q+1,p}$ (resp. $\lip$) (p. \pageref{cor:ParavSpinGrp})\\
$\$\mathfrak{g}_{n+1}$ & The Lie algebra of $\lip$ (p. \pageref{LipschitzLieAlgebraTheorem})\\
$S\lip$ & The set of Lipschitz spinors (p. \pageref{not:SLip})\\
$s_{\pV}$, $s(\pV,\kappa)$ & Sections of the tangent bundle $TS\lip$ (p. \pageref{not:TSLip})\\
$GL(2,\lip)^{\triangleright}$ & $GL(2,\lip)$ with the pseudo-determinant condition relaxed to allow any value in $\R$ (p. \pageref{not:GL2LipMonoid})\\
$N(\cdot)$, $|\cdot|^2$ & The paravector norm $N(\pV)=\pV\ol{\pV}$ (p. \pageref{defn:CliffordNorm})\\
$x,y,z,...$ & Generic elements of $\clifn{p,q}$ \\
$\pU,\pV,\pW,...$ & Generic paravectors \\
$U,V,W,...$ & Generic vectors \\
$i_j$ & Basis elements of $\clifn{p,q}$, $1 \leq j \leq p+q$ (p. \pageref{defn:CliffordAlgebra})\\
$\sigma(x)(y)$ & The action $xyx^*$ (p. \pageref{defn:TwistedAdjointSigma})\\
$\widetilde{Ad}_xy$ & The twisted adjoint action $x'yx^{-1}$ (p. \pageref{not:TwistAdj})\\
$Ad_xy$ & The adjoint action $xyx^{-1}$ (p. \pageref{not:Adj})\\
$\mathcal{B}^n$ & The bi-paravectors in $\clif$ (p. \pageref{BiParavectorDefinition})\\
$\kappa$ & A generic spinor (p. \pageref{def:LipschitzSpinor})\\
$(\xi,\eta)$ & Components of a generic Lipschitz spinor (p. \pageref{def:LipschitzSpinor})\\
$\check{D}$ & The complement of an element $D$ of $\clif^2$ (p. \pageref{def:Complement})\\
$\cdot^{\dagger}$ & The conjugate transpose (p. \pageref{not:ConjTrans})\\
$\cpara$ & One-point compactified $\para$; generalised Riemann sphere (p. \pageref{not:1ptCompact})\\
$\langle \cdot,\cdot \rangle$ & The generalised Hermitian form (p. \pageref{sec:HermitianForm})\\
$\{\cdot,\cdot\}$ & The bracket operation (p. \pageref{sec:TheBracket})\\
pdet & Pseudo-determinant (p. \pageref{not:pdet})\\
$P$, $P^{\kappa}$, $P^{\kappa}_{\pV}$ & The group of parabolic translation matrices (respectively; the subgroup of $P$ fixing $\kappa$, and the parabolic translation matrix fixing $\kappa$ and parameterised by $\pV$) (p. \pageref{def:ParabolicTranslation}, p. \pageref{ParabolicSubgroupDefinition})\\
\hline
 
 \textbf{Topology} & \\
 $\R p$, $\R^+ p$ & Real (resp. positive) scalar multiples of a point $p$ in a vector space\\
 $\Tilde{f}$ & The lift of a function $f$\\
 $\partial M$ & The boundary (possibly ideal) of a manifold or complex $M$\\
 $\widehat{M}$ & $M \cup \partial M$ for a manifold or complex $M$\\ 
 $A \times B$ & The topological product of spaces $A$ and $B$ \\
 $A/B$ & The quotient space of topological spaces $A$ and $B$ \\
$S^n$ & The $n$-sphere\\
 Dim$(X)$ & The topological dimension of $X$\\
 $\Pi$, $\Pi_{\gamma}$ & A hyperplane (resp. geodesic hyperplane)\\
\hline
 
 \textbf{Differential Geometry} & \\
 $TM,\ T_pM$ & The tangent bundle of $M$ (resp. tangent space to $M$ at $p$)  \\
 $(D_pf)(v)$ & The derivative of the function $f$ at $p$ in the direction $v$ \\
 $Df$ & The Jacobian matrix of $f$\\
 $\partial_j$ & A vector field in `the $j$ direction'\\
\hline
 
 \textbf{Algebra} & \\
 $\text{Re}(x)$ & The `real component' of $x \in A$, for $A$ a real algebra with canonical embedding $\R \subset A$ \\
 $R^{\times}$ & The invertible subgroup of the ring $R$ \\
 $A.X$ & The group, ring, algebra, etc. action of $A$ on $X$ \\
 $End(X)$ & Endomorphisms of the space $X$ \\
 $\oplus$ & The direct sum\\
 $\text{mod }n$ & The modulus function\\
 $x^{-1}$ & Multiplicative inverse of $x$\\
 $A/B$ & The quotient group, ring, etc. of $A$ over $B$\\
\hline
 
 \textbf{Linear Algebra} & \\
 $\star$ & The Hodge star operator (p. \pageref{not:HodgeStar})\\
 $a \wedge b$ & The exterior product of $a$ and $b$\\
 $\bigwedge \R^{p,q}$, $\bigwedge^n \R^{p,q}$ & The exterior algebra on $\R^{p,q}$ (resp. the subspace of $n$-vectors)\\
 $O(n)$, $SO(n)$ & The orthogonal (resp. special orthogonal) $n \times n$ matrix groups\\
 $GL(2,X)$ & The $2\times 2$ general linear group over a group or ring $X$ \\
 $SL(2,X)$, $SU(2,X)$ & The $2\times 2$ special linear (resp. special unitary) group over $X$ \\
 $PV$ & The projectivisation of a vector space $V$ (p. \pageref{not:projectivisation})\\
 $\mathcal{M}(n,X)$ & The algebra of $n \times n$ matrices over a group or ring $X$ \\
 tr, det & The trace and determinant of a matrix, respectively\\
$v^{\perp}$ & The orthogonal subspace to a vector $v$ in some vector space $V$ (p. \pageref{SphereTangentSpaceLemma})\\
$\otimes$ & The tensor product\\
\hline

\textbf{Minkowski Geometry} & \\
$\R^{1,n+2}$ & Generalised Minkowski space (p. \pageref{not:GenMinkSpace})\\
$L$,$L^+$ & The light-cone (resp. positive light-cone) of $\R^{1,n+2}$ (p. \pageref{not:LightCone})\\
$\ell$ & A light-ray in $L^+$ (p. \pageref{not:ell})\\
$(\ell,\psi)$ & A decorated ideal point (p. \pageref{not:decIdealPt})\\
$\mathscr{S}^+$, $\mathscr{S}^+_T$ & The celestial sphere (resp. celestial sphere at constant $T$) (p. \pageref{not:CelestialSphere})\\
$\mathscr{S}^{+D}$ & The space of decorated ideal points (p. \pageref{not:SpaceDecIdeal})\\
$(T,Z; X_0,...,X_n)_{\R}$, $(T,Z;\pX)_{\R}$ & Coordinates for $\R^{1,n+2}$ (p. \pageref{OrderOfVariables})\\
$\mathcal{P}^{1,n+2}$ & Paravector Hermitian matrices (p. \pageref{defn:ParavectorHermitianMatrices})\\
$(\cdot|\cdot)_{1,n+2}$ & The Minkowski metric (p. \pageref{not:MinkMetric}) \\
$(\cdot|\cdot)_{n+3}$ & The Euclidean metric on $\R^{n+3}$ \\
$SO(1,n+2)^+$ & The time-orientation-preserving component of $SO(1,n+2)$\\
$S_p$ & The paravector Hermitian matrix representation of $p \in \R^{1,n+2}$ (p. \pageref{not:ParaMatrixPt})\\
$\phi_1$ & The basepoint map from $S\lip$ to $L^+ \subset \R^{1,n+2}$ (p. \pageref{defn:BasepointMap})\\
$\Phi_1$ & The double covering map from $S\lip$ to $\mathcal{MF}^n$ (p. \pageref{defn:Phi1})\\
$\phi_2$ & The map from $L^+$ to horospheres in $\hyp^{n+2}$ (p. \pageref{not:phi2})\\
$\phi_2^{\partial}$ & The map from $p \in L^+$ to the centre of $\phi_2(p)$ (p. \pageref{not:phi2partial})\\
$\Phi_2$ & The diffeomorphism $\mathcal{MF}^n  \to D\text{Hor}^{n+2}$ (p. \pageref{Phi2})\\
$\phi$, $\Phi$ & $\phi_2 \circ \phi_1$ (resp. $\Phi_2 \circ \Phi_1$) (p. \pageref{not:phi}, p. \pageref{not:Phi})\\
$\phi^{\partial}$ & $\phi_2^{\partial} \circ \phi_1$ (p. \pageref{phipartial})\\
$[[p,v]]$ & The null flag with flagpole $\R p$ and flagplane spanned by $\{p,v\}$ (p. \pageref{not:FlagSingle})\\
$[[p;v_1,v_2,...v_n]]$ & The multiflag with subflags $[[p,v_1]],[[p,v_2]]...,[[p,v_n]]$ (p. \pageref{def:Multiflag})\\
$\mathcal{MF}^n_p$, $\mathcal{MF}^n$ & The space of multiflags based at $p$ (resp. bundle of multiflags) (p. \pageref{def:MultiflagSpace})\\
$\Psi$ & The correspondence for multiflags and decorated ideal points (p. \pageref{prop:MFlagsAreDecIdealPoints})\\
\hline

 \textbf{Hyperbolic Geometry} & \\
 $\hyp^n$ & Hyperboloid model (ch. 3), hyperbolic $n$-space (ch. 4-5) (p. \pageref{defn:HypCoords}, p. \pageref{not:Hypn})\\
 $\U^n$ & The upper half-space model of hyperbolic $n$-space (p. \pageref{not:UpperHalfSpace}, p. \pageref{not:UpperHalf2})\\
 $\B^n$ & The ball model of hyperbolic $n$-space (p. \pageref{not:BallModel})\\
 $(T,Z; X_0,...,X_n)_{\hyp}$, $(T,Z;\pX)_{\hyp}$ & Standard coordinates on the hyperboloid (p. \pageref{defn:HypCoords})\\
 $(w;y_0,...,y_n)_{\B}$, $(w;\textbf{y})_{\B}$ & Standard coordinates on the ball model (p. \pageref{not:BallCoords})\\
 $(z;x_0,...,x_n)_{\U}$, $(z;\textbf{x})_{\U}$ & Standard coordinates in the upper half-space (p. \pageref{not:UpperHalfCoords})\\
 $(...)_{\partial M}$ & Coordinates for the boundary of $M$ (p. \pageref{not:BdryCoords})\\
 $(...)_{TM}$ & Coordinates for the tangent space of $M$ (p. \pageref{not:TanCoords})\\
 $\lambda_{ij}$ & The lambda length between horospheres indexed $i$ and $j$ (p. \pageref{defn:LambdaLength})\\
$\gamma(\pX \to \pY)$ & The oriented geodesic between $\pX, \pY \in \partial \U^{n+2}$ (p. \pageref{def:GeodesicsInHyp})\\
$\gamma(w_1;\textbf{x} \to w_2;\textbf{y})$ & The oriented geodesic between $(w_1;\textbf{x})_{\partial\B}, (w_2;\textbf{y})_{\partial\B}$ (p. \pageref{def:GeodesicsInHyp})\\
Hor$^{n}$ & The space of horospheres in $\hyp^n$ (p. \pageref{not:HorN})\\
$L$Hor$^{n}$ & The space of horospheres in $\hyp^n$ bearing $n$ linearly independent oriented line fields (p. \pageref{def:LHor})\\
$D$Hor$^{n}$ & The space of decorated horospheres in $\hyp^n$ (p. \pageref{not:DHor})\\
$S$Hor$^{n}$ & The space of spin decorated horospheres in $\hyp^n$ (p. \pageref{not:SHor})\\
$\pi_1$, $\pi_1^{\partial}$ & The map (resp. boundary map) from $\hyp^{n+2} \to \B^{n+2}$ (p. \pageref{sec:MapsBetweenModels})\\
$\pi_2$, $\pi_2^{\partial}$ & The map (resp. boundary map) from $\B^{n+2} \to \U^{n+2}$ (p. \pageref{def:pi2}, p. \pageref{not:pi2partial})\\
$\pi$, $\pi^{\partial}$ & $\pi_2 \circ \pi_1$ (resp. $\pi_2^{\partial} \circ \pi_1^{\partial}$) (p. \pageref{not:pi}, p. \pageref{not:pipartial})\\
$\mathfrak{S}$ & The generalised Cayley transform $\widehat{\U}^{n+2} \to  \widehat{\B}^{n+2}$ (p. \pageref{defn:GeneralisedCayleyTransform})\\
$N^{in}$, $N^{out}$ & The in/outwards normal vector field on a horosphere (p. \pageref{Not:NinNout})\\
%$Fr(M)$ & The orthonormal frame bundle over a manifold $M$ (p. \pageref{not:FrM})\\
Fr$^n$, $S$Fr$^n$ & The frame (resp. spin frame) bundle over $\hyp^n$ (p. \pageref{not:Fr}, p. \pageref{not:SFr})\\
$(W^{in},W^{out})$ & A spin decoration or frame (p. \pageref{SpinFrameDefinition})\\
$(H,W)$ & A spin decorated horosphere (p. \pageref{not:SpinDecHoro})\\
$1^{in}$, $1^{out}$ & The Hodge star component of a horosphere frame field (p. \pageref{not:HodgeOne})\\
\hline
 
 \textbf{Quasideterminants} & \\
 $|A|_{jk}$ & The $(j,k)$-quasideterminant of the matrix $A$ (p. \pageref{def:quasideterminant})\\
 $A^{(j)}$ & The matrix formed by deleting the $j$th column of $A$ (p. \pageref{not:DeleteColj})\\
 $q_{jk}^{(s)}(A)$, $q_{jk}^I(A)$ & The left quasi-Pl\"{u}cker coordinates for a matrix $A$ (p. \pageref{not:LeftQuasiPlucker})\\
\hline
\end{tabularx}
\normalsize

\printbibliography
\end{document}